\documentclass[dvips,aos,preprint]{imsart}

\RequirePackage[OT1]{fontenc}
\usepackage{amssymb}
\usepackage{amsmath,amsthm}

\RequirePackage[numbers]{natbib}
\RequirePackage[colorlinks,citecolor=blue,urlcolor=blue]{hyperref}
\RequirePackage{hypernat}
\usepackage{color}
\usepackage{colordvi}
\usepackage{mathrsfs}

\arxiv{math.PR/0000000}

\startlocaldefs
\numberwithin{equation}{section}
\theoremstyle{plain}
\newtheorem{thm}{Theorem}[section]
\newtheorem{Cor}{Corollary}[section]
\newtheorem{Lemma}{Lemma}[section]
\newtheorem{proposition}{Proposition}[section]
\newtheorem{rem}{Remark}[section]

\newcommand{\1}{1\!{\rm l}}

\renewcommand{\t}{\tau}

\newcommand{\eps}{\varepsilon}

\newcommand{\RR}{\mathbb{R}}

\newcommand{\NN}{\mathbb{N}}

\renewcommand{\a}{\alpha}
\renewcommand{\phi}{\varphi}
\renewcommand{\epsilon}{\varepsilon}

\newcommand{\sign}{\text{sign}}

\endlocaldefs

\begin{document}

\begin{frontmatter}
\title{Asymptotic behaviour of the empirical Bayes posteriors associated to maximum marginal likelihood estimator  }
\runtitle{empirical Bayes MMLE}

\begin{aug}
\author{\fnms{Judith} \snm{Rousseau}\thanksref{t2,m1,m2}\ead[label=e1]{rousseau@ceremade.dauphine.fr}},
\author{\fnms{Botond} \snm{Szabo}\thanksref{t1,m3,m4}\ead[label=e2]{b.t.szabo@math.leidenuniv.nl}}

\thankstext{t2}{The project was partially supported by the ANR IPANEMA, the labex ECODEC}
\thankstext{t1}{The project was partially supported by the labex ECODEC, the European
  Research Council under ERC Grant Agreement 320637, Netherlands Organization for Scientific Research}

\runauthor{Rousseau and Szabo}

\affiliation{University Paris Dauphine\thanksmark{m1}, and CREST-ENSAE\thanksmark{m2}, and Budapest University of Technology\thanksmark{m3}, and Leiden University\thanksmark{m4}}

\address{CEREMADE, University Paris Dauphine\\
Place du Mar\'echal deLattre de Tassigny \\
\printead{e1}
}
\address{Leiden University,\\
Mathematical Institute,\\
Niels Bohrweg 1,
Leiden, 2333 CA,\\
The Netherlands\\
\printead{e2}
}
\end{aug}

\begin{abstract}
We consider the asymptotic behaviour of the marginal maximum likelihood empirical Bayes posterior distribution in general setting. First we characterize the set where the maximum marginal likelihood estimator is located with high probability. Then we provide oracle type of upper and lower bounds for the contraction rates of the empirical Bayes posterior. We also show that the hierarchical Bayes posterior achieves the same contraction rate as the maximum marginal likelihood empirical Bayes posterior. We demonstrate the applicability of our general results for various models and prior distributions by deriving upper and lower bounds for the contraction rates of the corresponding empirical and hierarchical Bayes posterior distributions.
\end{abstract}



\end{frontmatter}

\section{Introduction} \label{Intro}

In the Bayesian approach, the whole inference is based on the posterior distribution, which is proportional to the likelihood times the prior (in case of dominated models).  The task of designing a prior distribution $\Pi$ on the parameter $\theta \in \Theta$ is difficult and in large dimensional models cannot be performed in a fully subjective way. It is therefore common practice to consider a family of prior distributions  $\Pi(\cdot  | \lambda)$ indexed by a hyper-parameter $\lambda \in \Lambda$ and to either put a hyper-prior on $\lambda $  (hierarchical approach) or to choose $\lambda $ depending on the data, so that $\lambda = \hat\lambda(\mathbf x_n)$ where $\mathbf x_n$ denotes the collection of observations. The latter is refered to as an empirical Bayes (hereafter EB) approach, see for instance \cite{lehmann:casella:1998}. There are many ways to select the hyper-parameter $\lambda$ based on the data, in particular depending on the nature of the hyper-parameter.  

Recently \cite{petrone:rousseau:scricciolo:14} have studied the asymptotic behaviour of the posterior distribution for general empirical Bayes approaches; they provide conditions to obtain consistency of the EB posterior and in the case of parametric models characterized the behaviour of the maximum marginal likelihood estimator $\hat \lambda_n\equiv\hat \lambda (\mathbf x_n)$ (hereafter MMLE), together with the corresponding posterior distribution $\Pi(\cdot|\hat\lambda_n;\mathbf x_n)$ on $\theta$. They show that  asymptotically the MMLE converges to some oracle value $\lambda_0 $ which maximizes, in $\lambda$, the prior density calculated at the true value $\theta_0$ of the parameter, $\pi(\theta_0|\lambda_0)   = \sup \{\pi(\theta_0|\lambda), \lambda \in \Lambda\}$, where the density is with respect to Lebesgue measure. This cannot be directly extended to the nonparametric setup, since in this case, typically the prior distributions $\Pi(\cdot |\lambda)$, $\lambda \in \Lambda$ are not absolutely continuous with respect to a fixed measure. In the nonparametric setup the asymptotic behaviour of the MMLE and its associated EB  posterior distribution has been studied in the (inverse) white noise model under various families of Gaussian prior processes by \cite{ BelitserEnikeeva,knapikSVZ2012,SzVZ,florens:simoni,szabo:vdv:vzanten:13}, in the nonparametric regression problem with smoothing spline priors \cite{Paulo} and rescaled Brownian motion prior \cite{sniekers:15}, and in a sparse setting by \cite{JS04}. In all these papers, the results have been obtained via explicit expression of the marginal likelihood. Interesting phenomena have been observed in these specific cases. In \cite{SzVZ} an infinite dimensional Gaussian prior was considered with fixed regularity parameter $\a$ and a scaling hyper-parameter $\tau$. Then it was shown that the scaling parameter can compensate for possible mismatch of the base regularity $\a$ of the prior distribution and the regularity $\beta$ of the true parameter of interest up to a certain limit. However, too smooth truth can only be recovered sub-optimally by MMLE empirical Bayes method with rescaled Gaussian priors. In contrast to this in \cite{knapikSVZ2012} it was shown that by substituting the MMLE of the regularity hyper-parameter into the posterior, then one can get optimal contraction rate (up to a $\log n$ factor) for every Sobolev regularity class, simultaneously.

In this paper we are interested in generalizing the specific results of \cite{knapikSVZ2012} (in the direct case), \cite{SzVZ} to more general models, shading light on what is driving the asymptotic behaviour of the MMLE in nonparametric or large dimensional models. We also provide sufficient conditions to derive posterior concentration rates for EB procedures based on the MMLE. Finally we investigate the relationship between the MMLE empirical Bayes and hierarchical Bayes approaches. We show that the hierarchical Bayes posterior distribution (under mild conditions on the hyper-prior distribution) achieves the same contraction rate as the MMLE empirical Bayes posterior distribution. Note that our results do not answer the question whether empirical Bayes and hierarchical Bayes posterior distributions are strongly merging, which is certainly of interest, but would require typically a much more precise analysis of the posterior distributions.

More precisely, set  $\mathbf x_n$ the vector of observations and assume that conditionally on some parameter $\theta \in \Theta$, $\mathbf x_n$ is distributed according to $P_\theta^n$ with density $p_\theta^n$ with respect to some given measure $\mu$.  Let $\Pi(\cdot |\lambda), \, \lambda \in \Lambda $ be a  family of prior distributions on $\Theta$. Then the associated posterior distributions are equal to
$$ \Pi( B | \mathbf x_n; \lambda )  = \frac{ \int_B p_\theta^n(\mathbf x_n) d\Pi(\theta|  \lambda) }{ \bar m(\mathbf x_n |\lambda) }, \quad \bar m(\mathbf x_n |\lambda) = \int_\Theta p_\theta^n(\mathbf x_n) d\Pi(\theta|  \lambda)$$
for all $\lambda \in \Lambda$ and any borelian subset $B $ of $\Theta$. The MMLE is defined as 
\begin{equation} \label{MMLE}
\hat \lambda_n \in \mbox{argmax}_{\lambda \in \Lambda_n} \bar m(\mathbf x_n | \lambda) 
\end{equation}
for some $\Lambda_n\subseteq\Lambda$, and the associated EB posterior distribution by $\Pi( \cdot | \mathbf x_n , \hat \lambda_n)$. We note that in case there are multiple maximizers one can take an arbitrary one. Furthermore from practical consideration (both computational and technical) we allow the maximizer to be taken on the subset $\Lambda_n\subseteq\Lambda$.

Our aim is two fold, first to characterize the asymptotic behaviour of $\hat \lambda_n$ and second to derive posterior concentration rates in such models, i.e. to determine sequences $\epsilon_n$ going to 0 such that 
\begin{equation}\label{posterior-rate}
\Pi\left(\left.\theta\,:\,  d( \theta, \theta_0 ) \leq \epsilon_n \right| \mathbf x_n ,\hat \lambda_n\right) \rightarrow 1
\end{equation}
in probability under $P_{\theta_0}^n$, with $\theta_0 \in \Theta$ and $d(. , . ) $ some appropriate positive loss function on $\Theta$ (typically a metric or semi-metric, see condition (A2) later for more precise description).
There is now a substantial  literature on posterior concentration rates in large or infinite dimensional models initiated by the seminal paper of \cite{ghosal:ghosh:vdv:00}. Most results, however, deal with fully Bayesian posterior distributions, i.e. associated to priors that are not data dependent. The literature on EB posterior concentration rates deals mainly with specific models and specific priors. 

Recently, in  \cite{DRRS:arxiv},  sufficient conditions are provided for deriving  general EB posterior concentration rates when it is known that $\hat \lambda_n $ belongs to a well chosen subset $\Lambda_0$ of $\Lambda$. In essence, their result boils down to controlling 
 $\sup_{\lambda \in \Lambda_0} \Pi\left(\left.  d( \theta, \theta_0 ) > \epsilon_n \right| \mathbf x_n , \lambda\right)$. Hence either $\lambda $ has very little influence on the posterior concentration rate and it is not so important to characterize precisely $\Lambda_0$ or $\lambda$ is influential and it becomes crucial to determine properly $\Lambda_0$. In \cite{DRRS:arxiv}, the authors focus on the former. In this paper we are mainly concerned with the latter, with $\hat \lambda_n $ the MMLE. Since the MMLE is an implicit estimator (as opposed to the moment estimates considered in \cite{DRRS:arxiv}) the main difficulty here is to understand what the set $\Lambda_0$ is. 

We show in this paper that $\Lambda_0$ can be characterized  roughly as 
 $$ \Lambda_0 = \{ \lambda:\, \epsilon_{n}(\lambda) \leq M_n \epsilon_{n,0}\}$$
 for any sequence $M_n$ going to infinity and with 
  $ \epsilon_{n,0} = \inf \{ \epsilon_{n}(\lambda) ; \, \lambda \in \Lambda_n\}$ and $\epsilon_n(\lambda)$  satisfying 
   \begin{equation}\label{rate}
    \Pi\left( \left.\|\theta - \theta_0\| \leq K\epsilon_{n}(\lambda) \right| \lambda\right) = e^{-n \epsilon_{n}^2(\lambda)},
    \end{equation}
  with $(\Theta , \| \cdot \|)$ a Banach space and for some large enough constant $K$ (in the notation we omitted the dependence of $\eps_n(\lambda)$ on $K$ and $\theta_0$).
We then prove that the concentration rate of the MMLE empirical Bayes posterior distribution is of order  $O(M_n \epsilon_{n,0})$. We also show that the preceding rates are sharp, i.e. the posterior contraction rate is bounded from below by $\delta_n\eps_{n,0}$ (for arbitrary $\delta_n=o(1)$). Hence our results reveal the exact posterior contraction rates for every individual $\theta_0\in\Theta$.  Furthermore, we also show that the hierarchical Bayes method behaves similarly, i.e. the hierarchical posterior has the same upper ($M_n\eps_{n,0}$) and lower ($\delta_n\eps_{n,0}$) bounds on the contraction rate for every $\theta_0\in\Theta$ as the MMLE empirical Bayes posterior.

Our aim is not so much to advocate the use of the MMLE empirical Bayes approach, but rather to understand its behaviour. Interestingly, our results show that it is driven by the choice of the prior family $\{\Pi(\cdot | \lambda), \lambda \in \Lambda)\}$ in the neighbourhood of the true parameter $\theta_0$. This allows to determine a priori which family of prior distributions will lead to well behaved MMLE empirical Bayes posteriors and which won't. 
In certain cases, however, the computation of the MMLE is very challenging. Therefore it would be interesting to investigate other type of estimators for the hyper-parameters like the cross validation estimator. At the moment there is only a limited number of papers on this topic and only for specific models and priors, see for instance \cite{svz3,sniekers:15}.

These results are summarized in Theorem \ref{thm:main}, in Corollary \ref{thm: contraction}, and in Theorem \ref{thm: hierarchical}, in Section \ref{mainresult}. Then three different types of priors on $\Theta = \ell_2 = \{(\theta_j)_{j \in \mathbb N};\, \sum_j \theta_j^2 < +\infty \}$ are studied, for which upper bounds on $\epsilon_n(\lambda)$ are given in Section \ref{sec:seq}. We apply these results to three different sampling models: the Gaussian white noise, the regression and the estimation of the density based on iid data models in Sections \ref{sec:regression} and \ref{sec:density}.  Proofs are postponed to Section \ref{sec:proof}, to the appendix for those concerned with the determination of $\epsilon_n(\lambda)$ and to the Supplementary material \cite{rousseau:szabo:15:supp}

\subsection{Notations and setup} \label{sec:notation}
We assume that the observations $\mathbf x_n \in \mathcal X_n$ (where $\mathcal X_n$ denotes the sample space) are distributed according to a distribution $P_\theta^n$ (they are not necessarily i.i.d.), with $\theta \in \Theta$, where  $(\Theta, \| \cdot \|)$ is a Banach space. We denote by $\mu$  a dominating measure and by $p_\theta^n$ and $E_{\theta}^n$ the corresponding  density and expected value of $P_\theta^n$, respectively. We consider the family of prior distributions $\{\Pi(\cdot | \lambda), \, \lambda \in \Lambda\}$ on $\Theta$ with 
$\Lambda \subset \RR^d$ for some $d\geq 1$ and we denote by $ \Pi( \cdot | \mathbf x_n; \lambda )$ the associated posterior distributions. 

Throughout the paper $K( \theta_0, \theta)$ denotes the Kullback-Leibler divergence between $P_{\theta_0}^n $ and $P_\theta^n$ for all $\theta, \theta_0\in \Theta$ while $V_2(\theta_0,\theta)$ denotes the centered second moment of the log-likelihood:
\begin{equation*}
\begin{split}
K( \theta_0, \theta) &= \int_{\mathcal X_n} p_{\theta_0}^n (\mathbf x_n) \log \left( \frac{ p_{\theta_0}^n}{ p_{\theta}^n}(\mathbf x_n) \right)d\mu(\mathbf x_n) , \\
V_2(\theta_0, \theta)  &= E_{\theta_0}^n \left( \left| \ell_n(\theta_0)-\ell_n(\theta) - K(\theta_0, \theta)\right|^2 \right) 
\end{split}
\end{equation*} 
with 
$\ell_n(\theta) = \log p_{\theta}^n (\mathbf x_n ) $.
As in \cite{ghosal:vdv:07}, we define the Kullback-Leibler neighbourhoods of $\theta_0$ as 
\begin{equation*}
B( \theta_0, \epsilon , 2) = \{ \theta; K(\theta_0, \theta) \leq n\epsilon^2, \, V_2(\theta_0, \theta) \leq n\epsilon^2\}
\end{equation*}
and note that in the above definition $V_2(\theta_0, \theta) \leq n\epsilon^2$ can be replaced by $V_2(\theta_0, \theta) \leq Cn\epsilon^2$ for any positive constant $C$ without changing the results.

For any subset $A \subset \Theta$ and $\epsilon>0$, we denote $\log N( \epsilon, A, d(\cdot,\cdot))$ the $\epsilon $ - entropy of $A$ with respect to the (pseudo) metric $d(\cdot,\cdot)$, i.e. the logarithm of the covering number of $A$ by $d(\cdot,\cdot) $ balls of radius $\epsilon$. 

We also write 
 $$m(\mathbf x_n|\lambda) = \frac{ \bar m(\mathbf x_n|\lambda) }{ p_{\theta_0}^n ( \mathbf x_n) } =  \frac{ \int_\Theta p_\theta^n ( \mathbf x_n ) d\Pi(\theta| \lambda)  }{ p_{\theta_0}^n ( \mathbf x_n) }. $$

For any bounded function $f$, $\|f\|_\infty = \sup_{x}|f(x)|$ and if $\phi$ denotes a countable collection of functions $(\phi_i, i\in \NN) $, then $\|\phi\|_\infty = \max_i \|\phi_i\|_\infty$. If the function is integrable then $\|f\|_1$ denotes its $L_1$ norm while $\|f\|_2$ its $L_2$ norm and if $\theta \in \ell_r = \{ \theta = (\theta_i)_{i\in \NN} , \sum_i |\theta_i|^r<+\infty\}$, with $r\geq 1$, $\|\theta\|_r = (\sum_i |\theta_i|^r)^{1/r}$. 

Throughout the paper $x_n \lesssim y_n$ means that there exists a constant $C$ such that for $n$ large enough $x_n \leq Cy_n$, similarly with $x_n \gtrsim y_n$ and $x_n \asymp y_n$ is equivalent to $y_n \lesssim x_n \lesssim y_n$. For equivalent (abbreviated) notation we use the symbol $\equiv$.

\section{ Asymptotic behaviour of the MMLE, its associated posterior distribution and the hierarchical Bayes method} \label{mainresult}

Although the problem can be formulated as a classical parametric maximum likelihood estimation problem, since $\lambda $ is finite dimensional, its study is more involved than the usual regular models due to the complicated nature of the  marginal likelihood. Indeed $m(\mathbf x_n|\lambda)$ is an integral over an infinite (or large) dimensional space.

 For $\theta_0 \in \Theta$ denoting the true parameter, define the sequence  $\epsilon_{n}(\lambda)\equiv\eps_{n}(\lambda, \theta_0,K)$ as
\begin{equation}\label{def: eps}
\Pi(\,\theta\,:\, \| \theta - \theta_0\| \leq K\epsilon_{n}(\lambda) | \lambda ) = e^{-n\epsilon_{n}(\lambda)^2 }, 
\end{equation}
for some positive parameter $K>0$. 
If the cumulative distribution function of $\|\theta- \theta_0\|$ under $\Pi( \cdot |\lambda)$ is not continuous, then the definition of $\epsilon_n(\lambda)$ can be replaced by
\begin{equation}\label{def:eps:2}
\tilde{c}_0^{-1}n\epsilon_{n}(\lambda)^2\leq -\log \Pi(\,\theta\,:\, \| \theta - \theta_0\| \leq K\epsilon_{n}(\lambda) | \lambda ) \leq \tilde{c}_0 n\epsilon_{n}(\lambda)^2,
\end{equation}
for some $\tilde{c}_0\geq 1$ under the assumption that such a sequence $\eps_n(\lambda)$ exists.

Roughly speaking, under the assumptions stated below, $\log m(\mathbf x_n|\lambda) \asymp n \epsilon_n^2(\lambda) $ and  {$\epsilon_{n}(\lambda)$ is the posterior concentration rate associated to the prior} $\Pi(\cdot |\lambda) $ and the best possible  (oracle) posterior concentration rate over $\lambda \in \Lambda_{n}$ is denoted 
\begin{align*}
\eps_{n,0}^2=\inf_{\lambda\in\Lambda_{n}} \{\epsilon_{n}(\lambda)^2:\, \eps_{n}(\lambda)^2\geq m_n (\log n)/n\}\vee m_n (\log n)/n, 
\end{align*}
with any sequence $m_n$ tending to infinity.

With the help of the oracle value $\eps_{n,0}$ we define a set of hyper-parameters with similar properties, as:
\begin{align}
\Lambda_0\equiv\Lambda_0(M_n)\equiv\Lambda_{0,n}(K,\theta_0,M_n)=\{\lambda\in\Lambda_n:\, \epsilon_{n}(\lambda)\leq M_n \eps_{n,0}\},\label{def: Lambda0}
\end{align}
with any sequence $M_n$ going to infinity. We show that under general (and natural) assumptions the marginal maximum likelihood estimator $\hat{\lambda}_n$ belongs to the set $\Lambda_0$ with probability tending to one, for some constant $K>0$ large enough.  The parameter $K$ provides extra flexibility to the approach and simplifies the proofs of the upcoming conditions in certain examples.
In practice at least in the examples we have studied) the constant $K$ essentially modifies $\epsilon_{n}(\lambda)$ by a multiplicative constant and thus does not modify the final posterior concentration rate, nor the set $\Lambda_0$ since $M_n$ is any sequence going to infinity.  Note that our results are only meaningful in cases where $\epsilon_n(\lambda)$ defined by \eqref{def:eps:2} vary with $\lambda$.

 We now give general conditions under which the MMLE is inside of the set $\Lambda_0$ with probability going to 1 under $P_{\theta_0}^n$. Using \cite{DRRS:arxiv}, we will then deduce  that the concentration rate of the  associated MMLE empirical Bayes posterior distribution is bounded by $M_n \epsilon_{n,0}$. 

Following \cite{petrone:rousseau:scricciolo:14} and \cite{DRRS:arxiv} we construct
for all $\lambda,\lambda'\in\Lambda_n$ a transformation $\psi_{\lambda,\lambda'}:\,\Theta\mapsto\Theta$ such that if $\theta\sim\Pi(\cdot|\lambda)$ then $\psi_{\lambda,\lambda'}(\theta)\sim\Pi(\cdot|\lambda')$ and for a given sequence $u_n\rightarrow0$ we introduce the notation
\begin{align}
 q_{\lambda,n}^{\theta}(\mathbf x_n)= \sup_{\rho(\lambda,\lambda')\leq u_n} p^n_{\psi_{\lambda,\lambda'}(\theta)}(\mathbf x_n)\label{def: QLambda}
\end{align}
where $\rho:\,\Lambda_n\times\Lambda_n\rightarrow \mathbb{R}^+$ is some loss function
and $Q_{\lambda,n}^{\theta}$ the associated measure. Denote by $N_n(\Lambda_0),N_n(\Lambda_n\setminus\Lambda_0)$, and $N_n(\Lambda_n)$ the covering number of $\Lambda_0,\Lambda_n\setminus\Lambda_0$ and $\Lambda_n$ by balls of radius $u_n$, respectively, with respect of the loss function $\rho$. 

We consider the following set of assumptions to bound $ \sup_{\lambda \in \Lambda_n\setminus\Lambda_0} m(\mathbf x_n|\lambda) $ from above. 
\begin{itemize}
\item (A1) There exists $N>0$  such that for all $\lambda \in  \Lambda_n\setminus\Lambda_0$ and $n\geq N$, there exists  $\Theta_n(\lambda)\subset \Theta $ 
\begin{equation}\label{cond: Q}
  \sup_{\{\|\theta - \theta_0\|\leq K\epsilon_n(\lambda)\}\cap \Theta_n(\lambda)  }\frac{\log Q_{\lambda,n}^{\theta}(\mathcal X_n)}{n\epsilon_{n}(\lambda)^2}  =o(1), 
\end{equation}
and such that
 \begin{equation}\label{cond: complement}
  \int_{\Theta_n(\lambda)^c }  Q_{\lambda,n}^{\theta}(\mathcal X_n) d\Pi(\theta|\lambda) \leq e^{- w_n^2 n \epsilon_{n,0}^2},
   \end{equation}
for some positive sequence $w_n$ going to infinity. 
\item (A2) [tests]  There exists  $0< \zeta,c_1 < 1$ such that for all $\lambda \in \Lambda_n\setminus\Lambda_0$ and all $\theta \in \Theta_n(\lambda)$, there exist tests $\phi_n(\theta)$ such that 
\begin{equation}\label{cond: test}
 E_{\theta_0}^n\phi_{n}(\theta) \leq e^{ -c_1 n d^2(\theta, \theta_0) } , \quad \sup_{d(\theta, \theta') \leq \zeta d(\theta, \theta_0)} Q_{\lambda,n}^{\theta'}(1 - \phi_{n}(\theta) ) \leq e^{-c_1 n d^2(\theta, \theta_0) }, 
\end{equation}
where $d(\cdot,\cdot)$ is a semi-metric satisfying 
\begin{equation} \label{change:distance} 
\Theta_n(\lambda) \cap \{ \|\theta -\theta_0\| >K\epsilon_n(\lambda) \} \subset \Theta_n(\lambda) \cap \{ d( \theta, \theta_0) > c(\lambda) \epsilon_n(\lambda) \}
\end{equation}
for some $c(\lambda) \geq w_n \epsilon_{n,0} /\epsilon_{n}(\lambda)$ and
\begin{equation} \label{cond: entropy} 
\log N( \zeta u , \{ u \leq d(\theta, \theta_0 )\leq 2 u \}\cap \Theta_n(\lambda) , d(\cdot , \cdot) ) \leq c_1 n u^2/2 
\end{equation}
 for all $u \geq c(\lambda)\epsilon_{n}(\lambda)$.
\end{itemize}

\begin{rem}
We note that we can weaken $\eqref{cond: Q}$ to
\begin{equation*}
  \sup_{\{\|\theta - \theta_0\|\leq \epsilon_n(\lambda)\}\cap \Theta_n(\lambda)  } Q_{\lambda,n}^{\theta}(\mathcal X_n) \leq e^{cn\epsilon_n^2(\lambda)}, 
\end{equation*}
for some positive constant $c<1$ in case the cumulative distribution of $\|\cdot - \theta_0\|$ under $\Pi(\cdot|\lambda)$ is continuous and hence the definition $\eqref{def: eps}$ is meaningful.
\end{rem}

Conditions \eqref{cond: Q} and \eqref{cond: complement} imply that we can control the small perturbations  of the likelihood $p_{\psi_{\lambda, \lambda'}(\theta)}^n(\mathbf x_n) $ due to the change of measures $\psi_{\lambda,\lambda'}$ and are similar to those used in \cite{DRRS:arxiv}. They allow us  to control $m(\mathbf x_n|\lambda)$ uniformly over $\Lambda_n\setminus\Lambda_0$. They are rather weak conditions since $u_n$ can be chosen very small. In \cite{DRRS:arxiv}, the authors show that they hold even with complex priors such as nonparametric mixture models. Assumption (A2) \eqref{cond: test}, together with \eqref{cond: entropy} have been verified in many contexts, with the difference that here the tests need to be performed with respect to the perturbed likelihoods $q_{\lambda,n}^{\theta}$. Since the $u_n$ - mesh of $\Lambda_n\setminus\Lambda_0$ can be very fine, these perturbations can be well controlled over the sets $\Theta_n(\lambda)$, see for instance \cite{DRRS:arxiv} in the context of density estimation or intensity estimation of Aalen point processes. The interest of the above conditions is that they are very similar to standard conditions considered in the posterior concentration rates literature, starting with \cite{ghosal:ghosh:vdv:00} and \cite{ghosal:vdv:07}, so that there is a large literature on such types of conditions which can be applied in the present setting. Therefore, the usual variations on these conditions can be considered. For instance an alternative condition to (A2) is:

(A2 bis)  There exists  $0< \zeta < 1$ such that for all $\lambda \in \Lambda_n\setminus\Lambda_0$ and all $\theta \in \Theta_n(\lambda)$, there exist tests $\phi_n(\theta)$ such that \eqref{cond: test} is verified and for all $j \geq K$, writing 
$$B_{n,j}(\lambda) = \Theta_n(\lambda) \cap \{ j  \epsilon_n(\lambda) \leq \|\theta -\theta_0\|  < (j+1)\epsilon_n(\lambda) \},$$
then
 \begin{equation*}
B_{n,j}(\lambda)  \subset \Theta_n(\lambda) \cap \{ d( \theta, \theta_0) > c(\lambda,j) \epsilon_n(\lambda) \}
\end{equation*}
with $$\sum_{j\geq K} \exp\left( -\frac{ c_1}{2} n c(\lambda,j)^2  \epsilon_n(\lambda)^2 \right)  \lesssim e^{ - n w_n^2 \epsilon_{n,0}^2 } $$
 and 
\begin{equation*}
\log N( \zeta c(\lambda,j) \epsilon_n(\lambda), B_{n,j}(\lambda) , d(\cdot,\cdot)) \leq \frac{c_1  n c(\lambda,j)^2 \epsilon_{n}(\lambda)^2}{2}. 
\end{equation*}

\vspace{0.3 cm}
 Here the difficulty lies in the comparison between the metric $\| \cdot \|$ of the Banach space and the testing distance $d(\cdot,\cdot)$, in condition \eqref{change:distance}. Outside the white noise model, where the Kullback and other moments of the likelihood ratio are directly linked to the $L_2$ norm on $\theta- \theta_0$, such comparison may be non trivial.  In \citet{vvvz08}, the prior had some natural Banach structure and norm, which was possibly different to the Kullback-Leibler and the testing distance $d(\cdot,\cdot)$, but comparable in some sense. Our approach is similar in spirit.  We illustrate this here  in the special cases of regression function and density estimation under different families of priors, see Sections \ref{sec:regression} and \ref{dens:loglinear}.
In Section \ref{sec:hist} we use a prior which is not so much driven by a Banach structure and the norm $\| \cdot \|$ is replaced by the Hellinger distance. Hence in full generality $\| \cdot \|$ could be replaced by any metric, for instance the testing metric $d(\cdot,\cdot)$, as long as the rates $\epsilon_n(\lambda)$ can be computed.

The following assumption is used to bound from below $\sup_{\lambda \in \Lambda_0} m(\mathbf x_n|\lambda)$ 
\begin{itemize}
\item (B1) There exist $\tilde \Lambda_0\subset \Lambda_0$ and  $M_2\geq 1$ such that for every $\lambda \in \tilde \Lambda_0$ 
 $$ \{ \|\theta - \theta_0\| \leq K \epsilon_{n}(\lambda)\} \subset B( \theta_0, M_{2}\epsilon_{n}(\lambda) , 2) ,$$
 and such that there exists $\lambda_0 \in \tilde \Lambda_0$ for which $\epsilon_n(\lambda_0) \leq M_1 \epsilon_{n,0}$ for some positive $M_1$.
\end{itemize}

\begin{rem}\label{rem:B1}

A variation of (B1) can be considered where $ \{ \|\theta - \theta_0\| \leq K \epsilon_{n}(\lambda)\}$ is replaced by 
$ \{ \|\theta - \theta_0\| \leq K \epsilon_{n}(\lambda)\} \cap \tilde \Theta_n(\lambda)$ where $\tilde \Theta_n(\lambda) \subset \Theta$ verifies 
$$ \Pi\left(\left. \{ \|\theta - \theta_0\| \leq K \epsilon_{n}(\lambda)\} \cap \tilde \Theta_n(\lambda) \right| \lambda\right) \gtrsim e^{-K_2 n \epsilon_{n}^2(\lambda) },$$
for some $K_2\geq 1$. This is used in Section \ref{sec:density}.
\end{rem}

\subsection{ Asymptotic behaviour of the MMLE and empirical Bayes posterior concentration rate} \label{sec:mmle}
We now present the two main results of this Section, namely : asymptotic behaviour of the MMLE and concentration rate of the resulting empirical Bayes posterior. We first describe the asymptotic behaviour of $\hat \lambda_n$.
\begin{thm}\label{thm:main}
Assume that there exists $K>0$ such that conditions (A1),(A2), and (B1) hold with $w_n =o(M_n)$, then if $\log N_n (\Lambda_n\setminus\Lambda_0)= o(n w_n^2 \epsilon_{n,0}^2 )$,
\begin{align*}
\lim_{n\rightarrow\infty}P_{\theta_0}^n\left( \hat{\lambda}_n\in \Lambda_0\right)= 1.
\end{align*}
\end{thm}
The proof of Theorem \ref{thm:main} is given in Section \ref{pr:thmain}.

Note that in the definition of $\Lambda_0(M_n)$, $M_n$ can be any sequence going to infinity. In the examples we have considered in Section \ref{sec:seq}, $M_n$ can be chosen to increase to infinity arbitrarily slowly. If $\epsilon_n(\lambda)$ is (rate) constant \eqref{thm:main} presents no interest since $\Lambda_0= \Lambda_n$, but if for some $\lambda \neq \lambda'$ the fraction $\epsilon_n( \lambda) /\epsilon_n(\lambda')$ either goes to infinity or to 0, then choosing $M_n $ increasing slowly enough to infinity, Theorem \ref{thm:main} implies that the MMLE converges to a meaningful subset of $\Lambda_n$.  In particular our results are too crude to be informative in the parametric case.  Indeed from \cite{petrone:rousseau:scricciolo:14}, in the parametric non degenerative case $\epsilon_n(\lambda) \asymp \sqrt{(\log n)/n}$ in definition \eqref{def:eps:2} for all $\lambda$ and $\Lambda_0 = \Lambda$. In the parametric degenerative case, where the $\lambda_0$ belongs to the boundary of the set $\Lambda$ then one would have at the limit  $\pi( \cdot | \lambda_0 ) = \delta_{\theta_0}$ corresponding to $\epsilon_n(\lambda_0)= 0 $. So we do recover the oracle parametric value of  \cite{petrone:rousseau:scricciolo:14}.
However for the condition $\log N_n (\Lambda_n\setminus\Lambda_0)= o(n w_n^2 \epsilon_{n,0}^2 )$ to be valid one would require $w_n^2 n\epsilon_{n,0}^2 \asymp \log n$, corresponding essentially to $\Lambda_0 $ being the whole set.

Using the above theorem, together with  \cite{DRRS:arxiv}, we obtain the associated posterior concentration rate, controlling uniformly 
$\Pi( d( \theta_0, \theta)\leq \epsilon_n | \mathbf x_n, \lambda ) $ over $\lambda \in \Lambda_0$, with $\epsilon_n = M_n \epsilon_{n,0}$.
To do so we consider the following additional assumptions:
\begin{itemize}
\item (C1)
For every $c_2>0$ there exists constant $N>0$ such that for all $\lambda \in  \Lambda_0$ and $n\geq N$,  there exists $\Theta_n(\lambda) $ satisfying
 \begin{equation}\label{cond:comp:bis}
  \sup_{\lambda \in \Lambda_0} \int_{\Theta_n(\lambda)^c }  Q_{\lambda,n}^{\theta}(\mathcal X_n) d\Pi(\theta|\lambda) \leq e^{-c_2 n \epsilon_{n,0}^2}
   \end{equation}
  \item (C2)  
  There exists  $0< c_1,\zeta < 1$ such that for all $\lambda \in \Lambda_0$ and all $\theta \in \Theta_n(\lambda)$, there exist tests $\phi_n(\theta)$ satisfying \eqref{cond: test} and  \eqref{cond: entropy}, where \eqref{cond: entropy} is supposed to hold for any $u\geq M M_n \epsilon_{n,0}$ for some $M>0$. 
  \item (C3) There exists $C_0>0 $ such that for all $\lambda \in \Lambda_0$, for all $\theta \in \{d(\theta_0, \theta)\leq M_n \epsilon_{n,0}\}\cap \Theta_n(\lambda)$, 
  $$\sup_{\rho(\lambda,\lambda')\leq u_n} d( \theta, \psi_{\lambda, \lambda'}(\theta) )\leq C_0M_n \epsilon_{n,0}.$$
\end{itemize}

\begin{Cor}\label{thm: contraction}
Assume that $\hat \lambda_n \in \Lambda_0$ with probability going to 1 under $P_{\theta_0}^n$ and that assumptions (C1)-(C3) and (B1) are satisfied, then if $\log N_n(\Lambda_0) \leq O(n\epsilon_{n,0}^2)$, there exists $M>0$ such that
\begin{equation} \label{post:contr}
E_{\theta_0}^n\Pi\Big( \left. \theta:\,  d(\theta,\theta_0) \geq M M_n\eps_{n,0} \right| \mathbf x_n; \hat{\lambda}_n\Big)=o(1).
\end{equation}
\end{Cor}

A consequence of Corollary \ref{thm: contraction} is in terms of frequentist risks of Bayesian estimators. Following \cite{belitser:ghosal:03} one can construct an estimator based on the posterior which  converges  at the posterior concentration rate: $E_{\theta_0}\left[d(\hat \theta, \theta_0) \right] = O( M_n\eps_{n,0})$. Similar results can also be derived for the posterior mean  in case $d(\cdot,\cdot) $ is convex and bounded, and \eqref{post:contr} is of order $O( M_n\eps_{n,0})$, see for instance \cite{ghosal:ghosh:vdv:00}. 

Corollary \ref{thm: contraction} is proved in a similar way to Theorem 1 of \cite{DRRS:arxiv}, apart from the lower bound on the marginal likelihood since here we use the nature of the MMLE which simplifies the computations. The details are presented in Section \ref{pr:thm-contraction}. We can refine the condition on tests (C3) by considering slices as in \cite{DRRS:arxiv}.

Next we provide a lower bound on the contraction rate of the MMLE empirical Bayes posterior distribution. For this we have to introduce some further assumptions. First of all we extend assumption $\eqref{cond: Q}$ to the set $\Lambda_0$.  Let $e : \Theta \times \Theta \rightarrow \mathbb R^+$ be a pseudo-metric and assume that for all $\lambda\in\Lambda_0$  and some $\delta_n$ tending to zero we have
\begin{equation} \label{cond: Q2}
\begin{split}
& \sup_{\{\|\theta - \theta_0\|\leq \epsilon_n(\lambda)\}\cap \Theta_n(\lambda)  } \frac{\log Q_{\lambda,n}^{\theta}(\mathcal X_n)}{n\epsilon_n^2(\lambda)}  =o(1)\\
& \sup_{\lambda \in \Lambda_0} \frac{n\epsilon_{n,0}^2}{-\log \Pi(\,\theta\,:\, e(\theta,\theta_0)\leq 2\delta_n \epsilon_{n,0}|\lambda ) }  = o(1) 
\end{split}
\end{equation}
and consider the modified version of (C3): (C3bis) 
 There exists $C_0>0 $ such that for all $\lambda \in \Lambda_0$, for all $\theta \in \{e(\theta_0, \theta)\leq \delta_n \epsilon_{n,0}\}\cap \Theta_n(\lambda)$, 
  $$\sup_{\rho(\lambda,\lambda')\leq u_n} d( \theta, \psi_{\lambda, \lambda'}(\theta) )\leq C_0\delta_n \epsilon_{n,0}.$$ 



\begin{thm}\label{thm: LB}
Assume that conditions (A1)-(C2)  and (C3bis) together with assumption $\eqref{cond: Q2}$  
 hold.
In case $\log N_n(\Lambda_0)=o(n\eps_{n,0}^2)$  and $\eps_{n,0}^{2}>m_n(\log n)/n$ we get that
\begin{align*}
E_{\theta_0}^n \Pi(\,\theta\,:\, e(\theta,\theta_0)\leq \delta_n\eps_{n,0}|\hat{\lambda}_n,\mathbf x_n )=o(1).
\end{align*}
\end{thm}

Typically $e(.,.)$ will be either $d(\cdot ,\cdot )$ or $\| \cdot \|$. 
The lower bound is proved using the same argument as the one used to bound 
$E_{\theta_0}^n\Big( \Pi(\Theta_n^c |\hat{\lambda}_n,\mathbf x_n ) \Big)$, see Section \ref{pr:thmain} and \ref{pr:thm-contraction}, where $\{d(\theta,\theta_0) \leq\delta_n\eps_{n,0}\}$ plays the same role as $\Theta_n^c$.  We postpone the details of the proof to Section B.7 of the supplementary material \cite{rousseau:szabo:15:supp}.


Theorem \ref{thm:main} describes the asymptotic behaviour of the MMLE $\hat \lambda_n$, via the oracle set $\Lambda_0$, in other words it minimizes $\epsilon_{n}(\lambda)$. The use of the Banach norm is particularly adapted to the case of priors on parameters $\theta = (\theta_i)_{i \in \NN} \in \ell_2$, where the $\theta_i's  $ are assumed independent. This type of priors is studied in  Section \ref{sec:seq}.

\subsection{Contraction rate of the hierarchical Bayes posterior}
In this section we investigate the relation between the MMLE empirical Bayes method and the hierarchical Bayes method. We show that under the preceding assumptions complemented with not too restrictive conditions on the hyper-prior distribution the hierarchical posterior distribution achieves the same convergence rate as the MMLE empirical Bayes posterior. Let us denote by $\tilde\pi(\cdot)$ the density function of the hyper-prior, then the hierarchical prior takes the form
$$\Pi(\cdot)=\int_{\Lambda} \Pi(\cdot|\lambda)\tilde\pi(\lambda)d\lambda.$$
Note that we integrate here over the whole hyper-parameter space $\Lambda$, not over the subset $\Lambda_n\subseteq\Lambda$ used in the MMLE empirical Bayes approach.

Intuitively to have the same contraction rate one would need that the set of probable hyper-parameter values $\Lambda_0$ accumulates enough hyper-prior mass. Let us introduce a sequence $\tilde{w}_n$ satisfying $\tilde{w}_n=o(M_n\wedge w_n)$ and denote by $\Lambda_0(\tilde{w}_n)$ the set defined in $\eqref{def: Lambda0}$ with $\tilde{w}_n$.
\begin{itemize} 
\item (H1) Assume that $\tilde \Lambda_0 \subset \Lambda_0(\tilde{w}_n)$ and for some sufficiently large $\bar{c}_0>0$ there exists $N>0$ such that for all $n\geq N$ the hyper-prior satisfies
$$\int_{\tilde \Lambda_0 }\tilde\pi(\lambda)d\lambda\gtrsim e^{- n\eps_{n,0}^2}.$$
and 
$$\int_{\Lambda_n^c}\tilde\pi(\lambda)d\lambda\leq e^{-\bar{c}_0 n\eps_{n,0}^2}.$$
\item (H2) Uniformly over $\lambda \in \tilde \Lambda_0$ and $\{\theta:\,\|\theta-\theta_0\|\leq K\eps_n(\lambda)\}$ there exists $c_3>0$ such that
$$P_{\theta_0}^n
\Big\{\inf_{\lambda':\,\rho(\lambda,\lambda')\leq u_n}\ell_n\big(\psi_{\lambda,\lambda'}(\theta)\big)-\ell_n(\theta_0)\leq -c_3 n\eps_n(\lambda)^2 \Big\}=O\big(e^{-n\eps_{n,0}^2}\big). $$
\end{itemize}
We can then show that the preceding condition is sufficient for giving upper and lower bounds for the contraction rate of the hierarchical posterior distribution.

\begin{thm}\label{thm: hierarchical}
Assume that the conditions of Theorem \ref{thm:main} and Corollary \ref{thm: contraction} hold alongside with conditions (H1) and (H2). Then the hierarchical posterior achieves the oracle contraction rate (up to a slowly varying term) 
$$E_{\theta_0}^n\Pi(\theta:\,d(\theta,\theta_0)\geq MM_n \eps_{n,0}|\mathbf x_n)=o(1).$$
Furthermore if condition $\eqref{cond: Q2}$ also holds we have that
$$E_{\theta_0}^n\Pi(\theta:\,d(\theta,\theta_0)\leq \delta_n \eps_{n,0}|\mathbf x_n)=o(1). $$
\end{thm}

The proof of the theorem is given in Section \ref{sec: HBproof}.

\section{Application to sequence parameters and histograms} \label{sec:seqhist}

\subsection{Sequence parameters}\label{sec:seq}
In this section we apply Theorem \ref{thm:main} and Corollary \ref{thm: contraction} to the case of priors on $(\Theta,\|\cdot\|) = (\ell_2,\|\cdot\|_2)$. We endow the sequence parameter $\theta=(\theta_1,\theta_2,...)$ with independent product priors of the following three types: 
\begin{enumerate}
\item [\textbf{ (T1)}] Sieve prior : The hyper-parameter of interest is $\lambda= k$ the truncation: For $2 \leq k$, 
$$ \theta_j \stackrel{ind}{\sim } g( \cdot ) , \quad  \mbox{if } j \leq k,\quad\text{and} \quad \theta_j = 0 \quad \mbox{ if } j >k.$$
We assume that $\int e^{s_0|x|^{p^*}}g(x) dx = a<+\infty $ for some $s_0>0$  and $p^*\geq 1$. 

\item [\textbf{ (T2)}] Scale parameter  of a Gaussian process prior: let $\tau_j  = \tau j^{- \alpha - 1/2} $ and $\lambda = \tau$ with 
$$ \theta_j \stackrel{ind}{\sim }  \mathcal N( \cdot ,  \tau_j^2), \quad  1\leq j \leq n,\quad\text{and} \quad \theta_j = 0 \quad \mbox{ if } j >n.$$
\item [\textbf{ (T3)}] Rate parameter : same prior as above but this time $\lambda = \alpha $. 
\end{enumerate}
\begin{rem}
Alternatively one could consider the priors (T2) and (T3) without truncation at level $n$. The theoretical behaviour of the truncated and non-truncated versions of the priors are very similar, however from a practical point of view the truncated priors are arguably more natural. 

\end{rem}

In the hierarchical setup with a prior on $k$, Type (T1) prior has been studied by \cite{arbel, shen:ghoshal:15} for generic models, by \cite{rivoirard:rousseau:12} for density estimation, by \cite{babenko:belitser:10} for Gaussian white noise model and by \cite{ray:2013} for inverse problems. Type (T2) and (T3) priors have been studied with fixed hyper-parameters by \cite{cox:1993,zhao:00,vvvz08,castillo08,knapik} or using a prior on $\lambda= \tau$ and $\lambda=\a$ in \cite{belitser:ghosal:03,Lian,SzVZ,knapikSVZ2012}. In the white noise model, using the explicit expressions of the marginal likelihoods and the posterior distributions,  \cite{knapikSVZ2012, SzVZ} have derived posterior concentration rates and described quite precisely the behaviours of the MMLE using type (T3) and (T2) priors, respectively. 

In the following, $\Pi( \cdot | k ) $ denotes a prior in the form (T1), while $\Pi( \cdot | \tau, \alpha)$ denotes either (T2) or (T3). 

\subsection{Deriving $\epsilon_{n}(\lambda)$ for priors (T1) - (T3)} \label{sec:eps:T123}

It appears from Theorem \ref{thm:main} that  a key quantity to describe the behaviour of the MMLE is $\epsilon_{n}(\lambda)$ defined by \eqref{def: eps}. In the following Lemmas we describe $\eps_n(\lambda)\equiv\epsilon_{n}(\lambda, K) $ for any $K>0$ under the three types of priors above and for true parameters $\theta_0$ belonging to either hyper-rectangles 
$$ \mathcal H_\infty ( \beta,L) = \{ \theta_0 = (\theta_{0,i})_i : \, \max_i i^{2\beta+1} \theta_{0,i}^2 \leq L\} $$ or Sobolev balls
 $$ \mathcal S_\beta ( L) = \{\theta_0 = (\theta_{0,i})_i : \, \sum_{i=1}^\infty i^{2\beta} \theta_{0,i}^2 \leq L\}.$$

\begin{Lemma} \label{lem:eps:T1}
Consider priors of type (T1),  with $g$ positive and continuous on $\RR$  and let $\theta_0 \in \ell_2$, then for all $K >0$ fixed and if $k \in \{ 2, \cdots, \epsilon n/ \log n\}$, with $\epsilon >0$ a small enough constant 
$$  \epsilon_n(k)^2 \asymp    \sum_{i=k+1}^\infty \theta_{0,i}^2  + \frac{ k\log n}{ n }.$$
Moreover  if $\theta_0 \in \mathcal H_\infty ( \beta,L)  \cup \mathcal S_\beta ( L)$ with $\beta >0$ and $L$ any positive constant,
\begin{equation} \label{eps0:k}
\epsilon_{n,0} \lesssim (n/\log n)^{-\beta/(2\beta+1)},
\end{equation}
and there exists $\theta_0 \in \mathcal H_\infty ( \beta,L)  \cup \mathcal S_\beta ( L)$ for which \eqref{eps0:k} is also a lower bound. 
\end{Lemma}

The proof of Lemma \ref{lem:eps:T1} is postponed to Appendix \ref{app:T1}.
We note that it is enough in the above Lemma to assume that $g$ is positive and continuous over the set $\{ |x| \leq M\}$ with $M> 2 \|\theta_0\|_\infty$. 
\begin{rem}
One can get rid of the $\log n$ factor in the rate by allowing the density $y$ to depend on $n$, see for instance \cite{babenko:belitser:10}, \cite{Gao:Zhou:13}. These results can be recovered (and adapted to the MMLE empirical Bayes case) by a slight modification of the proof of Lemma \ref{lem:eps:T1}.
\end{rem}

Priors of type (T2) and (T3) are Gaussian process priors, thus following \cite{vvvz08}, let us introduce the so called concentration function
 \begin{align}
\phi_{\theta_0}(\epsilon; \alpha , \tau )=\inf_{h\in\mathbb{H}^{\alpha , \tau}:\, \|h-\theta_0\|_2\leq\epsilon} \|h\|_{\mathbb{H}^{\alpha,\tau}}^2-\log \Pi(\|\theta\|_2\leq \epsilon | \alpha, \tau ),\label{def: ConcFunc}
\end{align}
where $\mathbb{H}^{\alpha , \tau}$ denotes the Reproducing Kernel Hilbert Space (RKHS) associated to the Gaussian prior $\Pi(\cdot|\a,\t)$ 
$$\mathbb{H}^{\alpha , \tau} = \{ \theta = (\theta_i)_{i\in \NN}; \,  \sum_{i=1}^n i^{2\alpha+1} \theta_i^2 < +\infty,\quad \theta_i=0\,\,\text{for $i>n$}\}=\mathbb{R}^n,$$ with for all $\theta \in \mathbb{H}^{\alpha , \tau}$
$$\|\theta\|_{\mathbb{H}^{\alpha , \tau}}^2 = \tau^{-2} \sum_{i=1}^{n} i^{2\alpha+1} \theta_i^2.$$
Then from Lemma 5.3 of \cite{vaart:zanten:2008a}
\begin{equation}\label{Equiv: ConcFunc and SmallBall}
\phi_{\theta_0}(K\epsilon;\alpha, \tau) \leq  - \log \Pi ( \|\theta -\theta_0\|_2 \leq  K\epsilon | \alpha, \tau ) \leq \phi_{\theta_0}(K\epsilon/2;\alpha, \tau)
\end{equation}
We also have that
\begin{equation} \label{centeredGP}
\tilde{c}_1^{-1} \left( K\epsilon/\tau \right)^{-1/\alpha}\leq-\log \Pi(\|\theta\|_2\leq K \epsilon | \alpha , \tau) \leq \tilde{c}_1 \left( K\epsilon/\tau \right)^{-1/\alpha},
\end{equation}
for some $\tilde{c}_1\geq 1$, see for instance Theorem 4 of \cite{KuelbsLi2}. This leads to the following two lemmas. 


\begin{Lemma} \label{lem:eps:T2T3}
In the case of Type (T2) and (T3) priors, with $\theta_0 \in  \mathcal S_\beta ( L)\cup \mathcal H_\infty ( \beta,L)$:

$\bullet $ If $\beta \neq  \alpha +1/2$ 
\begin{equation} \label{eps:lambda:beta+-}
\frac{\|\theta_0\|_2}{\sqrt{n\tau^2}}\1_{n \tau^2 > 1} + n^{-\frac{\alpha}{2 \alpha+1}} \tau^{\frac{1}{2\alpha+1}} \lesssim  \epsilon_n(\lambda) \lesssim  n^{-\frac{\alpha}{2 \alpha+1}} \tau^{\frac{1}{2\alpha+1}} + \left( \frac{ a(\alpha, \beta) }{ n\tau^2 } \right)^{ \frac{\beta }{2\alpha+1} \wedge  \frac{1}{ 2 } },
\end{equation}
where $a(\alpha, \beta) =  L^{\frac{\alpha+1/2}{\beta}}/|2 \alpha- 2 \beta +1|$ if  $\theta_0 \in \mathcal H_\infty ( \beta,L) $ while $a(\alpha,\beta)=L^{\frac{\alpha+1/2}{\beta}}$ if  $\theta_0 \in  \mathcal S_\beta ( L)$. 
The constants depend possibly on $K$ but neither on $n, \tau$ or $\alpha$. 

$\bullet $ If $\beta=\alpha +1/2$ then 
\begin{equation} \label{eps:lambda:beta=}
\frac{\|\theta_0\|_2}{\sqrt{n\tau^2}}\1_{n \tau^2 > 1} +n^{-\frac{\alpha}{2 \alpha+1}} \tau^{\frac{1}{2\alpha+1}} \lesssim  \epsilon_n(\lambda) \lesssim n^{-\frac{\alpha}{2 \alpha+1}} \tau^{\frac{1}{2\alpha+1}} + \left(\frac{ \log( n \tau^2)  }{ n\tau^2 }\right)^{ \frac{1}{ 2 } }\1_{n \tau^2 > 1}, 
\end{equation}
where the term $\log (n\tau^2)$ can be eliminated in the case where $\theta_0 \in \mathcal S_\beta (L) $. 

\end{Lemma} 


\begin{Lemma}
\label{lem:eps:2:T2T3}
 In the case of prior type (T2)  (with $\lambda= \tau$):
\begin{itemize}
\item If $\alpha +1/2 < \beta$ then  for all $\theta_0 \in \mathcal H_\infty ( \beta,L)  \cup \mathcal S_\beta ( L)$
\begin{equation}\label{eps0:tau:1}\epsilon_{n,0} \lesssim n^{-(2\alpha+1)/(4\alpha+4)},\end{equation}
and for all $\theta_0 \in\ell_2(L)$ satisfying $\|\theta_0\|_2 \geq c$ for some fixed $c>0$,    \eqref{eps0:tau:1} is also a lower bound.
\item If $\alpha +1/2 > \beta $ then 
\begin{equation}\label{eps0:tau:2}
\epsilon_{n,0} \lesssim n^{-\beta/(2\beta+1)}.
\end{equation}
\item If $\alpha +1/2 = \beta $ then
\begin{equation}\label{eps0:tau:3}
\begin{split}
\epsilon_{n,0} &\lesssim n^{-\beta/(2\beta+1)} \log n^{1/(2\beta+1)}, \quad \mbox{if } \quad \theta_0 \in \mathcal H_\infty ( \beta,L),  \\
\epsilon_{n,0} &\lesssim n^{-\beta/(2\beta+1)} , \quad \mbox{if } \quad \theta_0 \in \mathcal S_\beta (L),  
\end{split}
\end{equation}
and there exists $\theta_0 \in \mathcal H_\infty ( \beta,L) $ for which the upper bound \eqref{eps0:tau:3} is also a lower bound.
\end{itemize}
In the case of prior type (T3) (with $\lambda= \alpha$),
\begin{equation}\label{eps0:alpha:HS}
\epsilon_{n,0} \lesssim n^{-\beta/(2\beta+1)}, \quad \mbox{if } \quad \theta_0 \in \mathcal S_\beta( L)\cup  \mathcal H_\infty ( \beta,L).
\end{equation}
\end{Lemma}

We note that for the scaling prior (T2) in the case $\a+1/2<\beta$ Lemma \ref{lem:eps:2:T2T3} provides us the sub-optimal rate $\eps_{n,0}\asymp n^{-(2\alpha+1)/(4\alpha+4)}$. Therefore under condition \eqref{cond: Q2} (verified in the supplementary material for prior (T2))  in all three types of examples studied in this paper (white noise, regression and estimation of density models), we get that for all $\theta_0 \neq 0$ with $\a+1/2<\beta$, the type (T2) prior leads to sub-optimal posterior concentration rates (and in case $\theta_0\in\mathcal{H}_{\infty}(\beta,L)$, $\beta=\a+1/2$ as well). 

An important tool to derive posterior concentration rates in the case of empirical Bayes procedures is the construction of the change of measure $\psi_{\lambda, \lambda'} $.  We present in the following section how these changes of measures can be constructed in the context of priors (T1)-(T3). 

\subsection{Change of measure} \label{sec:eps:T123}

In the case of prior (T1), there is no need to construct $\psi_{\lambda,\lambda'}$ due to the discrete nature of the hyper-parameter $\lambda = k$ the truncation threshold.

In the case of prior (T2) if $\tau, \tau' > 0$ then define  for all $i \in \mathbb N$ 
 \begin{equation}\label{change:T2}
 \psi_{\tau, \tau'}(\theta_i) = \frac{\tau'}{\tau}\theta_i 
 \end{equation}
so that $\psi_{\tau, \tau'}(\theta) = (\psi_{\tau, \tau'}(\theta_i), i\in \NN) = \theta \tau'/\tau$ and if $\theta \sim \Pi( \cdot | \tau, \alpha)$, then  $\psi_{\tau, \tau'}(\theta)  \sim  \Pi( \cdot | \tau', \alpha)$. 

Similarly, in the case of Type (T3) prior,
 \begin{equation}\label{change:T3}
 \psi_{\alpha, \alpha'}(\theta_i) =i^{\alpha - \alpha'} \theta_i 
 \end{equation}
so that $\psi_{\alpha, \alpha'}(\theta) = (\psi_{\alpha, \alpha'}(\theta_i), i\in \NN) $ and if $\theta \sim \Pi( \cdot | \tau, \alpha)$, then  $\psi_{\alpha, \alpha'}(\theta)  \sim  \Pi( \cdot | \tau, \alpha')$. 
Note in particular that if $\alpha'\geq \alpha$ and $\sum_i \theta_i^2 < +\infty$ hold then $\sum_i \psi_{\alpha, \alpha'}(\theta_i)^2 < \infty$. This will turn out to be usefull in the sequel. 

\subsection{Choice of the hyper-prior}

In this section we give sufficient conditions on the hyper-priors in the case of the prior distribution (T1)-(T3), such that condition (H1) is satisfied. The proofs are deferred to Section D of the supplementary material \cite{rousseau:szabo:15:supp}. 

\begin{Lemma}\label{Lem: hyperT1}
In case of prior (T1) we choose $\Lambda_n=\{2,3,...,c_0 n/\log n\}$ for some small enough constant $c_0>0$ and assume that $\theta_0\in S_{\beta}(L)\cup \mathcal{H}_{\infty}(\beta,L)$ for some $\beta\geq \beta_1>\beta_0\geq 0$. Then for any hyper-prior satisfying 
\begin{equation}\label{H1fork}
k^{-c_2 k}\lesssim\tilde\pi(k)\lesssim e^{-c_1 k^{1/(1+2\beta_0)}},
\end{equation}
for some $c_1,c_2>0$, assumption (H1) holds. In case the prior has support on $\Lambda_n$ the upper bound condition is not needed on $\tilde\pi$.
 \end{Lemma}
Note that the Hypergeometric and the Poisson distribution satisfies the above conditions.

\begin{Lemma}\label{Lem: hyperT2}
Consider the prior (T2) and take $\Lambda_n=[e^{-c_0\bar{c}_0\tilde{w}_n^2n\eps_{n,0}^2},e^{c_0\bar{c}_0\tilde{w}_n^2n\eps_{n,0}^2}]$ for some positive $c_0>0$ (and $\bar{c}_0$ given in condition (H1)). Then for any hyper-prior satisfying 
\begin{align*}
e^{-c_1\t^{\frac{2}{1+2\a}}}\lesssim \tilde\pi(\t)\lesssim \t^{-c_2}\quad\text{for $\t\geq 1$ with some $c_1>0$ and $c_2>1+1/c_0$},\\
e^{-c_3\t^{-2}}\lesssim \tilde\pi(\t)\lesssim \t^{c_4}\quad\text{for $\t\leq 1$ with some $c_2>0$ and $c_4>1/c_0-1$}
\end{align*}
assumption (H1) holds. Furthermore the upper bound condition can be removed if the prior has support on $\Lambda_n$.
 \end{Lemma}

Note that for instance the inverse gamma and Weibull distributions satisfy this assumption.

\begin{rem}\label{Rem: hyperT2}
To obtain the polynomial upper bound of the hyper-prior densities $\tilde\pi(\t)$ in Lemma \ref{Lem: hyperT2} the set $\Lambda_n$ is taken to be larger than it is necessary in the empirical Bayes method to achieve adaptive posterior contraction rates, see for instance Propositions \ref{prop:reg:T2} and \ref{prop:dens:T2}. Nevertheless the  conditions on the hyper-entropy are still satisfied, i.e. by taking $u_n=e^{-2c_0\bar{c}_0\tilde{w}_n^2n\eps_{n,0}^2}$ on $\Lambda\setminus\Lambda_0$ and $u_n=n^{-d}$ (for any $d>0$) on $\Lambda_0$ we get that $\log N_n(\Lambda_n)=o(w_n^2n\eps_{n,0}^2)$ and $\log N_n(\Lambda_0)=o(n\eps_{n,0}^2)$.
\end{rem}

\begin{Lemma}\label{Lem: hyperT3}
Consider the prior (T3), take $\Lambda_n=[0,c_0n^{c_1}]$ for some positive constants $c_0,c_1$ and assume that $\theta_0\in S_{\beta}(L)\cup \mathcal{H}_{\infty}(\beta,L)$ for some $\beta>\beta_0>0$. Then for any hyper-prior satisfying 
$$ e^{-c_2\a}\lesssim \tilde\pi(\a)\lesssim e^{-c_0\a^{1/c_1}},\quad \text{for $\a>0$} $$ 
and for some $c_0,c_1,c_2> 0$,  assumption (H1) holds. The upper bound on $\tilde \pi$ can be removed by taking the support of the prior to be $\Lambda_n$.
 \end{Lemma}

In the following sections, we prove that in the Gaussian white noise, regression and density estimation models the MMLE empirical Bayes posterior concentration rate is bounded from above by  $M_n \epsilon_{n,0}$ and from below by $\delta_n\eps_{n,0}$, where $\epsilon_{n,0}$ is given in Lemma \ref{lem:eps:2:T2T3} under priors (T1)-(T3) and $M_n$, respectively $\delta_n$, tends to infinity, respectively 0, arbitrary slowly.

\subsection{Application to the nonparametric regression model}\label{sec:regression}



In this section we show that our results apply to the nonparametric regression model. We consider the fixed design regression problem, where we assume that the observations $\mathbf x_n=(x_1,x_2,...,x_n)$ satisfy
\begin{align}
x_i=f_0(t_i)+Z_i, \quad i=1,2,...,n, \label{model: GRfixed}
\end{align}
where $Z_i\stackrel{iid}{\sim}N(0,\sigma^2)$ random variables (with known $\sigma^2$ for simplicity) and $t_i=i/n$.

Let us denote by $\theta_0=(\theta_{0,1},\theta_{0,2},..)$ the Fourier coefficients of the regression function $f_0\in L_2(M)$:
$f_0(t)=\sum_{j=1}^{\infty}\theta_{0,j} e_j(t),$
so that $(e_j(.))_j $ is the Fourier basis.
We note that following from Lemma 1.7 in \cite{tsybakov} and Parseval's inequality we have that
\begin{align*}
\|f_0\|_2=\|\theta_0\|_2=\|f_0\|_n,
\end{align*}
where $\|f_0\|_n$ denotes the $L_2$-metric associated to the empirical norm. 

First we deal with the random truncation prior (T1) where applying Theorem \ref{thm:main}, Corollary \ref{thm: contraction} and Theorem \ref{thm: hierarchical} combined with Lemma \ref{lem:eps:T1} we get that both the MMLE empirical Bayes and hierarchical Bayes posteriors are rate adaptive (up to a $\log n$ factor). The following proposition is proved in Section B.1 of  the supplementary material \cite{rousseau:szabo:15:supp}.
\begin{proposition} \label{prop:reg:T1prior}
 Assume that $f_0\in \mathcal{H}_{\infty}(\beta,L)\cup S_\beta(L)$ and consider a type (T1) prior. Let $\Lambda_n = \{2, \cdots, k_n\}$ with $k_n = \epsilon n/ \log n$ for some small enough constant $\eps>0$. 
 Then, for any $M_n$ tending to infinity and $K>0$ the MMLE estimator $\hat k_n \in \Lambda_0 =  \{ k: \epsilon_n(k) \leq M_n  \epsilon_{n,0}\} $ with probability going to 1 under $P_{\theta_0}^n$, where $\epsilon_n(k)$ and $\epsilon_{n,0}$ are given in Lemma \ref{lem:eps:T1}.
 
Furthermore we also have the following contraction rates:
for all $0< \beta_1 \leq \beta_2 < +\infty $, uniformly over $\beta \in (\beta_1, \beta_2)$
  \begin{align*}
\sup_{f_0\in \mathcal{H}_{\infty}(\beta,L)\cup S_\beta(L)} E^n_{f_0}  \Pi\left( \left.f:\, \|f_{0}-f\|_2 \geq M_n (n/\log n)^{-\frac{\beta}{2\beta+1}} \right| \mathbf x_n ; \hat k_n\right) = o(1),\\
\sup_{f_0\in \mathcal{H}_{\infty}(\beta,L)\cup S_\beta(L)} E^n_{f_0}  \Pi\left( \left.f:\, \|f_{0}-f\|_2 \geq M_n (n/\log n)^{-\frac{\beta}{2\beta+1}} \right| \mathbf x_n \right) = o(1),
	\end{align*}
	where the latter is satisfied if the hyper prior on $k$ satisfies  \eqref{H1fork}. 
	
Finally we note that the above bounds are sharp in the sense that  both the MMLE empirical and the hierarchical Bayes posterior contraction rates are bounded from below by $\delta_n (n/\log n)^{-\beta/(2\beta+1)}$ with $P_{\theta_0}^n$-probability tending to one, for any $\delta_n=o(1)$ and some $\theta_0 \in \mathcal H_\infty(\beta,L)\cup S_\beta(L)$. 
 \end{proposition}

Next we consider the priors (T2) and (T3). As a consequence of Theorem \ref{thm:main}, Corollary \ref{thm: contraction}, Theorem \ref{thm: hierarchical}, and Lemma \ref{lem:eps:2:T2T3} we can show that both the hierarchical Bayes and the MMLE empirical Bayes method for the rescaled Gaussian prior (T2) is optimal only in a limited range of regularity classes $\mathcal{S}_\beta(L)\cup\mathcal H_{\infty}(\beta,L)$ satisfying $\beta<\alpha+1/2$, else the posterior achieves a sub-optimal contraction rate $n^{-(2\alpha+1)/( 4 \alpha + 4)}$. However, by taking the MMLE of the regularity hyper-parameter $\a$ or endowing it with a hyper-prior distribution in the Gaussian prior (T3), the posterior achieves the minimax contraction rate $n^{-\beta/(1+2\beta)}$. Similar results were derived in \cite{SzVZ} and \cite{knapikSVZ2012} in the context of the (inverse) Gaussian white noise model using semi-explicit computations. We note that our implicit (and general) approach not just reproduces the previous findings in the direct (non inverse problem) case, but also improves on the posterior contraction rate in case of the prior (T3), where in \cite{knapikSVZ2012} an extra $\log n$ factor was present.

\begin{proposition} \label{prop:reg:T2}
 Assume that $f_0\in \mathcal S_\beta(L)\cup \mathcal{H}_{\infty}(\beta,L)$ for some $\beta>0$ and consider  type (T2) and (T3) priors with $\alpha >0$. Furthermore take $\Lambda_n(\tau)=[n^{-1/(4\a)},n^{\a/2}]$ and  $\Lambda_n(\a)=(0,c_0n^{c_1}]$, respectively, for some $c_0, c_1>0$. Then $\hat \lambda_n \in \Lambda_0$ with $P_{f_0}^n$-probability tending to 1. Furthermore, both in the case of the MMLE empirical Bayes and hierarchical Bayes approach we have for any $M_n $ going to infinity with hyper-priors satisfying (H1) (see for instance Lemma \ref{Lem: hyperT2} and Lemma \ref{Lem: hyperT3}) that
\begin{itemize}
\item For the multiplicative scaling prior (T2)
 \begin{itemize}
 \item If $\beta > \alpha + 1/2$, the posterior concentration rate is bounded from above by 
  $$ M_n \epsilon_{n,0} \asymp M_n n^{-(2\alpha+1)/( 4 \alpha + 4)},$$
	and for $\delta_n=o(1)$ and  $\|f_0\|_2\geq c$ (for some positive constant $c$) it is bounded from below by
	$$\delta_n\eps_{n,0}\asymp \delta_n n^{-(2\alpha+1)/( 4 \alpha + 4)}.$$
  \item If $\beta < \alpha +1/2 $, the posterior concentration rate is bounded by 
  $$M_n \epsilon_{n,0} \lesssim M_n n^{-\beta/(2\beta +1)},$$
  with an extra $\log n$ term if $\beta = \alpha +1/2$ and $f_0\in \mathcal{H}_{\infty}(\beta,L)$.
 \end{itemize}
\item For the regularity prior (T3) the posterior contraction rate is also
$$M_n \epsilon_{n,0} \lesssim M_n n^{-\beta/(2\beta +1)}.$$
\end{itemize}
\end{proposition}

Proposition \ref{prop:reg:T2} is proved in Section B.2 of  the supplementary material \cite{rousseau:szabo:15:supp}.

\begin{rem}
In fact our results are stronger than the minimax results presented in Propositions \ref{prop:reg:T1prior} and \ref{prop:reg:T2}. From Theorem \ref{thm:main} and Corollary \ref{thm: contraction} it follows that for both the MMLE empirical Bayes and the hierarchical Bayes methods the posterior contracts around the truth for every $\theta_0\in\Theta$ with rate $M_n\eps_{n,0}(\theta_0)$, which is more informative than a statement on the worst case scenario over some regularity class, i.e. the minimax result. 
\end{rem}

\begin{rem}
We note that in the case of the Gaussian white noise model the same posterior contraction rate results (both for the empirical Bayes and hierarchical Bayes approaches) hold for the priors (T1)-(T3) as in the nonparametric regression model. The proof of this statement can be easily derived as a special case of the results on the nonparametric regression, see the end of the proofs of Propositions \ref{prop:reg:T1prior} and \ref{prop:reg:T2}.
\end{rem}
 
\subsection{Application to density estimation}  \label{sec:density}
In this Section we consider the density estimation problem on $[0,1]$, i.e. the observations $\mathbf x_n= ( x_1, \cdots, x_n)$ are independent and identically distributed from a distribution with density $f$ with respect to Lebesgue measure. We consider two families of priors on  the set of densities $\mathcal F = \{ f : [0,1] \rightarrow \RR^+ ; \int_0^1 f(x)dx =1 \}$.  In the first case   we parameterize the densities as 
\begin{equation}\label{dens:model}
f (x) = f_\theta(x) = \exp\left( \sum_{j=1}^\infty \theta_j \phi_j(x) - c(\theta) \right), \,\, e^{c(\theta)} = \int_0^1 \exp\left( \sum_{j=1}^\infty\theta_j \phi_j(x) \right)dx
\end{equation}
where $(\phi_j)_{j\in \NN}$ forms an orthonormal basis with $\phi_0=1$ and $\theta = (\theta_j)_{j\in \NN} \in \ell_2$. Hence \eqref{dens:model} can be seen either as a log - linear model or as an infinite dimensional exponential family, see for instance \cite{verdinelli:wasserman:1998}, \cite{vvvz08},  \cite{rivoirard:rousseau:09},  \cite{rivoirard:rousseau:12} and \cite{arbel}. 

In the second we consider random histograms to parameterize $\mathcal F$. 

\subsubsection{Log-linear model} \label{dens:loglinear}
We study priors based on the parameterization \eqref{dens:model} and  we assume that the true density has the form
 $ f_0  = f_{\theta_0}$ for some $\theta_0 \in \ell_2$ and throughout the Section we will assume that $f_0 $ verifies $\|\log f_0 \|_\infty <+\infty $ and that $\theta_0 \in \mathcal S_\beta (L )$ for some $L>0$. We study the MMLE empirical Bayes and hierarchical Bayes methods based on priors of type (T1), (T2) and (T3) in this model. We consider the usual metric in the context of density estimation, namely the Hellinger metric $h(f_1, f_2)^2  = \int_0^1 (\sqrt{f_1}(x) - \sqrt{f_2}(x) )^2 dx $.
 
 First we  consider the type (T1) prior where $\lambda = k$. We show that Theorems \ref{thm:main}, \ref{thm: hierarchical}, and Corollary \ref{thm: contraction} can be applied so that the MMLE empirical Bayes and hierarchical posterior rates are  minimax adaptive over a collection of Sobolev classes.

 \begin{proposition}\label{prop:dens:T1}
 Assume that $\theta_0\in \mathcal S_\beta(L)$ with $\beta >1/2$, consider a type (T1) prior, and let $\Lambda_n = \{2, \cdots, k_n\}$ with $k_n = k_0 \sqrt{ n}/\log n^3$.
 Then, for any $M_n$ going to infinity and $K>0$,  if $\hat k_n$ is the MMLE over $\Lambda_n$, with probability going to 1 under $P_{\theta_0}^n$, 
 $\hat k_n \in \Lambda_0 =  \{ k; \epsilon_n(k) \leq M_n  \epsilon_{n,0}\} $, where $\epsilon_n(k)$ and $\epsilon_{n,0}$ are given in Lemma \ref{lem:eps:T1} and for all $1/2 < \beta_1 \leq \beta_2 < +\infty $
  $$ \sup_{ \beta \in (\beta_1, \beta_2)} \sup_{\theta_0\in \mathcal S_\beta(L)} E_{\theta_0}^n \left\{ \Pi\left(\left. h(f_{\theta_0}, f_\theta) \geq M_n (n/\log n)^{-\frac{\beta}{2\beta+1}} \right| \mathbf x_n ; \hat k_n\right)\right\} = o(1). $$
	Similarly in the hierarchical posterior distribution with hyper-prior satisfying the conditions of Lemma \ref{Lem: hyperT1}  also achieves the (nearly) minimax contraction rate
	  $$ \sup_{ \beta \in (\beta_1, \beta_2)} \sup_{\theta_0\in \mathcal S_\beta(L)} E_{\theta_0}^n \left\{ \Pi\left(\left. h(f_{\theta_0}, f_\theta) \geq M_n (n/\log n)^{-\frac{\beta}{2\beta+1}} \right| \mathbf x_n\right)\right\} = o(1). $$
	
    Moreover there exists $\theta_0 \in \mathcal S_\beta(L)$ for which $\delta_n (n/\log n)^{-\beta/(2\beta+1)}$ is a lower bound on the posterior concentration rate for both adaptive Bayesian methods.  
 \end{proposition}
 
 The proof of Proposition \ref{prop:dens:T1} is presented in Section B.3 of  the supplementary material \cite{rousseau:szabo:15:supp}.
 
We now apply Theorems \ref{thm:main}, \ref{thm: hierarchical}, and Corollary \ref{thm: contraction} to priors (T2) and (T3) and derive similar concentration rates as in the case of the regression model.
Let 
\begin{equation*}
  \bar \tau_n =  n^{\alpha/2-1/4}, \quad 
  \underline \tau_n = n^{-1/4+ 1/(8 \a)}.  
  \end{equation*}
 \begin{proposition} \label{prop:dens:T2}
 Assume that $\theta_0\in \mathcal S_\beta(L)$ with $\beta >1/2$  and consider a type (T2) prior with $\alpha >1/\sqrt{2}$  and $\Lambda_n = ( \underline \tau_n, \bar \tau_n)$. Then $\hat \lambda_n \in \Lambda_0$ with probability going to 1 under $P_{\theta_0}^n$ and the same conclusions as in Proposition \ref{prop:reg:T2}  hold.
\end{proposition}

The constraint $\alpha > 1/\sqrt{2}$ is to ensure that for all $\beta \leq \alpha +1/2 $, 
 $n^{-(\beta - \alpha)/(2\beta +1)}$ which corresponds to the minimizer of $\epsilon_n(\tau)$ (up to a multiplicative constant) belongs to the set $(\underline \tau_n, \bar \tau_n)$.

\begin{proposition} \label{prop:dens:T3}
 Assume that $\theta_0\in \mathcal S_\beta(L)$ with $\beta >1/2$  and consider a type (T3) prior with $\alpha >1/2$  and $\Lambda_n = [1/2+1/n^{1/4}, \bar \lambda_n]$, with $\bar \lambda_n = \log n/( 16 \log \log n)$. Then for any $M_n $ going to infinity the MMLE empirical Bayes posterior achieves the minimax contraction rate
  $$M_n \epsilon_{n,0} \lesssim M_n n^{-\beta/(2\beta +1)}.$$
	Furthermore the hierarchical posterior also achieves the minimax contraction rate for hyper-priors satisfying (H1).
\end{proposition}
The proofs of Propositions \ref{prop:dens:T2} and \ref{prop:dens:T3} are presented in Sections B.4 and B.5 of  the supplementary material \cite{rousseau:szabo:15:supp}.

We now consider the second family of priors.

\subsubsection{Random histograms} \label{sec:hist}

In this section we parameterize $\mathcal F$ using piecewise constant functions, as in  \cite{castillo:rousseau:main} for instance. In other words we define 
\begin{equation}\label{hist} 
f_\theta (x) = k\sum_{j=1}^k \theta_j \1_{I_j}, \quad I_j=((j-1)/k, j/k], \quad \sum_{j=1}^k \theta_j=1, \quad \theta_j \geq 0,
\end{equation}
and we consider a Dirichlet prior on $\theta = (\theta_1, \cdots, \theta_k)$ with parameter $(\alpha, \cdots, \alpha)$. The hyper-parameter on which maximization is performed is $\lambda = k$, as in the case of the truncation prior (T1). We define the sequence $\epsilon_n(k)$ in terms of the Hellinger distance, i.e. it satisfies \eqref{def: eps} with $h(f_0, f_\theta) $ replacing $\|\theta - \theta_0\|$. 

We then have the following result,
\begin{proposition}\label{prop:hist}
Assume that $f_0 $ is continuous and bounded from above and below by $C_0 $ and $c_0 $ respectively. 
If $\Lambda = \{1, \cdots, k_n\}$, with $k_n = O((n/\log n))$ and if $\alpha  \leq A$  for some constant $A$ independent on $k$, then 
for all $k \in \Lambda$
\begin{equation}\label{epsilonHist}
 b(k)^2 + \frac{ k \log(n/k) }{ n }  \lesssim \epsilon_n(k)^2 \lesssim  b(k)^2 + \frac{k \log n}{ n}, \end{equation}
with $$b(k)^2 =  \sum_{j=1}^k \int_{I_j} \left( \sqrt{f_0} - \tilde \eta_j k \right)^2 dx, \quad \tilde \eta_j = \int_{I_j}\sqrt{f_0}(x) dx.$$

Now suppose that $f_0 \in \mathcal H_{\infty}(\beta, L ) $, with $L>0$ and $\beta \in (0,1]$. 
The MMLE empirical Bayes posterior achieves the minimax contraction rate (up to a $\log n$ term), i.e. for all   $M_n \rightarrow +\infty$
$$M_n \epsilon_{n,0} \lesssim M_n (n/\log n)^{-\beta/(2\beta +1)}$$
and
$$\Pi\left( h(f_0, f_\theta)\leq M_n \epsilon_{n,0}| \mathbf x_n^n , \hat k\right) = 1 + o_p(1).$$
\end{proposition}
Equation \eqref{epsilonHist} of Proposition \ref{prop:hist} is proved in Appendix \ref{sec:pr:hist}, while the rest of the proof is given in Section B.6 of the supplementary material \cite{rousseau:szabo:15:supp}.

\section{Proofs} \label{sec:proof}

\subsection{Proof of Theorem \ref{thm:main}}\label{pr:thmain}

Following from the definition of $\hat{\lambda}_n$ given in \eqref{MMLE}
 we have that $m(\mathbf x_n|\lambda)\leq m(\mathbf x_n|\hat{\lambda}_n)$ for all $\lambda\in\Lambda_n$.
Therefore to prove our statement it is sufficient to show that with $P_{\theta_0}^n$-probability tending to one we have
$$\sup_{\lambda\in\Lambda_n\setminus\Lambda_0} m(\mathbf x_n|\lambda)< m(\mathbf x_n|\lambda_0)\leq \sup_{\lambda\in\Lambda_0}m(\mathbf x_n|\lambda),$$
where $\lambda_0$ is some hyper-parameter belonging to $\Lambda_0$ (possibly dependent on $n$).

We proceed in two steps. First we show that there exists a constant $C>0$ such that with $P_{\theta_0}^n$-probability tending to one we have
\begin{align}
m(\mathbf x_n|\lambda_0)\geq e^{-Cn\eps_{n,0}^2}\label{eq: LB_like}.
\end{align}
Then we finish the proof by showing that for any sequence $w_n'=o(M_n^2\wedge w_n^2)$ going to infinity
\begin{align}
P_{\theta_0}^n\left( \sup_{\lambda\in\Lambda_n\setminus\Lambda_0} m(\mathbf x_n|\lambda)   >e^{- nw_n' \epsilon_{n,0}^2}\right)=o(1).\label{eq: UB_like}
\end{align}

We prove the first inequality $\eqref{eq: LB_like}$ using the standard technique for lower bounds of the likelihood ratio (e.g. Lemma 10 of \cite{ghosal:vdv:07}). Without loss of generality we can assume that there exists $\lambda\in\Lambda_n$ such that $\eps_n(\lambda)\geq\eps_{n,0}$. Then take an arbitrary $\lambda_0\in\tilde\Lambda_0$ such that $\eps_n(\lambda_0)\leq M_1 \eps_{n,0}$ for an arbitrary $M_1>1$.
Then we have from the assumption (B1) and the definition of $\eps_n(\lambda)$ given in $\eqref{def:eps:2}$ that with $P_{\theta_0}^n$-probability tending to one  the following inequality holds
\begin{equation}\label{eq: Split01}
\begin{split}
m(\mathbf x_n|\lambda_0) &\geq \int_{\theta\in B_n( \theta_0, M_2 \epsilon_n(\lambda_0) , 2) }e^{\ell_n(\theta)-\ell_n(\theta_0)}d\Pi(\theta | \lambda_0)\\
&\geq
 \Pi\left( \left. B_n( \theta_0, M_2 \epsilon_n(\lambda_0) ,2) \right| \lambda_0   \right) e^{ - 2n\epsilon_n^2 (\lambda_0) M_2^2 }\\ 
&\geq e^{-(\tilde{c}_0+2M_2^2)M_1n\eps_{n,0}^2}.
\end{split}
\end{equation}

We now prove $\eqref{eq: UB_like}$.  Split $\Lambda_n\setminus\Lambda_0$ into balls of size $u_n/2$ and  choose for each  ball a point in $\Lambda_n\setminus\Lambda_0$. We denote by $(\lambda_i)_{i=1}^{N_n(\Lambda_n\setminus\Lambda_0)}$ these points.     Consider the set $\Theta_n(\lambda_i)$ defined in $\eqref{cond: complement}$ and divide it into sieves
$$S_{n,j}^{(i)} = \{ \theta \in \Theta_n(\lambda_i) ; j\epsilon_n(\lambda_i) c(\lambda_i) \leq d( \theta, \theta_0) \leq (j+1)    \epsilon_n(\lambda_i) c(\lambda_i)  \}.$$
We have following from assumption \eqref{cond: entropy}  that for all $j$
\begin{equation}\label{Nj}
\log N(\zeta  j\epsilon_n(\lambda_i) c(\lambda_i) {,} S_{n,j}^{(i)}, d(\cdot, \cdot) ) \leq c_1 n  j^2\epsilon_n(\lambda_i)^2 c(\lambda_i)^2 /2
\end{equation}
 and constructing a net of $S_{n,j}^{(i)}$ with radius $ \zeta j\epsilon_n(\lambda_i) c(\lambda_i)$ we have following from assumption $\eqref{cond: test}$ that there exist tests $\phi_{n,j}^{(i)}$ satisfying 
 \begin{equation}\label{test}
 \begin{split}
 E_{\theta_0}^n\left( \phi_{n,j}^{(i)} \right) &\leq e^{ -c_1 nj^2\epsilon_n(\lambda_i)^2 c(\lambda_i)^2  } , \\
 \int_{S_{n,j}^{(i)}}  Q^{\theta}_{\lambda_i,n}(1 - \phi_{n,j}^{(i)} ) d\Pi(\theta|\lambda_i) &\leq e^{-c_1 nj^2\epsilon_n(\lambda_i)^2 c(\lambda_i)^2 }\Pi(S_{n,j}^{(i)}|\lambda_i).
 \end{split}
 \end{equation}
Let us take the test $\phi_{n,i} = \max_j \phi_{n,j}^{(i)}$ and for convenience introduce the notation $B_n(\lambda)=\Theta_n(\lambda) \cap \{\theta:\, \|\theta-\theta_0\|\leq K\eps_n(\lambda)\}$. Then using the chaining argument, Markov's inequality, Fubini's theorem and $\eqref{cond: test}$ we get that 
\begin{equation}\label{eq: Split1}
\begin{split}
P_{\theta_0}^n&\left( \sup_{\lambda\in\Lambda_n\setminus\Lambda_0} m(\mathbf x_n|\lambda)> e^{- n w_n' \epsilon_{n,0}^2} \right)\\
& \leq 
\sum_{i=1}^{N_n(\Lambda_n\setminus\Lambda_0)} P_{\theta_0}^n\left( \sup_{\rho(\lambda_i,\lambda) \leq u_n} m(\mathbf x_n| \lambda)  >e^{-n w_n' \epsilon_{n,0}^2} \right) \\
&\leq \sum_{i=1}^{N_n(\Lambda_n\setminus\Lambda_0)} E_{\theta_0}^n [ \phi_{n,i} ] + e^{n w_n' \epsilon_{n,0}^2}\Big\{ \\
&\quad\sum_{i=1}^{N_n(\Lambda_n\setminus\Lambda_0)}E_{\theta_0}^n\Big(\sup_{\rho(\lambda_i,\lambda) \leq u_n}\int_{\psi_{\lambda_i,\lambda}^{-1}\{B_n(\lambda_i)\}}e^{\ell_n(\theta)-\ell_n(\theta_0)}d\Pi(\theta|\lambda)\Big) \\
& \quad +   \sum_{i=1}^{N_n(\Lambda_n\setminus\Lambda_0)} E_{\theta_0}^n\Big( \sup_{\rho(\lambda_i,\lambda) \leq u_n}\int_{\psi_{\lambda_i,\lambda}^{-1}\{\Theta_n(\lambda_i) \cap B_n(\lambda_i)^c\}} e^{\ell_n(\theta)-\ell_n(\theta_0)} ( 1 -\phi_{n,i} )d\Pi( \theta|\lambda) \Big) \\
&\quad+ \sum_{i=1}^{N_n(\Lambda_n\setminus\Lambda_0)} E_{\theta_0}^n\Big(  \sup_{\rho(\lambda_i,\lambda) \leq u_n}  \int_{\psi_{\lambda_i,\lambda}^{-1}\{\Theta_n(\lambda_i)^c\}}e^{\ell_n(\theta)-\ell_n(\theta_0)}d\Pi( \theta|\lambda) \Big)\Big\} \\
&  \leq  N_n(\Lambda_n\setminus\Lambda_0) 2e^{ -c_1 n\inf_{i}\epsilon_n(\lambda_i)^2 c(\lambda_i)^2}
+ e^{n w_n' \epsilon_{n,0}^2}\Big\{\\
&\quad\sum_{i=1}^{N_n(\Lambda_n\setminus\Lambda_0)}  \int_{B_n(\lambda_i) } Q_{\lambda_i,n}^\theta(\mathcal X_n) d\Pi( \theta|\lambda_i)\\
&\quad+\sum_{i=1}^{N_n(\Lambda_n\setminus\Lambda_0)}\int_{\Theta_n(\lambda_i) \cap B_n(\lambda_i)^c} Q^\theta_{\lambda_i,n}( 1 -\phi_{n,i} )d\Pi( \theta|\lambda_i)\\
&\quad+{\sum_{i=1}^{N_n(\Lambda_n\setminus\Lambda_0)} \int_{\Theta_n(\lambda_i)^c}Q^{\theta}_{\lambda_i,n}(\mathcal{X}_n)d\Pi( \theta|\lambda_i)}\Big\}.
\end{split}
\end{equation}

Next we deal with each term on the right hand side of $\eqref{eq: Split1}$ separately and show that all of them tend to zero. One can easily see that since $\lambda_i\in\Lambda_n\setminus\Lambda_0$ and following the definition of $c(\lambda_i)$ given below $\eqref{change:distance}$, we have that 
\begin{align*}
N_n(\Lambda_n\setminus\Lambda_0) e^{ -(c_1/2) n\inf_{i}\epsilon_n(\lambda_i)^2 c(\lambda_i)^2}\leq N_n(\Lambda_n\setminus\Lambda_0) e^{ - (c_1/2)w_n^2 n\epsilon_{n,0}^2   } = o(1).
\end{align*}

For the second term we have following from assumption $\eqref{cond: Q}$, the definitions of $\eps_n(\lambda)$ and the set $\Lambda_0$ given in $\eqref{def:eps:2}$ and $\eqref{def: Lambda0}$, respectively, that
\begin{align*}
e^{  n w_n' \epsilon_{n,0}^2 }\sum_{i=1}^{N_n(\Lambda_n\setminus\Lambda_0)} & \int_{B_n(\lambda_i) } Q_{\lambda_i,n}^\theta(\mathcal X_n) d\Pi(\theta| \lambda_i )\\
 &\leq
 \sum_{i=1}^{N_n(\Lambda_n\setminus\Lambda_0)} e^{ n w_n' \epsilon_{n,0}^2 } e^{o(1)n\eps_n^2(\lambda_i)}\Pi(B_n(\lambda_i) |\lambda_i)\\
&\leq e^{- n M_n^2 \epsilon_{n,0}^2(\tilde{c}_0^{-1}+o(1))}=o(1).
\end{align*}

Next following from $\eqref{test}$ we have that 
 \begin{align*}
 e^{  n w_n' \epsilon_{n,0}^2 }\sum_{i=1}^{N_n(\Lambda_n\setminus\Lambda_0)} \int_{\Theta_n(\lambda_i) \cap B_n(\lambda_i)^c}& Q^\theta_{\lambda_i,n}( 1 -\phi_{n} )d\Pi( \theta | \lambda_i)\\
 & \leq e^{  n w_n' \epsilon_{n,0}^2 }\sum_{i=1}^{N_n(\Lambda_n\setminus\Lambda_0)} e^{-c_1n \epsilon_n(\lambda_i)^2 c(\lambda_i)^2} \\
& \leq  e^{-c_1nw_n^2\epsilon_{n,0}^2 (1+ o(1))  } = o(1).
\end{align*}
Finally we have following assumption $\eqref{cond: complement}$ that the fourth term on the right hand side of $\eqref{eq: Split1}$ can be bounded from above by
\begin{align*}
  e^{n w_n' \epsilon_{n,0}^2}{\sum_{i=1}^{N_n(\Lambda_n\setminus\Lambda_0)} \int_{\Theta_n(\lambda_i)^c}Q^{\theta}_{\lambda_i,n}(\mathcal{X}^{(n)})d\Pi(\theta | \lambda_{i})}&\leq
N_n(\Lambda_n\setminus\Lambda_0) e^{-  (w_n^2-w_n')n\epsilon_{n,0}^2} \\
&\leq e^{- n w_n^2\epsilon_{n,0}^2( 1 + o(1))} = o(1).
\end{align*}

\subsection{Proof of Corollary \ref{thm: contraction}} \label{pr:thm-contraction}

The proof of Corollary \ref{thm: contraction}, follows the same lines of reasoning as Theorem 1 in \cite{DRRS:arxiv}, with the adding remark that 
 $$m(\mathbf x_n|\hat \lambda_n) \geq m(\mathbf x_n|\lambda) , \quad \forall \lambda \in \Lambda_n,$$
 so that no uniform lower bound in the form  $\inf_{\lambda\in\Lambda_0} m(\mathbf x_n|\lambda)$ is required. We have 
\begin{align*}
 E_{\theta_0}^n &\Pi\left( d(\theta, \theta_0) > M M_n\epsilon_{n,0} | \mathbf x_n ; \hat \lambda_n \right)\\
   &\quad=   E_{\theta_0}^n\left( \frac{ \int_{d(\theta, \theta_0) > M M_n\epsilon_{n,0}} e^{\ell_n(\theta) - \ell_n(\theta_0) }  d\Pi( \theta | \hat \lambda_n )}{\int_{\Theta} e^{\ell_n(\theta) - \ell_n(\theta_0) }  d\Pi( \theta | \hat \lambda_n ) } \right)
\equiv E_{\theta_0}^n\left(  \frac{ H_n(\hat \lambda_n) }{m(\mathbf x_n|\hat \lambda_n) } \right).
 \end{align*}
We  construct $\phi_n =  \max_{\lambda_i}  \max_j \max_l \phi_n^{(i)}(\theta_{j,l})$, with $(\lambda_i)_{i\leq N_{n}(\Lambda_0)}$ a net of $\Lambda_0$ with radius $u_n$, and for all $j \geq M M_n$,  $(\theta_{j,l})_{l\leq N_{n,j}}$ a $\zeta j \epsilon_n(\lambda_i)$ net of $\bar S_{n,j} = \{\theta, j \epsilon_{n,0} \leq d(\theta, \theta_0)\leq (j+1) \epsilon_{n,0}\} \cap \Theta_n(\lambda_i)$. By assumption (C2), $\log N_{n,j} \leq  c_1 n j^2 \epsilon_n^2 /2 $ and $\log N_{n}(\Lambda_0) \leq c_3 n \epsilon_{n,0}^2$ (for some $c_3>0$). Then we have for any $c_2>0$
\begin{equation} \label{decomp:rate}
\begin{split}
E_{\theta_0}^n\left(  \frac{ H_n(\hat \lambda_n) }{m(\mathbf x_n|\hat \lambda_n) } \right)&\leq P_{\theta_0}^n(\hat\lambda_{n}\notin \Lambda_0) +E_{\theta_0}^n\left( \phi_n \right)+ P_{\theta_0}^n[ m(\mathbf x_n|\hat \lambda_n) < e^{- c_2 n\epsilon_{n,0}^2 } ]  \\
&  \qquad+ e^{c_2 n\epsilon_{n,0}^2 } E_{\theta_0}^n\left[ (1 - \phi_n)\sup_{\lambda\in \Lambda_0} H_n(\lambda )\right].
\end{split}
\end{equation}
We assumed that the first term tends to zero (see Theorem \ref{thm:main} for verification of this condition in case of MMLE). Furthermore by construction  
$$E_{\theta_0}^n\left( \phi_n \right)\leq N_{n}(\Lambda_0)\sup_{i} \sum_{j \geq M M_n} e^{c_1 n j^2 \epsilon_n^2(\lambda_i) /2}e^{- c_1 n j^2 \epsilon_n^2(\lambda_i) } \lesssim e^{- n c_1 M_n^2 \epsilon_{n,0}^2 /4} .$$
Also
$$ P_{\theta_0}^n[ m(\mathbf x_n|\hat \lambda_n) < e^{- c_2 n\epsilon_{n,0}^2 } ] \leq  P_{\theta_0}^n[ m(\mathbf x_n|\lambda_0) < e^{- c_2 n\epsilon_{n,0}^2 } ] = o(1) $$
following from \eqref{eq: LB_like}  with $c_2\geq c_3+M_1(\tilde{c}_0+2M_2^2+2)$. The control of the  last term of \eqref{decomp:rate} follows from the proof of Theorem 1 of \cite{DRRS:arxiv}.

\subsection{Proof of Theorem \ref{thm: hierarchical}}\label{sec: HBproof}

As a first step for notational convenience let us denote by $B_n^c$ the sets $\{\theta:\, d(\theta,\theta_0)\geq MM_n \eps_{n,0}\}$ or $\{\theta:\, d(\theta,\theta_0)\leq \delta_n \eps_{n,0}\}$
\begin{align}
\Pi( B_n^c |\mathbf x_n)&=\int_{\Lambda_0(M_n)}\Pi( B_n^c |\mathbf x_n,\lambda)\tilde\pi(\lambda|\mathbf x_n)d\lambda
+\int_{\Lambda_0(M_n)^c}\Pi( B_n^c |\mathbf x_n,\lambda)\tilde\pi(\lambda|\mathbf x_n)d\lambda\nonumber\\
&\leq \sup_{\lambda\in\Lambda_0(M_n)}\Pi( B_n^c |\mathbf x_n,\lambda)+ \int_{\Lambda_0(M_n)^c}\tilde\pi(\lambda|\mathbf x_n)d\lambda.\label{eq: help001} 
\end{align}
Then from the proofs of Theorem 1 of \cite{DRRS:arxiv} and Theorem \ref{thm: LB} follows that the expected value of the first term on the right hand side of the preceding display tends to zero. We note that assumption (H2) is needed to deal with the denominator in the posterior, unlike in Corollary \ref{thm: contraction}, where weaker assumptions were sufficient following from the definition of the maximum marginal likelihood estimator $\hat\lambda_n$.

Hence it remained to deal with the second term on the right hand side of $\eqref{eq: help001}$. The hyper-posterior takes the form
$$\pi(\lambda|\mathbf x_n)\propto m(\mathbf x_n |\lambda)\tilde\pi(\lambda) $$
and from the proof of Theorem 1 of \cite{DRRS:arxiv} (page 10-11) and $\eqref{eq: Split1}$ in the proof of Theorem \ref{thm:main} we have with $P_{\theta_0}^n$-probability tending to one that
\begin{align*}
m(\mathbf x_n |\lambda)\geq e^{-(\tilde{c}_0+2M_2^2)\tilde{w}_n^2 n\eps_{n,0}^2}\quad\text{for $\lambda\in\tilde\Lambda_0(\tilde{w}_n)$,}\quad\text{and}\\ 
m(\mathbf x_n |\lambda)\leq e^{-w_n' n\eps_{n,0}^2}\quad\text{for $\lambda\in\Lambda_n\setminus\Lambda_0(M_n)$,}
\end{align*}
for any $w_n'=o(M_n^2\wedge w_n^2)$, hence there exists $w_n'$, which also satisfies $\tilde{w}_n=o(w_n')$.
Therefore with $P_{\theta_0}^n$-probability tending to one we also have that 
\begin{align*}
\int_{\Lambda_n\setminus\Lambda_0(M_n)}\pi(\lambda|\mathbf x_n)d\lambda
&\leq \frac{e^{-w_n'^2 n\eps_{n,0}^2}}{e^{-(\tilde{c}_0+2M_2^2)\tilde{w}_n^2 n\eps_{n,0}^2}\int_{\tilde\Lambda_0(\tilde{w}_n)}\tilde\pi(\lambda) d\lambda }=o(1).
\end{align*}
Finally similarly to the preceding display we have that
\begin{align*}
E_{\theta_0}^n\int_{\Lambda\setminus\Lambda_n}\pi(\lambda|\mathbf x_n)d\lambda&\leq \frac{\int_{\Lambda\setminus\Lambda_n}E_{\theta_0}^n m(\mathbf x_n |\lambda) \tilde\pi(\lambda) d\lambda}{e^{-(\tilde{c}_0+2M_2^2)\tilde{w}_n^2 n\eps_{n,0}^2}\int_{\tilde\Lambda_0(\tilde{w}_n)}\tilde\pi(\lambda) d\lambda}+o(1)\\
&\lesssim e^{(\tilde{c}_0+2M_2^2+1)\tilde{w}_n^2 n\eps_{n,0}^2} \int_{\Lambda\setminus\Lambda_n} \tilde\pi(\lambda) d\lambda+o(1)= o(1),
\end{align*}
finishing the proof.


\begin{appendix}

\section{Proof of the Lemmas about the rate $\eps_n(\lambda)$}

\subsection{Proof of Lemma  \ref{lem:eps:T1}} \label{app:T1}

We have $\|\theta -\theta_0\|_2^2 = \sum_{j=1}^k (\theta_j - \theta_{0,j})^2 + \sum_{j=k+1}^\infty \theta_{0,j}^2 $ so that $\|\theta -\theta_0\|_2^2 \leq K^2\epsilon^2$ if and only if $  \sum_{j=1}^k (\theta_j - \theta_{0,j})^2 \equiv \|\theta - \theta_{0,[k]}\|_2^2 \leq  \delta^2$, with $\delta^2 = K^2 \epsilon^2 -  \sum_{j=k+1}^\infty \theta_{0,j}^2 $, and $\theta_{0,[k]} = ( \theta_{0,j}, j\leq k)$. Then 
\begin{equation*}
\begin{split}
\int_{\theta \in \RR^k }g(\theta)  \1\{ \|\theta-\theta_{0,[k]}\|_2\leq \delta \} d\theta &\leq  \|g\|_\infty^k \frac{ \pi^{k/2} \delta^k }{ \Gamma( k/2+1)} \\
 & \geq \underline g^k \frac{ \pi^{k/2} \delta^k }{ \Gamma( k/2+1)} 
\end{split}
\end{equation*}
with $\underline g = \inf_{B_k(\delta)}g(x) $ where $B_k(\delta) = \{ x; \min_{i\leq k} |x-\theta_{0,i}| \leq \delta\}$. The Sterling formula implies that both the lower and upper bounds have the form 
$ \exp\{ k \log ( C \delta /\sqrt{k} )\}$ and since $\delta = o(1)$ this is equivalent to\\
 $\exp\{ k \log (\delta /\sqrt{k} ) ( 1 + o(1) )\}$. We thus have 
\begin{equation*}
\epsilon_n(k) > \left( \sum_{i>k} \theta_{0,i}^2\right)^{1/2}/K \quad \mbox{and } \quad n \epsilon_n^2 (k) = k \log ( \sqrt{k}/ s_n) ( 1+ o(1) ), 
\end{equation*} 
with  $ s_n^2 = K^2 \epsilon_n^2(k) -  \sum_{j=k+1}^\infty \theta_{0,j}^2 $. In other words $ s_n >0$ and 
\begin{equation}\label{deltan}  s_n^2 + \sum_{j=k+1}^\infty \theta_{0,j}^2 = \frac{ K^2 k}{n} \log \left( \frac{ \sqrt{k}}{ s_n} \right)(1 + o(1)). \end{equation}
 Also if $  \sum_{j=k+1}^\infty \theta_{0,j}^2 = o( k\log n/n) $, then \eqref{deltan} implies that 
$$ s_n^2  = \frac{ K^2 k}{n} \log \left( \frac{ \sqrt{k}}{ s_n} \right)( 1 + o(1) )  \quad \Rightarrow\quad  s_n^2 = \frac{ K^2 k}{ 2 n }\log ( 2n/K^2)(1+ o(1) ).$$
Now take $\theta_0 \in \mathcal H_{\infty}( \beta, L) \cup\mathcal S_\beta(L)$, since $\sum_{i>k}\theta_{0,i}^2 \lesssim k^{-2\beta}$, 
choosing $k = \lfloor (n/\log n)^{1/(2 \beta+1)}\rfloor $ leads to 
 $ \epsilon_{n,0} \lesssim (n/\log n)^{-\beta/(2\beta+1)}$ . Finally considering $\theta_{0,i}^2 = (1+i)^{-2\beta -1} $ for $\mathcal H_{\infty}( \beta, L)$ implies that this is also a lower bound in this case.
Furthermore for all $\delta_n = o(1/M_n) $ and for all $k $ such that 
 $$k^{-2\beta} +\frac{ k\log n }{ n} \leq M_n^2 (n/\log n)^{-2\beta/(2\beta+1)}\quad \Rightarrow\quad k \lesssim M_n^2 (n/\log n)^{1/(2\beta+1)}$$
 and $ \delta_n^2 \left( k^{-2\beta} + k\log n/n \right) = o (k^{-2\beta}) = o(\sum_{i>k}\theta_{0,i}^2)$
 so that 
$$ \Pi( \|\theta  - \theta_0\|\leq \delta_n \epsilon_n(k) |k) = 0$$
  and condition \eqref{cond: Q2} is verified.


\subsection{Proof of Lemma \ref{lem:eps:T2T3}} \label{app:T2T3}
We need to study 
$$\inf_{h\in\mathbb{H}^{\alpha,\tau}:\, \|h-\theta_0\|_2\leq\epsilon_n} \|h\|_{\mathbb{H}^{\alpha,\tau}}^2.$$
Let us distinguish three cases $\beta>\a+1/2$, $\beta<\a+1/2$ and $\beta=\a+1/2$, and note that the following computations hold both for the truncated and non-truncated versions of the priors (T2) and (T3). 

In the case $\beta > \alpha +1/2$ and if  $\theta_{0,i}^2 \leq L i^{-2\beta -1} $ for all $i$, then 
\begin{equation*}
\begin{split}
\inf_{h\in\mathbb{H}^{\alpha,\tau}:\, \|h-\theta_0\|_2\leq\epsilon} \|h\|_{\mathbb{H}^{\alpha,\tau}}^2 & \leq \tau^{-2}L \sum_{i=1}^\infty i^{2\alpha-2\beta } \lesssim \frac{ L \tau^{-2}}{  \beta - \alpha - 1/2}
\end{split}
\end{equation*}
while when $\theta_0 \in \mathcal S_\beta(L)$ 
$\inf_{h\in\mathbb{H}^{\alpha,\tau}:\, \|h-\theta_0\|_2\leq\epsilon} \|h\|_{\mathbb{H}^{\alpha,\tau}}^2  \leq \tau^{-2}L.$
Also
\begin{equation*}
  n^{-\frac{\alpha}{2\alpha +1}}\tau^{\frac{1}{2\alpha+ 1}}\lesssim \epsilon_n(\alpha,\tau) \lesssim   n^{-\frac{\alpha}{2\alpha +1}}\tau^{\frac{1}{2\alpha+ 1}}+ \left(\frac{1}{n\tau^2(\beta - \alpha -1/2 )} \right)^{1/2} 
  \end{equation*}
if $\theta_0 \in \mathcal H_\infty ( \beta, L)$, while 
\begin{equation*}
n^{-\frac{\alpha}{2\alpha +1}}\tau^{\frac{1}{2\alpha+ 1}}\lesssim \epsilon_n(\alpha,\tau) \lesssim n^{-\frac{\alpha}{2\alpha +1}}\tau^{\frac{1}{2\alpha+ 1}}+ \left(\frac{1}{n\tau^2} \right)^{1/2} 
  \end{equation*}
if $\theta_0 \in \mathcal S_\beta(L)$.  Now, if $0 < \beta < \alpha +1/2$, with $\theta_0 \in \mathcal H_\infty ( \beta, L)$ 
\begin{equation*}
\begin{split}
\inf_{h\in\mathbb{H}^{\a,\tau}:\, \|h-\theta_0\|_2\leq\epsilon} \|h\|_{\mathbb{H}^{\alpha,\tau}}^2 & \leq \tau^{-2}L\sum_{i=1}^{(\frac{L}{2\beta})^{\frac{1}{2\beta}} \epsilon_n^{-\frac{1}{\beta}}} i^{2\alpha-2\beta } \lesssim L^{\frac{2\a+1}{2\beta}}\frac{\tau^{-2} \epsilon^{- \frac{2\alpha - 2 \beta +1}{\beta} }}{ 2 \alpha +1 -2 \beta}  
\end{split}
\end{equation*}
and when $\theta_0 \in \mathcal S_\beta ( L)  $
\begin{equation*}
\begin{split}
\inf_{h\in\mathbb{H}^{\a,\tau}:\, \|h-\theta_0\|_2\leq\epsilon} \|h\|_{\mathbb{H}^{\alpha,\tau}}^2 & \leq \tau^{-2}L^{\frac{2\a+1}{2\beta}} \epsilon^{- (2\alpha - 2 \beta +1)/\beta }.
\end{split}
\end{equation*}

If $\beta = \alpha + 1/2$, the same result holds for $\theta_0 \in \mathcal S_\beta ( L)$, but it becomes 
\begin{equation*}
\begin{split}
\inf_{h\in\mathbb{H}^{\a,\tau}:\, \|h-\theta_0\|_2\leq\epsilon} \|h\|_{\mathbb{H}^{\alpha,\tau}}^2 & \leq \frac{\tau^{-2}L}{\beta} |\log ( \epsilon)|(1+o(1)).
\end{split}
\end{equation*}
when  $\theta_0 \in \mathcal H_\infty ( \beta, L)$ . 
These lead to the upper bound in \eqref{eps:lambda:beta+-} and \eqref{eps:lambda:beta=}.

Furthermore for every $\theta_0\in \mathcal S_{\beta}(L) \cup \mathcal H_{\infty}( \beta, L )$ satisfying $\|\theta_0\|_2>2 \eps$, when $\|h-\theta_0\|_2\leq \eps$ then $\|h\|_2>\|\theta_0\|_2/2$, hence
\begin{align*}
\inf_{h\in\mathbb{H}^{\a,\tau}:\, \|h-\theta_0\|_2\leq\epsilon} \|h\|_{\mathbb{H}^{\alpha,\tau}}^2 \geq  \tau^{-2} \inf_{h\in\mathbb{H}^{\a,\tau}:\, \|h-\theta_0\|_2\leq\epsilon} \|h\|_2^2\gtrsim\|\theta_0\|_2^2 \tau^{-2}.
\end{align*}
Hence if $\|\theta_0\|_2>2\epsilon_{n}(\alpha,\tau)$ for $a(\alpha,\beta)$ defined in Lemma \ref{lem:eps:T2T3}, 
$$\epsilon_n( \lambda) \gtrsim \frac{\|\theta_0\|_2}{\sqrt{n\tau^2} }+ n^{-\alpha/(2\alpha +1)}\tau^{1/(2\alpha+ 1)}$$
and for all $\tau^2 n  $ lower bounded by a positive constant the above inequality remains valid when $\|\theta_0\|_2 \leq 2\epsilon_{n}(\lambda)$, providing us the lower bound in \eqref{eps:lambda:beta+-} and \eqref{eps:lambda:beta=}.


\subsection{Proof of Lemma \ref{lem:eps:2:T2T3}} \label{app:2:T2T3}
The proof is based on minimizing the upper bounds obtained in Lemmas \ref{lem:eps:T1} and   \ref{lem:eps:T2T3}. 

$\bullet$ First consider  $\lambda= \tau$. When $\beta > \alpha +1/2$, note that for all $\tau \geq n^{-1/(4\alpha+4)}$
$$ \left(\frac{1}{n\tau^2} \right)^{1/2} \lesssim n^{-\alpha/(2\alpha +1)}\tau^{1/(2\alpha+ 1)}$$
so that 
$\epsilon_n(\tau) \asymp n^{-\alpha/(2\alpha +1)}\tau^{1/(2\alpha+ 1)}$ which is minimized at $\tau \asymp n^{-1/(4\alpha+4)}$ so that 
\eqref{eps0:tau:1} is verified. Following from $\eqref{eps:lambda:beta+-}$ the lower bound is obtained with every $\|\theta_0\|_2\geq c>0$, for any arbitrary positive constant $c$. Indeed in this case, we have
$\epsilon_n(\tau)\gtrsim (n\tau^2)^{-1/2}$
  which implies that the lower bound  is the same as the upper bound \eqref{eps0:tau:1}. Furthermore we note that the lower bound 
	\begin{align}
	\eps_{n,0}\gtrsim n^{-(2\a+1)/(4\a+4)}\label{eq: LBgen_alpha}
 \end{align}
holds for every $\theta_0\neq 0$ (and large enough $n$). Therefore we also have for every $\t_0$ satisfying $\eps_{n}(\t_0)\lesssim \eps_{n,0}$ that $\t_0\gtrsim n^{-1/(4\a+4)}$.

 When $\beta < \alpha + 1/2$ we have for all $\tau \geq n^{-(\beta - \alpha)/( 2\beta +1)}$ that 
$\epsilon_n(\tau) \asymp n^{-\frac{\alpha}{2\alpha +1}}\tau^{\frac{1}{2\alpha+ 1}}$, which is minimized at $\tau \asymp n^{-(\beta - \alpha)/( 2\beta +1)}$,
 leading to the upper bound \eqref{eps0:tau:2}. The upper bound is obtained choosing for instance 
  $\theta_{0,i} = \sqrt{L}i^{-\beta -1/2}$ for all $i\leq K_n$, for some sequence $K_n$ going to infinity, so that
\begin{equation*}
\begin{split}
 \inf_{\|h-\theta\|_2 \leq \epsilon_n(\tau)} \|h\|_{\mathbb H^{\alpha,\tau} }^2 
 &\geq \tau^{-2}  \sum_{i=1}^{K_n} i^{2\alpha+1}[\theta_{0,i}^2 - 2 \theta_{0,i}( \theta_{0,i}-h_i )] \\
 &\gtrsim \tau^{-2}\left(  L K_n^{2\alpha - 2\beta+1} - 2 \sqrt{L} \epsilon_n(\tau) K_n^{2\alpha -\beta +1} \right)\\
 &\gtrsim \tau^{-2} K_n^{2\alpha - 2\beta+1}
 \end{split}
 \end{equation*}
and $K_n \leq k_0 \epsilon_n(\tau)^{-1/\beta} $. This leads to $\epsilon_n(\tau ) \geq (n\tau^2)^{-\beta/(2\alpha+1)} $, with an extra $\log n$ term in the case $\alpha + 1/2 = \beta$ and $\theta_0\in \mathcal{H}_{\infty}(\beta,L)$ so that the lower bound is of the same order as the upper bound \eqref{eps:lambda:beta=} which in terms implies that the lower bound  is the same as the upper bound 
\eqref{eps0:tau:2}.
 
We now consider the case $\lambda = \alpha$, then we have a generic upper bound for 
  $\epsilon_n( \alpha) $ in the form $n^{-(\alpha \wedge \beta)/(2\alpha+1)}$ following from $\eqref{eps:lambda:beta+-}$ and   $\theta_0 \in \mathcal H_\infty ( \beta, L)\cup \mathcal S_\beta(L)$, while the lower bound is a multiple of $n^{-\alpha/(2\alpha+1)}$. We thus have 
  $\epsilon_{n,0}  \lesssim n^{-\beta/(2 \beta +1)} $ for all $\theta_0 \in \mathcal H_\infty ( \beta, L)\cup S_\beta(L)$ and the constant depends only on $\beta $ and $L$.

\subsection{Proof of Equation \eqref{epsilonHist}  in Proposition \ref{prop:hist}}\label{sec:pr:hist}
We prove the first part of proposition, namely the bounds on $\epsilon_n(k)$. Denote by $g_0$ the function 
\begin{equation*}
g_0 (x)  = k\sum_{j=1}^k \tilde \eta_j \1_{I_j}(x),
\end{equation*}
then $g_0$ is the projection of $\sqrt{f_0}$ on the set of piecewise constant functions on a $k$ regular grid and for any $\theta \in \mathcal S_k$ the $k$-dimensional simplex, 
$$h^2(f_0, f_\theta)  = h^2(f_0,g_0^2) + \sum_{j=1}^k (\sqrt{\theta_j} - \tilde \eta_j \sqrt{k} )^2  \geq  h^2(f_0,g_0^2)  = b(k)^2.$$
Define $\bar \theta_{j,k} = (\tilde \eta_j\sqrt{ k})^2/ \sum_l \tilde\eta_l^2 k $ and for some  $v_n= o(1) $ consider $\theta  = (\theta_1,.., \theta_k)\in \mathcal S_k$ satisfying
$ |\theta_j -\bar \theta_{j,k}|\leq \bar \theta_{j,k} v_n $ for $j \leq k-1$. Then 
$ |\theta_k -\bar \theta_{k,k}| \leq \sum_{j=1}^{k-1}  \bar\theta_{j,k} v_n \leq v_n.$
Note that  $b(k)^2 = 1 - \sum_{j=1}^k \tilde \eta_j^2 k$, so that
\begin{equation*}
\begin{split}
\sum_{j=1}^k (\sqrt{\theta_j} - \tilde \eta_j \sqrt{k} )^2& = \sum_{j=1}^k \Big(\sqrt{\theta_j} -\sqrt{ \bar \theta_{j,k} } \sqrt{ \sum_l \tilde\eta_l^2 k } \Big)^2\\
& \leq 2 \sum_{j=1}^k (\sqrt{\theta_j} - \sqrt{\bar \theta_{j,k}} )^2 + 2 \sum_{j=1}^k \bar \theta_{j,k}(  \sqrt{ \sum_l \tilde\eta_l^2 k } - 1 )^2 \\
& \leq 2v_n^2 + 2 b(k)^2,
\end{split}
\end{equation*}
which implies that for such $\theta$, 
$h^2(f_0, f_\theta)  \leq 3 b(k)^2 + 2 v_n^2.$
Since $c_0 \leq f_0  \leq C_0$, $c_0/k \leq \bar \theta_{j,k}\leq C_0/k$ and  we also have, as in the proof of Lemma 6.1 of \cite{ghosal:ghosh:vdv:00}, that if $v_n \leq c_0/(2k) $,  then 
$v_n \leq \bar \theta_{k,k}/2$ and
\begin{equation*}
\begin{split}
\pi\left(  |\theta_j -\bar \theta_{j,k}|\leq \bar \theta_{j,k} v_n,\, \forall j \leq k-1\right) &\gtrsim  \frac{ \Gamma(k \alpha )}{ \Gamma( \alpha)^k }\bar \theta_{k,k}^{\alpha-1}  \prod_{j\leq k-1} \int_{\bar \theta_{j,k}(1- v_n)}^{\bar \theta_{j,k} (1+ v_n) }x^{\alpha-1} dx \\
&\gtrsim  \frac{  ( C_1v_n)^k \Gamma(k \alpha )}{ (\alpha\Gamma( \alpha))^{k-1}\Gamma(\alpha) }\prod_{j\leq k-1} \bar \theta_{j,k}^{\alpha}\\
& \gtrsim \frac{ { (C_2 v_n)^k \Gamma(k\alpha )k^{-k \alpha} }}{ (\alpha\Gamma( \alpha))^{k-1}\Gamma(\alpha) },
\end{split}
\end{equation*}
for some constant $C_1,C_2>0$. Since $\alpha \leq A  $,  if $v_n = n^{-h}$ for some $h>0$, 
\begin{equation*}
\pi\left(  |\theta_j -\bar \theta_{j,k}|\leq \bar \theta_{j,k} v_n , \, \forall j \leq k-1\right) \gtrsim e^{ - c k  \log n },
\end{equation*}
which implies that for all $k$ such  that $b(k)^2 \lesssim k\log n/ n $ we have $\epsilon_{n}(k)^2 \lesssim b(k)^2 + k \log n/n$. 
We now bound from below $\epsilon_n(k)$. Since $h^2(f_0, f_\theta)  = b(k)^2 + \sum_{j=1}^k \left( \sqrt{{\theta_j}} - \sqrt{\bar \theta_{j,k}}  \sqrt{1 - b(k)^2 }\right)^2 $, on the set 
$h^2(f_0, f_\theta)\leq \epsilon_n^2$, $b(k)^2 \leq \epsilon_n^2$ and $\sum_{j=1}^k \left( \sqrt{{\theta_j}} - \sqrt{\bar \theta_{j,k}}  \sqrt{1 - b(k)^2 }\right)^2 \leq \epsilon_n^2$. Using elementary algebra and Cauchy-Schwarz inequality we have if $\epsilon_n$ is small, $b(k)$ is small and 
\begin{equation*}
\begin{split}
\sum_{j=1}^k& \left( \sqrt{{\theta_j}} - \sqrt{\bar \theta_{j,k}}  \sqrt{1 - b(k)^2 }\right)^2\\ 
& \geq  \sum_{j=1}^k \left( \sqrt{\theta_j} - \sqrt{\bar \theta_{j,k}} \right)^2+ \frac{ b^4(k) }{ 4 } - 2 b^2(k) \sqrt{\sum_{j=1}^k \left( \sqrt{\theta_j} - \sqrt{\bar \theta_{j,k}} \right)^2} \\
&= \left(\sqrt{ \sum_{j=1}^k \left( \sqrt{{\theta_j}} - \sqrt{\bar \theta_{j,k}} \right)^2}  - \frac{ b(k)^2}{ 2 } \right)^2 .
\end{split}
\end{equation*}
Over the set $ \sum_{j=1}^k \left( \sqrt{{\theta_j}} - \sqrt{\bar \theta_{j,k}} \right)^2 \geq \epsilon_n^2/2$, then 
$$\sqrt{ \sum_{j=1}^k \left( \sqrt{{\theta_j}} - \sqrt{\bar \theta_{j,k}} \right)^2} \geq  b(k)^2 $$
and 
$$\sum_{j=1}^k \left( \sqrt{{\theta_j}} - \sqrt{\bar \theta_{j,k}}  \sqrt{1 - b(k)^2 }\right)^2  \geq \frac{ 1 }{ 4 } \sum_{j=1}^k \left( \sqrt{{\theta_j}} - \sqrt{\bar \theta_{j,k}} \right)^2,$$
so that if $h(f_0, f_\theta) \leq \epsilon_n$ small enough, then 
\begin{equation*}
h^2(f_0, f_\theta)  \geq b(k)^2 +\frac{ 1 }{ 4 } \sum_{j=1}^k \left( \sqrt{{\theta_j}} - \sqrt{\bar \theta_{j,k}} \right)^2  .
\end{equation*}
Hence 
$$\Pi\{h^2(f_0, f_\theta) \leq K \epsilon_{n}(k)^2 \} \leq \Pi\Big( \sum_{j=1}^k \left( \sqrt{\theta_j} - \sqrt{\bar \theta_{j,k}} \right)^2 \leq K \epsilon_{n}(k)^2 - b(k)^2 \Big),$$
 with $b(k)^2 {<}K \epsilon_{n}(k)^2$. Set $ s_n^2 = K \epsilon_{n}(k)^2 - b(k)^2 $. 
 On the set $$ \sum_{j=1}^k \left( \sqrt{\theta_j} - \sqrt{\bar \theta_{j,k}} \right)^2 \leq  s_n^2, $$ 
we split $\{1, \cdots, k-1\}$ into $| \sqrt{\theta_j} - \sqrt{\bar \theta_{j,k}} | \leq 1/\sqrt{k}$ and $| \sqrt{\theta_j} - \sqrt{\bar \theta_{j,k}} | > 1/\sqrt{k}$. The cardinality of the latter is bounded from above by $ s_n^2k$. Moreover if 
 $| \sqrt{\theta_j} - \sqrt{\bar \theta_{j,k}} | \leq 1/\sqrt{k}$ then {by triangle inequality} $\sqrt{\theta_j} \lesssim {1/\sqrt{k}}$ else $\sqrt{\theta_j} \lesssim  s_n$. We have
 \begin{equation*}
\begin{split}
& 
\Pi\Big( \sum_{j=1}^k ( \theta_j^{1/2} - \bar \theta_{j,k}^{1/2} )^2 \leq  s_n^2 \Big)  \leq 
\frac{ \pi^{k/2} \Gamma(\alpha k )  s_n^k }{ \Gamma(\alpha)^k \Gamma(k/2+1)} \sum_{l=0}^{\lfloor  s_n^2 k\rfloor} \binom{k}{l} s_n^{(2\alpha-1)l }k^{-(k-l)(\alpha-1/2)}\\
&\leq \frac{ \pi^{\frac{k}{2}} \Gamma(\alpha k )  s_n^k }{ \Gamma(\alpha)^k \Gamma(k/2+1)}\Big( k^{-k(\alpha-1/2)}\\
&\qquad + \sum_{l\leq  s_n^2 k} C e^{ l \log (k)  + 2l - (k-l)(\alpha-1/2)\log (k) + 2l(\alpha-1/2)\log  (s_n)} \Big)\\
&\lesssim \exp\left\{ \alpha k \log (k) -k \log \Gamma(\alpha) - \frac{k}{2}\log (k) +k\log ( s_n)  - k (\alpha-\frac{1}{2})\log k + O(k) \right\} \\
&\lesssim \exp\left( k\log ( s_n) + O(k) \right) 
 \end{split}
\end{equation*}
if $\alpha \geq 1/2$. If  $\alpha < 1/2$, 
for each $\theta$ split $\{1,\cdots, k-1\}$ into the set $S$ of indices where $\theta_i \geq \rho_n/k $ and its complement, with $\rho_n = o(1)$. The number of indices such that $\theta_i < \rho_n/k $  is bounded by $O( s_n^2 k) $ on the set $\sum_{j=1}^k \left( \sqrt{\theta_j} - \sqrt{\bar \theta_{j,k}} \right)^2 \leq  s_n^2$, so that 
\begin{equation*}
\begin{split}
& 
\Pi\left( \sum_{j=1}^k \left( \sqrt{\theta} - \sqrt{\bar \theta_{j,k}} \right)^2 \leq  s_n^2  \right) \\
& \leq   \frac{ \Gamma ( k\alpha )}{  \Gamma(\alpha)^{{k}} } \sum_{S\subset \{ 1,.., k \}} \int_{\sum_{i\in S} ( \theta_i^{1/2} - \bar \theta_{i,k}^{1/2} )^2  \leq  s_n^2  }\1_{\substack{\forall i \in S\\ \theta_i \geq \frac{\rho_n}{k}}}\prod_{i\in {S}}\theta_i^{\alpha-1} d\theta_i  \int \1_{\substack{\forall i \in S^c\\ \theta_i < \frac{\rho_n}{k}}}\prod_{i\in S^c} \theta_i^{\alpha-1} d\theta_i \\
&\leq    \frac{ \Gamma ( k\alpha )}{  \Gamma(\alpha)^k } \sum_{S\subset \{ 1,.., k \}}  \left( \int_{\sum_{i\in S} ( u_i - \bar \theta_{i,k}^{1/2} )^2  \leq  s_n^2  }\1_{\substack{\forall i \in S \\ u_i \geq (\frac{\rho_n}{k})^{\frac{1}{2}}}}\prod_{i\in S}u_i^{2\alpha-1} du_i \right)\left( \frac{\rho_n }{k } \right)^{|S^c|\alpha } \alpha^{-|S^c|}\\
&\leq \frac{ \Gamma ( k\alpha )}{  \Gamma(\alpha)^k }\sum_{l \geq k( 1 - s_n^2) }   \left( \frac{\rho_n }{k } \right)^{(k-l)\alpha } \alpha^{-(k-l)}\left( \frac{\rho_n }{k } \right)^{l(\alpha-1/2) }\frac{ \sqrt{\pi}^l  s_n^l }{ \Gamma(l/2+1) } \binom{k}{l}\\
&\leq \frac{ \Gamma ( k\alpha )}{  (\alpha\Gamma(\alpha))^k }(\rho_n/k)^{k \alpha} \sum_{l \geq k (1 -  s_n^2) }^k \alpha^l e^{ l \log \left( \frac{ k^{1/2} C  s_n }{ \sqrt{ l \rho_n} } \right) + k\log k - l\log l - (k-l)\log (k-l) + O(k)} \\
&\leq \exp\left\{ k \alpha \log (\rho_n) + k \log ( s_n/\sqrt{\rho_n}) + O(k) \right\} \leq e^{ k \log s_n - k (1/2-\alpha) \log \rho_n+ O(k)  } . 
\end{split}
\end{equation*}
Hence, choosing $|\log\rho_n| = o(|\log  s_n|) $ leads to 
\begin{equation*}
\Pi\left( \sum_{j=1}^k \left( \sqrt{\theta} - \sqrt{\bar \theta_{j,k}} \right)^2 \leq  s_n^2  \right)\leq e^{ k ( 1 + o(1) ) \log  s_n},
\end{equation*}
so that $ s_n^2 |\log  s_n| \geq k/n$ and $ s_n^2 \gtrsim k/n\log (n/k) $.

\end{appendix}
\section*{Acknowledgements}
The authors would like to thank the associate editor and the referees for their useful comments which lead to an improved version of the manuscript. 


\bibliographystyle{apalike}
\bibliography{biblio}

\begin{supplement}[id=suppA]
  \stitle{Asymptotic behaviour of the empirical Bayes posteriors associated to maximum marginal likelihood estimator: supplementary material}
 \sdescription{This is the supplementary material associated to the paper \citet{rousseau:szabo:2015}. We provide here the proofs of Propositions 3.1-3.6, together with some technical Lemmas used in the context of priors (T2) and (T3) and some technical Lemmas used in the study of the hierarchical Bayes posteriors. Finally some Lemmas used in the regression and density estimation problems are given. 
}
\end{supplement}

\section{Proof of the Propositions}

\subsection{Proof of Proposition 3.1}
It is sufficient to prove that all conditions of
 Theorems 2.1, 2.2, 2.3, 
  and Corollary 2.1 
  hold, since then the Proposition follows from the combination of them with
     Lemmas 3.3 and 3.5. 

As a first step we note that since there are only finite many truncation parameters ($|\Lambda_n|=o(n)$)  there is no need to introduce a change of measures $\psi_{k,k'}$, one can simply take $q_{k,n}^{\theta}=p_\theta^n$. Furthermore, we also have from $N_n(\Lambda_n) = o(n) $ and $n\epsilon_{n,0}^2 \geq m_n \log n$ that 
$\log N_n(\Lambda_n) \lesssim\log n=o(n \epsilon_{n,0}^2)$.

Next we define for all $k \leq \epsilon n/\log n$,  with $\epsilon >0$ fixed but arbitrarily small, the set 
 $ \Theta_n( k) = \{ \theta \in \RR^k; \max_j |\theta_j| \leq (M_n^2n \epsilon_{n,0}^2)^{1/p^*}\}$, so that the exponential moment condition on $g$ implies that 
  $$ \Pi(\Theta_n(k)^c|k) \lesssim k e^{- w_n^2n \epsilon_{n,0}^2}, \quad \mbox{ if } w_n^2 \leq s_0M_n^2,$$
and condition (2.6) 
 holds. Furthermore, following from Lemma 3.1 
 and
\begin{align*}
\log N(\zeta\eps_n(k),\Theta_n(k),\|\cdot\|_2)\lesssim k\log n,
\end{align*}
for every $\zeta\in(0,1)$, there exists a large enough constant $c(k)=K$ such that the entropy is bounded from above by $c(k)^2n\eps_n(k)^2/4$. We note that by slicing up the set $\Theta_n(k)$, see for instance the proof of Proposition 3.3, the upper bound on the entropy would hold for any $c(k)=K>0$. 

From  \cite{arbel} we have that 
$$2K(\theta_0, \theta) = V_2(\theta_0, \theta) = n \|f_{\theta_0}-f_\theta\|_2^2 = n\|\theta-\theta_0 \|_2^2$$
so that (B1) holds with $M_2=1$ and $\tilde \Lambda_0 = \{k_n\}$ where $k_n \in\{ \epsilon_n(k) \leq M_1 \epsilon_{n,0}\}$.
Then conditions (A2), $(C1)-(C3)$  follow from \cite{ghosal:vdv:07} with $d_n(f_\theta,f_{\theta_0})=\|f_\theta-f_{\theta_0}\|_n $ the empirical $L_2$-distance, which is also equal to the $\ell_2$ norm $\|\theta-\theta_0 \|_2=\|f_\theta-f_{\theta_0}\|_2$ (from Parseval inequality).
Finally condition (2.12) is proved in Lemma \ref{Lem: Cond_LB} and (H2) in Lemma \ref{lem:loglike}.

\subsection{Proof of Proposition 3.2} 
Similarly to Proposition  3.1 
it is sufficient to verify that all the conditions of Theorems 2.1, 2.2, 2.3,
and Corollary 2.1  hold.

Take $u_n\lesssim n^{-3}/\log n$ for $\lambda=\a$ and $u_n\lesssim n^{-(5/2+2\a)}$ for $\lambda=\tau$. Since $ n\epsilon_{n,0}^2 \geq m_n\log n$ and $N_n(\Lambda_n) \leq n^H$ for some $H>0$, $\log N_n(\Lambda_n)  =o(n \epsilon_{n,0}^2)$. Furthermore condition (B1) follows from Proposition 1 of \cite{arbel} with $M_2=1$. 

 The proof of conditions (A1) and (A2) are given in Lemma \ref{Lem: Entropy}, Lemma \ref{Lem: Test_reg}, and Lemma \ref{Lem: LikeRatioUB_reg} with $c_1=1/2,\zeta=1/18,c(\lambda)^2=K^2\geq 10\mu/c_1$ (where $\mu$ is defined in Lemma \ref{Lem: LikeRatioUB_reg}), and $d(\theta_1,\theta_2)=\|\theta_1-\theta_2\|_2$. Condition (H2) holds following from Lemma \ref{lem:loglike} with $c_3=2+3\sigma^{-2}K^2/2$. Finally for Corollary 2.1 conditions (C1)-(C2) follow again from the preceding lemmas with $M>10^{3/2}\sqrt{\mu}$, $c_2=\mu$, since $ w_n\eps_{n,0}=o\big(\eps_n(\lambda)\big)$ for all $\lambda\in\Lambda_n\setminus\Lambda_0$. Note also that from the proof of Lemma \ref{Lem: LikeRatioUB_reg} we also have for $u_n\lesssim n^{-2}$ that $\|\theta-\psi_{\lambda,\lambda'}\|_2=o(n^{-1})=o(\eps_{n,0})$, for every $\|\theta-\theta_0\|=O(1)$.

The lower bound in the case $\a+1/2\leq\beta$ follows from Theorem 2.2 and Lemma 3.4, since condition (2.12) is proved in Lemmas \ref{Lem: Cond_LB} and \ref{Lem: LikeRatioUB_reg}. 

Finally we note that the same results hold for the Gaussian white noise model as well. The proof can be easily derived from the proof on the regression model, by substituting $e_{j}(t_i)$ by $\delta_0(i-j)$ (where $\delta_0$ is the Dirac-delta measure) in Lemmas \ref{Lem: Test_reg} and \ref{Lem: LikeRatioUB_reg} and taking $\sigma^2=1/n$ (in this case $c_3=2+3K^2/2$). Furthermore one can choose $\zeta=c_1=1/2$ in the testing assumption (A2) by using the likelihood ratio test in the Gaussian white noise model, see for instance Lemma 5 of \cite{ghosal:vdv:07}. 

 
\subsection{Proof of Proposition 3.3} \label{pr:propdensT1} %

The proof
 consists in showing that assumptions (A1), (A2bis), (B1) and (C1)-(C3) are verified.

In the case of prior (T1), there is no need to consider a change of measure since $\Lambda$ is finite, so that $N_n(\Lambda_n)=o(n)$. Then similarly to the proof of Proposition 3.1 
we have that $\log N_n(\Lambda_n) =o(n  \epsilon_{n,0}^2)$.

 We first prove (B1), or more precisely the variation of (B1) given in Remark 2.2. Choose $k_0 \in \Lambda_0$ which verifies $\epsilon_{n,0} \leq \epsilon_n(k_0)  \leq M_1\epsilon_{n,0}$ for some $M_1 \geq 1$. 
 We have for all $k$ and all $\theta \in R^k$ that $\|\theta\|_1 \leq \sqrt{k}\|\theta - \theta_0\|_2  +\| \theta_{0}\|_1$.
Now let $\theta_0 \in \mathcal H_\infty (\beta, L) \cup \mathcal S_\beta(L) $ with $\beta>1/2$, then $\|\theta_0\|_1 < +\infty$, and 
if $k_0 \in \Lambda_0$ satisfies $\epsilon_{n,0} \leq \epsilon_n(k_0) \leq M_1 \epsilon_{n,0}$ for some $M_1\geq1$, then
$$\epsilon_n(k_0) \lesssim\frac{ \sqrt{ k_0\log n}}{\sqrt{n}} \vee k_0^{-\beta}, \quad \sqrt{k_0} \epsilon_n(k_0) = o(1), $$
so  that $$ \left\{\|\theta -\theta_0\|_2 \leq K\epsilon_n(k_0)\right\} \subset  \{\|\theta\|_1\leq M\},$$
if $M$ is large enough. 
Moreover, using Lemma \ref{lem:dist}, for all $M>0$, 
 $$\left\{\|\theta -\theta_0\|_2 \leq K\epsilon_n(k_0)\right\} \cap \{\|\theta\|_1\leq M\} \subset B(\theta_0, M_2 \epsilon_n(k_0) , 2),$$
and (B1) is verified.


We now verify assumption (A1). We have  $q_{k,n}^\theta = f_{\theta}^n$ for all $\theta \in \RR^k$, thus (2.5) 
 is obvious and (2.6) 
 follows from \cite{rivoirard:rousseau:12}, (verification of  condition \textbf{A}), with 
$$\Theta_n(k) = \{ \theta \in \RR^k ; \|\theta\|_2 \leq R_n(k)\}, \quad R_n(k) = R_0 (n\epsilon_n(k)^2)^{1/p^*},$$
for some $R_0>0$ large enough. Similarly the tests in (A2) are the Hellinger tests as in \cite{ghosal:ghosh:vdv:00} so that (2.7) 
is satisfied.

We now study the change of distance condition of the version (A2bis) of condition (A2). Define 
$B_{n,j}(k) = \{ \theta \in \Theta_n(k); \|\theta- \theta_0\|_2 \in ( j \epsilon_n(k), (j+1)\epsilon_n(k))\}$ for $j \geq K$ and let $\theta \in B_{n,j}(k)$. 
Since $\|\theta_0\|_2<+\infty$, $B_{n,j}(k) \neq \emptyset$ only if $j \leq 2 R_n(k) /\epsilon_n(k)$. 
Note  also that $\sqrt{k}\epsilon_n(k) \lesssim k^{-\beta+1/2} \vee k \sqrt{\log n/n}\leq 1$.  For all $j \leq j_0 ( \sqrt{k}\epsilon_n(k) )^{-1}$ with  $j_0 >0$ we have 
$\|\theta -\theta_0\|_1 \leq \sqrt{k}\epsilon_n(k) (j+1) \leq j_0+1$. Using Lemma \ref{lem:dist} in the Appendix, we obtain that 
\begin{equation*}
d(f_0, f_\theta) \geq e^{-c_1 (j_0+ 1)}\|\theta-\theta_0\|_2 \geq e^{-c_1 (j_0+ 1)} j \epsilon_n(k).
\end{equation*}
So that $c(k,j) = e^{-c_1 (j_0+1)} j$. Moreover  using  \cite{rivoirard:rousseau:12}, p. 8
$$d(f_\theta, f_{\theta'}) \leq e^{c_1 \|\theta - \theta'\|_1} \|\theta - \theta'\|_2 \leq e^{c_1\sqrt{k}\|\theta - \theta'\|_2 } \|\theta-\theta'\|_2,$$
so that if $\|\theta - \theta'\|_2 \leq\zeta e^{-2c_1 (j_0+1)} j\epsilon_n(k)$, $d(f_\theta, f_{\theta'}) \leq \zeta j\epsilon_n(k)e^{- c_1 (j_0+1)} $ as soon as $k$ or $j_0$ is large enough. Thus
\begin{equation*}
\begin{split}
\log N( \zeta c(k,j) \epsilon_n(k), B_{n,j}(k) , d(\cdot,\cdot) ) &\leq\log N( \zeta e^{-2c_1 (j_0+1)} j\epsilon_n(k), B_{n,j}(k), \| \cdot \|_2)\\
& \lesssim k = o(n\epsilon_n^2(k)).
\end{split}\end{equation*}
Hence, for $n$ large enough we have for $k\in\Lambda_n\setminus\Lambda_0$ 
\begin{equation} \label{part1}
\sum_{K \leq j \leq j_0/(\sqrt{k}\epsilon_n(k))} e^{-c_1 n c(k,j)^2 \epsilon_n(k)^2 /2 } \leq e^{-  c_1 e^{-c_1 (j_0+1)} n \epsilon_n(k)^2/4}  = o(e^{-nw_n^2\epsilon_{n,0}^2}),
\end{equation}
as soon as $w_n= o(M_n)$.
 Now consider $j > j_0 ( \sqrt{k}\epsilon_n(k) )^{-1}$ and let $\theta \in B_{n,j}(k)$, from equation (16) in the proof of  Lemma 3.1 of \cite{rivoirard:rousseau:12},  
$$d(f_0, f_\theta) \gtrsim \|\theta-\theta_0\|_2 \left( \sqrt{k}\epsilon_n(k) j + |\log(j \epsilon_n(k)) |  \right)^{-1}.$$ 
 For all $j \gtrsim \log (k)/ (\sqrt{k}\epsilon_n(k))$ we have $ \sqrt{k}\epsilon_n(k) j \gtrsim  |\log(j \epsilon_n(k)) |$ and 
 $$d(f_0, f_\theta) \gtrsim \frac{ 1}{ \sqrt{k} },$$
when $n$ is large enough and  we can choose $c(k,j) =ck^{-1/2}\eps_n(k)^{-1}$.  
For all $\theta, \theta'\in B_{n,j}(k)$, using equation (8)  of  \cite{rivoirard:rousseau:12} 
  \begin{equation}\label{upboundHell}
  d(f_\theta, f_{\theta'} ) \lesssim \sqrt{k}\|\theta - \theta'\|_2,
  \end{equation}
  so that there exists $c >0$
\begin{align*}
 \log N( \zeta \frac{ \epsilon_n(k) }{ \sqrt{k} } , B_{n,j}(k) , d(\cdot,\cdot) )
&\leq\log N( \zeta \frac{ \epsilon_n(k) c }{ k } , B_{n,j}(k) , \|\cdot\|_2 )\\
 &\lesssim  k \log ( j k) = o(n/k)
\end{align*}
for all $j \leq R_n(k) /\epsilon_n(k) $ and  for all $C_1, C_2>0$
\begin{equation}\label{ineq:B1bis}
\sum_{j =\lceil C_1 \log k/ (\sqrt{k}\epsilon_n(k)) \rceil }^{ \lfloor  R_n(k) /\epsilon_n(k) \rfloor }e^{- C_2 n/k} \leq e^{-C_2n/k }\frac{R_n(k)}{ \epsilon_n(k) } \leq
 e^{-\frac{C_2n}{ 2 k } } , \quad n \mbox{ large enough} . 
 \end{equation}
Combining \eqref{ineq:B1bis} with
 $$ n^{1/(2\beta + 1)}(\log n)^{2\beta / (2\beta+1)}\gtrsim n\epsilon_{n,0}^2 , \quad \text{and} \quad \frac{ n }{ k} \geq \sqrt{n}, \quad \forall k\lesssim \sqrt{n},$$
implies that 
\begin{equation}\label{part2}
\sum_{j =\lceil C_1 \log k/ (\sqrt{k}\epsilon_n(k)) \rceil }^{ \lfloor  R_n(k) /\epsilon_n(k) \rfloor }e^{- C_2 n/k}  = o(e^{- n w_n^2 \epsilon_{n,0}^2} ), 
\end{equation}
when $w_n^2 = o(n^{(\beta-1/2-\delta)/(2\beta+1)})$ with $0 < \delta < \beta -1/2$. 

We now consider $j_0 /\{\sqrt{k} \epsilon_n(k) \} \leq j \leq \delta \log (k)/ \{\sqrt{k}\epsilon_n(k)\} $ with $\delta$ arbitrarily small. Then 
$$d(f_0, f_\theta) \gtrsim \|\theta-\theta_0\|_2 \left(  |\log(j \epsilon_n(k)) |  \right)^{-1} \gtrsim \frac{  j \epsilon_n(k)  }{ \log n} $$ 
so that $c(k,j) \gtrsim j /\log n$. Note also that, similarly to before,  this implies that $d(f_0, f_\theta) > \delta_n \epsilon_{n,0}$ as soon as $k \leq k_0 (n/\log n)^{1/(2\beta+1)}$ for all $k_0>0$ and $n$ large enough.  Using \eqref{upboundHell},
$ \log N( \zeta \frac{ \eps_n(k)}{ \sqrt{k} } , B_{n,j}(k) , d(\cdot,\cdot) )\lesssim k \log (k) $. Moreover 
$$n c(k,j)^2 \epsilon_n(k)^2 \gtrsim \frac{ n \epsilon_n(k)^2 j^2 }{(\log n)^2} \gtrsim \frac{ nj_0^2 }{(\log n)^2 k } $$
for all $j \geq j_0 /(\sqrt{k}\epsilon_n(k))$. By choosing $j_0$ large enough, we thus have that for $k_0$ fixed and all $k\leq k_0 \sqrt{n} (\log n)^{-3}$, 
$$\log  N( \zeta \frac{ \eps_n(k)}{ \sqrt{k} } , B_{n,j}(k) , d(\cdot,\cdot) )\leq c_1 n c(k,j)^2 \epsilon_n(k)^2/2.$$
We also have that 
\begin{equation}\label{part3}
\sum_{j = \lceil j_0(\sqrt{k}\epsilon_n(k))^{-1}\rceil}^{ \lceil C_1 \log k/ (\sqrt{k}\epsilon_n(k)) \rceil +1 } e^{-\frac{C_2n}{ 2 k } } \leq o(e^{- n w_n^2 \epsilon_{n,0}^2} ).
\end{equation}
Combining \eqref{part1}, \eqref{part2} and \eqref{part3}, we finally prove (A2bis).
 
We now verify conditions (C1)-(C3) to obtain the posterior concentration rate. We already know from Lemma 3.1 that $\epsilon_{n,0} \lesssim (n/\log n)^{-\beta/(2\beta+1)}$ where the constant depends only on $L, \beta_1$, and $\beta_2$ if $\theta_0 \in \mathcal S_\beta (L )$ .
Since we do not need the change of measures $\psi_{\lambda, \lambda'}$,  (C1) and (C2) are proved in \cite{rivoirard:rousseau:12}. 

Finally for the lower bound on the contraction rate condition (2.12), take $\theta_{0i}^2 = i^{-2\beta-1} \in \mathcal H(\beta, L)$, so that $\epsilon_{n,0} \asymp (n/\log n)^{-\beta /(2\beta+1)} $ and if $\|\theta - \theta_0\|_2 \leq M \delta \epsilon_{n,0}$ with $\theta \in \mathbb R^k$ and $M>0$ then $k \gtrsim \delta_n^{-1/\beta} ( n/\log n)^{1/(2\beta+1)}$. The above computations imply also that if there exists $k \lesssim \delta_n^{-1/\beta} ( n/\log n)^{1/(2\beta+1)}$ then $d(f_0, f_\theta) \gtrsim \|\theta - \theta_0\|_2  $ on the set $\{d(f_0, f_\theta)\leq \delta_n \epsilon_{n,0}\}$  so that Lemma 3.1 implies condition (2.12). 

For the hierarchical Bayes result assumption (H2) is verified in Lemma \ref{lem:KL:dens}.

\subsection{Proof of Proposition 3.4}\label{sec:prT2dens} 

To prove Proposition 3.4 
we need to verify that (A1)-(A2) and (B1) are satisfied, together with (C1)-(C3), (2.12) and (H2). 
Let $\tau_0 \in \Lambda_0$ satisfying $M_1 \epsilon_{n,0} \geq\epsilon_n(\tau_0) $. 
Equation \eqref{L1:Kn} in Lemma \ref{L1norm}, with $K_n  = \lfloor ( \tau_0^2 n \epsilon_{n,0}^2)^{1/(2\alpha)}\rfloor $ implies that for $M\geq\|\theta_0\|_1$,
$$\Pi\left( \|\theta-\theta_0\|_2\leq K \epsilon_n(\tau_0) ; \|\theta-\theta_0\|_1 \leq K\sqrt{K_n}\epsilon_{n}(\tau_0) + 2M | \alpha, \tau_0\right)\gtrsim e^{-nM^2 \epsilon_{n,0}^2/2}.$$
Moreover $\sqrt{K_n}\epsilon_{n}(\tau_0) \lesssim ( \tau_0^2 n \epsilon_{n,0}^2)^{1/(4\alpha)} \epsilon_{n}(\tau_0 )$ and using Lemmas 3.2 
and 3.3 
if $\beta \leq \alpha+1/2$,
\begin{equation*}
\begin{split}
\sqrt{K_n}\epsilon_{n}(\tau_0)& \lesssim n^{1/[(2\beta+1)4\alpha] -\beta/(2\beta+1)} \tau_0^{1/(2\alpha)}, \quad \tau_0 \lesssim n^{-\frac{\beta(2\alpha+1)}{2\beta+1}+ \alpha}\\
\sqrt{K_n}\epsilon_{n}(\tau_0)& \lesssim n^{1/[(2\beta+1)4\alpha] -\beta/(2\beta+1)+ 1/2- \beta(2\alpha+1)/[2\alpha(2\beta+1)]} \\
&= n^{-\frac{(2\beta-1)(2\alpha+1)}{4\alpha(2\beta+1)}} = o(1).
\end{split}
\end{equation*}
Similarly if $\beta > \alpha+1/2$,
\begin{equation*}
\begin{split}
\sqrt{K_n}\epsilon_{n}(\tau_0)& \lesssim n^{1/[(\alpha+1)8\alpha] -(2\alpha+1)/(4\alpha+4)} \tau_0^{1/(2\alpha)}, \quad \tau_0 \lesssim n^{-\frac{(2\alpha+1)^2}{4\alpha+4}+ \alpha},\\
\sqrt{K_n}\epsilon_{n}(\tau_0)& \lesssim n^{- \frac{ 2 \alpha+1}{4(\alpha+1)}}=o(1).
\end{split}
\end{equation*}
So that for $n$ large enough,
 $$\Pi\left( \|\theta-\theta_0\|_2\leq K \epsilon_n(\tau_0) ; \|\theta-\theta_0\|_1 \leq  3M | \alpha, \tau_0\right)\gtrsim e^{-nM^2 \epsilon_{n,0}^2/2}$$
and using the same computations as in the verification of (B1) in Section \ref{pr:propdensT1} we obtain 
$$\left\{\|\theta -\theta_0\|_2 \leq K\epsilon_n(\tau_0)\right\} \cap \{\|\theta\|_1\leq M\} \subset B(\theta_0, M_2 \epsilon_n(\tau_0) , 2).$$
We now verify (A1), (A2), (C1)-(C3). Consider the transformation defined in (3.11). 
Then 
  \begin{equation} \label{like:up}
 \begin{split}
 \log f_{\psi_{\tau, \tau'}(\theta)} &= \frac{ \tau' }{ \tau } \left( \sum_{i} \theta_i \phi_i\right) - \log \left( \int_0^1 e^{\tau'/\tau  \sum_i\theta_i\phi_i (x)}dx \right)\\
 &\leq 
  \frac{ \tau' }{ \tau } \left( \sum_{i} \theta_i \phi_i - c(\theta) \right),
 \end{split}
 \end{equation}
 if $\tau'\geq \tau$. 
 $$\Theta_n(\tau) = \{ \theta =  \epsilon_n (\tau ) \theta_1 +R_n(\tau)  \theta_2, \quad \theta_1 \in  \mathbb B_1, \, \theta_2 \in  \mathbb H_1^\tau \} \cap \{ \|\theta_1\|_1 \leq \sqrt{n} \},$$
 with $R_n(\tau), \mathbb B_1, \mathbb H_1^\tau$ defined in Lemma \ref{Lem: Entropy}.  Lemma \ref{Lem: Entropy} implies that 
\begin{equation} \label{tildeThetan}
\Pi( \Theta_n(\tau)^c | \alpha, \tau ) \leq e^{- c n\epsilon_n^2(\tau) }. 
\end{equation}
 For all $\theta \in \Theta_n(\tau)$, 
$$\|\theta\|_1 \leq \epsilon_n(\tau) \|\theta_1\|_1 + R_n(\tau) \|\theta_2\|_1 \leq \sqrt{n}  \epsilon_n(\tau) + R_n(\tau) {\tau(\|\theta_2\|_{\mathbb H^\tau}+\alpha^{-1})} \leq C(\tau)\sqrt{ n} \epsilon_n(\tau), $$
so that if 
 $\t\leq\tau' \leq \tau (1 + u_n)$ with $u_n = o( n^{-3/2}\epsilon_n(\tau)^{-1} )$  
 \begin{equation}\label{upbound:qtau}
 q^\theta_{\tau,n} (\mathbf x_n ) \leq f_\theta^n(\mathbf x_n) e^{2 nu_n \|\sum_i  \theta_i\phi\|_\infty} \leq (1 + o(1) ) f_\theta^n(\mathbf x_n)
 \end{equation}
and 
$Q^\theta_{n,\tau}(\mathcal X^n)  \leq 2  $
for $n$ large enough, 
and condition (2.5) 
 in  (A1) is satisfied with $\tau_i=\underline\tau_n(1+u_n)^{i}$ the smallest point in the $(i+1)$th bin $[\underline\tau_n(1+u_n)^{i}, \underline\tau_n(1+u_n)^{i+1}]$ on $\Lambda_n$. Using \eqref{upbound:qtau}, we also have that 
\begin{equation*}
\begin{split}
\int_{ \Theta_n(\tau)^c}&q^\theta_{\tau,n}(\mathcal X^n) d\Pi(\theta|\tau) \leq \int_{{\Theta_n(\tau)^c}}e^{2 nu_n \|\sum_i  \theta_i \phi\|_\infty}d\Pi(\theta|\tau) \\
& \leq \Pi( \Theta_n(\tau)^c | \alpha, \tau )^{1/2} \left(  \int e^{2 nu_n \|\sum_i  \theta_i \phi\|_\infty}d\Pi(\theta|\tau) \right)^{1/2} \\
&\leq {e^{- c n\epsilon_n^2(\tau)/2 }\prod_{i}   2e^{n^2u_n^2\|\phi\|_\infty^2 \tau^2 i^{-2\alpha-1}}}\\
&\leq e^{2n^2u_n^2\|\phi\|_\infty^2 \tau^2\sum_i i^{-2\alpha-1} }e^{- c n\epsilon_n^2(\tau) /2 } \leq 2e^{- c n\epsilon_n^2(\tau) /2} ,
\end{split}
\end{equation*}
if $u_n = o( n^{-1}\tau^{-1})$,  and condition (2.6) 
is verified.

   Similarly to the case of prior (T1) the tests in condition (A2) are the Hellinger tests as in \cite{ghosal:ghosh:vdv:00} so that (2.7) 
   is satisfied using \eqref{upbound:qtau}. We now verify (2.8). 
    Recall that for all $\theta_0 \in \mathcal S_\beta(L) \cup \mathcal H_{\infty}(\beta, L)$ with $L>0$ and $\beta >1/2$, 
  $\|\theta_0\|_1 < +\infty$.     From Lemma \ref{lem:dist},
    $$ d(f_0, f_\theta) \gtrsim \|\theta - \theta_0\|_2 e^{- C \|\theta-\theta_0\|_1}.$$
    Define $\tilde \Theta_n(\tau) = \Theta_n(\tau) \cap \{\|\theta - \theta_0\|_2 > K \epsilon_n(\tau) ; \|\theta - \theta_0\|_1 \leq \sqrt{K_n} \|\theta -\theta_0\|_2 + M\} \cap \bar \Theta_n(\tau) $ with $K_n^{2\alpha} = n\tau^2 w_n^2\epsilon_{n,0}^2$, $M\geq 2 \|\theta_0\|_1$ and 
    $$ \bar \Theta_n( \tau ) = u_n \mathbb B_1^{L_1} + \bar R_n \mathbb H_1$$
    with $u_n = \tau C_\alpha^{-(\a -1/2)}(n\epsilon_{n,0}^2w_n)^{-(\a - 1/2)}\wedge M$ and   $\mathbb B_1^{L_1} $ is the $L_1$ unit ball and  $\bar R_n \asymp n\epsilon_{n,0}^2w_n^{1/2}$ if $\tau C_\alpha^{-(\a -1/2)}(n\epsilon_{n,0}^2w_n)^{-(\a - 1/2)} \leq M$ else $\bar R_n \asymp C_\alpha^{1/2}(\frac{M}{\tau})^{-\frac{1}{2(\alpha - 1/2)}}$. 
         From \eqref{tildeThetan}, \eqref{priormassL1} and Lemma \ref{L1norm}
    \begin{equation}\label{tildeThetan2}
    \Pi( \tilde \Theta_n^c |\tau ) \leq e^{-cn\epsilon_n(\tau)^2 } + e^{ - A n w_n \epsilon_{n,0}^2 } 
    \end{equation}
    where $ A$ can be chosen as large as need be by choosing $C_\alpha $ large enough.
    If  $\theta \in\tilde \Theta_n(\tau)$ with $\|\theta-\theta_0\|_2 \leq MK_n^{-1/2} $ then from Lemma \ref{lem:dist} $ d(f_0, f_\theta) \gtrsim \|\theta - \theta_0\|_2 e^{-2CM}$; so that $c(\tau) \geq e^{-2CM}$. 
 Now let  $\|\theta -\theta_0\|_2 > M K_n^{-1/2}$. Note that, from Lemma 3.2, 
 $$\epsilon_n(\tau) \leq n^{-\alpha/(2\alpha+1)}\tau^{1/(2\alpha+1)} +  \left( \frac{ C }{ n\tau^2 } \right)^{ \frac{\beta }{2\alpha+1} \wedge  \frac{1}{ 2 } } $$
 and $\epsilon_{n,0} \leq n^{-(2\alpha+1)/(4\alpha+4)}$ if $\beta > \alpha +1/2$ and else  $\epsilon_{n,0}\leq n^{-\beta/(2\beta+1)}$. This implies that   for all {$\tau \leq\bar\t_n= n^{\a/2-1/4}$}
 \begin{equation}\label{born:tau}
\epsilon_n(\tau) ( n \tau^2 w_n^2 \epsilon_{n,0}^2 )^{1/(4\alpha)} < M (\log n)^{-2},
\end{equation}
which combined with $\|\theta -\theta_0\|_2 > M K_n^{-1/2}$, leads to $\|\theta -\theta_0\|_2 > (\log n)^2 \epsilon_n(\tau)$.   Theorem 5.1 of \cite{wong:shen:1995} implies that either $d(f_0, f_\theta) \geq 1- 1/e$ or 
 $$V_{2}(f_0, f_\theta) \leq  Cd^2(f_0, f_\theta) \left( 1 + (\log n )^2  + \|\theta\|_1^2\right).$$
 Moreover, since $f_0 \geq c_0$, 
 \begin{equation*}
 \begin{split}
 V_{2}(f_0, f_\theta) &\geq c_0 \int_0^1 \left( \sum_{j\geq 2} (\theta_j - \theta_{0,j} ) (\phi_j - c_j(f_0)) \right)^2 dx \\
 &= c_0 \left( \|\theta - \theta_0\|_2 + \left( \sum_j (\theta_j - \theta_{0,j} )c_j(f_0) \right)^2 \right) \geq c_0  \|\theta - \theta_0\|_2^2 
 \end{split}
 \end{equation*}
 and 
 $$d(f_0, f_\theta) \gtrsim \frac{ \|\theta-\theta_0\|_2 }{ 1 + \log n + \|\theta -\theta_0\|_2 \sqrt{K_n}} \gtrsim K_n^{-1/2}/\log n $$
so that  $d(f_0, f_\theta) \gtrsim  \epsilon_n(\tau) \log n$ and (2.8) 
 is verified with $c(\tau ) >0 $. To verify condition (2.9), we need to control the Hellinger entropy. Over the subset of $\tilde \Theta_n(\alpha,\tau)$ defined by $\|\theta - \theta_0\|_2\leq M/\sqrt{K_n} $, $\| \theta - \theta_0 \|_1 \leq M$ and lemmas \ref{lem:dist} and \ref{Lem: Entropy} imply that the Hellinger entropy of this set is bounded by the $L_2$ entropy which is bounded by $ C  n\eps_n^2(\alpha, \tau)$. If $\| \theta - \theta_0\|_2 > M/\sqrt{K_n} $, the above computations imply that $d(f_0, f_\theta)\gtrsim MK_n^{-1/2}(\log n)^{-1}$. Moreover, if $\theta \in \tilde \Theta_n(\alpha, \tau  ) $ and $\| \theta - \theta'\|_2 \leq \epsilon_n(\alpha , \tau ) $, then for all $J>0$ 
 \begin{equation}\label{upb1}
 \sum_{j\leq J}|\theta_j - \theta_j' | \leq \sqrt{J} \|\theta - \theta'\|_2 \leq \sqrt{J}\epsilon_n(\a, \t). 
 \end{equation}
 Choose $J \asymp \epsilon_n(\a , \t)^{-2} $. Then, since $\theta  = \theta_1+ \theta_2$ and $\theta' = \theta_1'+\theta_2'$ with
  $\theta_1, \theta_1' \in u_n \mathbb B_1^{L_1}$ and $\theta_2, \theta_2' \in \bar R_n \mathbb H_1$, 
\begin{equation*}
\sum_{j \geq J+1} |\theta_j - \theta_j'| \leq 2 u_n + 2 \alpha^{-1/2}\bar R_n J^{-\alpha}\lesssim u_n+ \sqrt{\epsilon_n(\a , \t)^{4\a}n\epsilon_{n,0}^2 w_n}, 
\end{equation*}
if $\tau C_\alpha^{-(\a -1/2)}(n\epsilon_{n,0}^2w_n)^{-(\a - 1/2)} \leq M$ or 
\begin{equation*}
\sum_{j \geq J+1} |\theta_j - \theta_j'| \lesssim u_n+ \sqrt{\epsilon_n(\a , \t)^{4\a} \tau^{1/(\alpha-1/2)}},
\end{equation*}
 if $\tau C_\alpha^{-(\a -1/2)}(n\epsilon_{n,0}^2w_n)^{-(\a - 1/2)} > M$. So that $\|\theta - \theta'\|_1 =O(1)$ as soon as $\epsilon_n(\a , \t)^{4\a}n\epsilon_{n,0}^2 w_n = O(1)$. In the case $\beta \leq \alpha +1/2$ and $\tau = o( n^{(\a- 1/2)/(2\beta+1)})$,
\begin{equation} \label{cond:tau}
 n^{-4 \a^2/(2\a +1)} \tau^{4 \a /(2\a +1)} = o(n^{-1/(2\beta + 1)}), \quad (n\tau^2)^{-4\a\beta/(2\a +1)}=o(n^{-1/(2\beta + 1)}), 
\end{equation}
the former relation is satisfied when $\tau \lesssim n^{\alpha/2- 1/4}$ while the latter requires $\tau\gg n^{-1/4 + 1/(8\alpha)}$. In the case  $\tau \gtrsim  n^{(\a- 1/2)/(2\beta+1)}$, \eqref{cond:tau} is replaced with $ n^{-4 \a^2/(2\a +1)} \tau^{4 \a /(2\a +1)} \tau^{1/(\alpha-1/2)} = o(1)$, which is satisfied as soon as $\tau \leq n^{\alpha/2- 1/4}$. In the case $\beta > \alpha +1/2$ the same results hold.

Conditions (C1)-(C3) are direct consequences of  the transformation 
(3.11) which  in turns implies \eqref{upbound:qtau},  combined with the definition of $ \Theta_n(\tau)$ so that (C1) and (C2) hold.

Finally $N_n(\Lambda_n) $ is at most polynomial in $n$ so that $\log N_n(\Lambda_n) = o(n \epsilon_{n,0}^2)$. This terminates the proof of the upper bound on the contraction rate of the MMLE empirical Bayes posterior.
Then the lower bound in case $\beta>\a+1/2$ and $\|\theta \|_2 >c>0$ follows from the combination of Theorem 2.2 
 and Lemmas \ref{Lem: Cond_LB} and 3.3 together with the fact that when $\theta \in \tilde \Theta_n(\tau)  $, either $d(f_0, f_\theta) \gtrsim \|\theta - \theta_0\|_2 $ or $d(f_0, f_\theta) \gtrsim K_n^{-1/2} \asymp n^{-1/(4\alpha( \a+1))}\tau^{-1/\a} \gtrsim \epsilon_{n,0}$.

Finally for the hierarchical Bayes result we note that the transformation 
(3.11) implies, as in \eqref{upbound:qtau} that 
\begin{equation*}
\ell_n(\psi_{\tau,\tau'}(\theta)) \geq -2 nu_n\| \phi\|_\infty \sum_i  |\theta_i |   + \ell_n(\theta), \quad \mbox{ if } \tau' \leq \tau, 
\end{equation*}
which combined with Lemma \ref{lem:KL:dens} and the definition of $\epsilon_{n,0}$ implies (H2) as soon as $M_2n \epsilon_{n,0}^2 \geq M_0 \log n$ with $M_0$ large enough and $u_n \lesssim 1/n$. Then the statement is a direct consequence of Theorem 2.3.

\subsection{Proof of Proposition 3.5}\label{sec:prT3dens} 

As in the proof of Proposition 3.4, 
Let $\alpha_0$ replacing $\tau_0$ in the verification of (B1) in Section \ref{sec:prT2dens}.
Equation \eqref{L1:Kn} in Lemma \ref{L1norm}, with $K_n  = \lfloor (  n \epsilon_{n,0}^2)^{1/(2\alpha_0)}\rfloor $  implies that for $M>\|\theta_0\|_1$, 
$$\Pi\left( \|\theta-\theta_0\|_{2}\leq K \epsilon_n(\alpha_0) ; \|\theta-\theta_0\|_1 \leq K\sqrt{K_n}\epsilon_{n}(\alpha_0) + 2M \Big| \alpha_0\right) \gtrsim e^{-\frac{M^2   n \epsilon_{n,0}^2}{2}}$$
Moreover $\sqrt{K_n}\epsilon_{n}(\alpha_0) \lesssim ( n \epsilon_{n,0}^2)^{1/(4\alpha_0)} \epsilon_{n}(\alpha_0 )$ and by using Lemmas 3.2 and 3.3. 
(i.e. $n^{-\alpha_0/(2\alpha_0+1)} \lesssim \epsilon_{n}(\alpha_0) \lesssim (n/\log n)^{-\beta/(2\beta+1)}$, $\alpha_0 \geq \beta $) we have that
\begin{align*}
\sqrt{K_n}\epsilon_{n}(\tau_0) \lesssim (n/\log n)^{1/[(2\beta+1)4\alpha_0] -\beta/(2\beta+1)} =o(1)\quad\text{and}\\
\Pi\left( \|\theta-\theta_0\|_{2}\leq K \epsilon_n(\alpha_0) ; \|\theta-\theta_0\|_1 \leq3M | \alpha_0\right)\gtrsim e^{-M^2   n \epsilon_{n,0}^2/2}.
\end{align*}
As in the case of type (T2) prior, 
$$ \{\|\theta-\theta_0\|_{2}\leq K \epsilon_n(\alpha_0) ; \|\theta-\theta_0\|_1 \leq3M\} \subset B( \theta_0, M_2 \epsilon_n(\alpha_0), 2), $$
for some constant $M_2>0$. 

To study conditions (A1), (A2) and (C1)-(C3), recall that the change of variable $\psi_{\alpha, \alpha'}(\theta)$ is defined by (3.12) 
so that when $\alpha'\geq \alpha$
\begin{equation*}
 \begin{split}
 \log f_{\psi_{\alpha, \alpha'}(\theta)}(x) -  \log f_{\theta}(x) &= \sum_i( i^{\alpha - \alpha'}-1) \theta_i \phi_i \\
  & \quad - \log \Big\{ \int_0^1f_\theta(x) \exp\Big( \sum_i( i^{\alpha - \alpha'}-1) \theta_i \phi_i (x) \Big) dx\Big\}  \\
 & \leq 2|\alpha - \alpha'| \|\theta\|_1\|\phi\|_\infty .
 \end{split}
 \end{equation*}
Let $\alpha \in \Lambda_n\setminus\Lambda_0$ and define 
$$\Theta_n(\alpha) = \{ \theta = A \epsilon_n (\alpha ) \theta_1 +R_n(\alpha)  \theta_2, \quad \theta_1 \in  \mathbb B_1, \, \theta_2 \in  \mathbb H_1^\alpha \} \cap \{ \|\theta_1\|_1 \leq \sqrt{n} \}.$$
  For all $\theta \in \Theta_n(\alpha)$, 
 $$\|\theta\|_1 \leq \epsilon_n(\alpha) \|\theta_1\|_1 + R_n(\alpha) \|\theta_2\|_1 \leq \sqrt{n} \epsilon_n(\alpha) + R_n(\alpha) \|\theta_2\|_{\mathbb H^\alpha}\alpha^{-1} \leq C\sqrt{ n} \epsilon_n(\alpha), $$
 for some constant $C$ independent of $\alpha$.
Let $\theta \in \{\|\theta - \theta_0\|_2 \leq K \epsilon_n(\alpha)\} \cap \Theta_n(\alpha) $.
Then for all $\alpha \leq \alpha' \leq \alpha + u_n $
$$q^\theta_{\alpha,n}(\mathbf x_n) \leq e^{2Cn^{3/2}\epsilon_n(\alpha)\|\phi\|_\infty u_n } f_{\theta}^n(\mathbf x_n) $$
so that (2.5) 
 is verified on $\{\|\theta - \theta_0\|_2 \leq K \epsilon_n(\alpha)\} \cap \Theta_n(\alpha)$ as soon as $u_n \leq a n^{-3/2} \epsilon_n(\alpha)^{-1}$, $ \forall a >0$ when $n$ is large enough. 
To prove (2.6), 
we decompose $\Theta_n(\alpha)^c$ into 
 $\Theta_{n,j} = \Theta_n(\alpha)^c \cap \{ \|\theta\|_1  \in ( j \sqrt{n} \epsilon_n(\alpha), (j+1)\sqrt{n} \epsilon_n(\alpha ) )\}$, $j \geq 0$. We use Lemma \ref{L1norm} with  
$\alpha \geq 1/2 + {n^{-1/6}}$, so that 
 \begin{equation*}
E( \|\theta\|_1 |\alpha) \leq \sqrt{2/\pi} \frac{ 2 \alpha }{ 2 \alpha - 1} 
 \end{equation*}
and from Lemma 3.2, 
$E ( \|\theta\|_1|\alpha ) \lesssim \sqrt{n}\epsilon_n(\alpha) $. Also, for all $j \geq J_1$ with $J_1$ fixed and large enough following from Lemma \ref{L1norm} we have
 $$ 
  \Pi \left(  \|\theta\|_1 > j \sqrt{n} \epsilon_n(\alpha)|\alpha  \right)  \leq e^{ - c_0 j^2 n \epsilon_n^2(\alpha)  },$$
  for some $c_0>0$ independent of $\alpha$.
On $\Theta_{n,j}$ define $u_{n,j} = u_n/(j\log j)$ and construct a covering of $[\alpha , \alpha + u_n]$ with balls of radius $u_{n,j}$, the number of such balls is of order $N_j = O(j \log j)$.  Then since 
 $$ \sup_{|\alpha - \alpha'| \leq u_n} p_{\psi_{\alpha, \alpha'}(\theta) }^n(\mathbf x_n ) \leq \max_{i\leq N_j} \sup_{|\alpha' - \alpha_i|\leq u_{n,j}} p_{\psi_{\alpha_i, \alpha'}(\theta) }^n(\mathbf x_n )$$
 we have that for all $\theta \in B_{n,j}$ (where $B_{n,j}$ was defined in (A2 bis))
 $$Q^\theta_{\alpha,n} (\mathcal X^n ) \leq N_j e^{2a\|\phi\|_\infty/( \log j)} \leq 2 N_j,$$
 if $a$ is chosen small enough in the definition of $u_n$ and 
 $$\int_{\Theta_{n,j}}Q^\theta_{\alpha,n} (\mathcal X^n )  d\Pi( \theta| \alpha ) \lesssim N_j \Pi(\Theta_{n,j}|\alpha) \lesssim j \log j  e^{-c_0 j^2 n \epsilon_n^2 ( \alpha )}.$$

 Let $j \leq J_1$, then $\|\theta\|_1 \leq J_1 \sqrt{n}\epsilon_n(\alpha)$ and 
 $$Q^\theta_{\alpha,n} (\mathcal X^n ) \leq e^{c_0 J_1^2 n \epsilon_n^2 ( \alpha )/2}$$
and, since by choosing $A$ (in the definition of $\Theta_n(\a)$) large enough $\Pi\left( \Theta_n^c(\alpha) |\alpha\right) \leq e^{-c_0 J_1^2 n \epsilon_n^2 ( \alpha )}$, 
$$\int_{\cup_{j\leq J_1}\Theta_{n,j}}Q^{\theta}_{\alpha,n} (\mathcal X^n )  d\Pi( \theta| \alpha ) \leq e^{c_0 J_1^2 n \epsilon_n^2 ( \alpha )/2}
\Pi\left( \Theta_n^c(\alpha) |\alpha\right) \leq e^{-c_0 J_1^2 n \epsilon_n^2 ( \alpha )/2},$$
which implies  (2.6). 
 
Similarly to the case of prior (T2), we verify (A2). The tests are the same as in Section  \ref{sec:prT2dens}, since $q^\theta_{\alpha,n} \lesssim f_\theta^n $ if $u_n \leq a n^{-3/2} \epsilon_n(\alpha)^{-1}$ and the argument follows the same line, with $\epsilon_n(\alpha)$ replacing $\epsilon_n(\tau)$ and $\tilde \Theta_n(\alpha)$ replacing $\tilde \Theta_n(\tau)$ although the definitions remain the same. Note that in this case we do not have to split into $\tau $ large or small in the definition of $\bar \Theta$.  Equation \eqref{born:tau} is satisfied for all $\alpha \in ( 1/2, \log n/(16 \log \log n)]$ so that condition (A2) is verified. 

The verification of (C1)-(C3) follows the same lines as in the case of prior (T2) using the fact that if $\|\theta\|_1\leq M$, and if $u_n\leq n^{-3/2}\epsilon_{n,0}M_n/(M\|\phi\|_\infty)$, for all $\theta\in \Theta_n(\alpha)$ and $\alpha \in \Lambda_0$,
$$\inf_{\alpha \leq \alpha+u_n} \ell_n(\psi_{\alpha,\alpha'}(\theta))-\ell_n(\theta) \geq - 1 $$
and 
$$\sup_{\alpha \leq \alpha+u_n} \ell_n(\psi_{\alpha,\alpha'}(\theta))-\ell_n(\theta) \leq  1 $$
The control over $\Theta_n(\alpha)^c$ is done as before by splitting it into the subsets $\Theta_{n,j}$. Finally similarly to the preceding sections $\log N_n(\Lambda_n)= o(n\eps_{n,0}^2)$ for arbitrarily small $c_2>0$.

\subsection{Proof of Proposition 3.6}
From \cite{castillo:rousseau:suppl}, together with the fact that 
$q_{k,n}^\theta = f_{\theta}^{\otimes n} $ for all $\theta \in \mathcal S_k$ (where $\mathcal S_k$ denotes the $k$ dimensional simplex) and that in $\mathcal S_k$ the set 
 $$ \{u \leq  h(f_0, f_\theta)\leq 2 u \} \subset \Big\{ \sum_{j=1}^k\left( \sqrt{\theta_j}- \sqrt{\bar\theta_{j,k}}\right)^2 \leq 8 u^2 \Big\}$$
 and the covering number of this set with balls of radius $\zeta u$ in Hellinger distance is bounded from above by 
  $(Ck /u\zeta)^k $ so that for all $u \geq  A \sqrt{k/ n} $, condition (2.9)   is verified. Finally since $f_0$ is bounded from above and from below 
  the Kullback-Leiber divergence is bounded by a constant times the square of the Hellinger distance and condition (B1) is verified. Conditions (C1)-(C3) follow from the above arguments and the remark that when $f_0\in \mathcal H_{\infty}(\beta,L)$ then $\Lambda_0 \subset \{k \leq k_1(n/\log n)^{1/(2\beta+1)}\}$ for some $k_1$ large enough, as in the case of prior (T1).

\subsection{Proof of Theorem 2.2}

Similarly to the proof of Corollary 2.1 we can write 
\begin{align*}
E_{\theta_0}^n \Pi(\theta:\, d(\theta,\theta_0)\leq \delta_n \eps_{n,0}|\mathbf x_n, \hat\lambda_n)&=E_{\theta_0}^n\Big(\frac{G_n(\hat\lambda_n)}{m(\mathbf x_n|\hat\lambda_n)} \Big)\\
&\leq e^{c_2n \eps_{n,0}^2} E_{\theta_0}^n  \sup_{\lambda\in \Lambda_0}G_n(\lambda)+o(1),
\end{align*}
where $G_n(\lambda)=\int_{\theta:\, d(\theta,\theta_0)\leq \delta_n\eps_{n,0}}e^{\ell_n(\theta)-\ell_n(\theta_0)}d\Pi(\theta| \lambda)$.

Then similarly to the proof of Theorem 2.1 we take a  $u_n$ covering of the set $\Lambda_0$ with center points $\lambda_1,\lambda_2,...,\lambda_{N(\Lambda_0)}$ and we get (from conditions (A1) and (2.12)) that 
\begin{align*}
E_{\theta_0}^n  \sup_{\lambda\in \Lambda_0}G_n(\lambda)
&=E_{\theta_0}^n  \sup_{\lambda_i}\sup_{\rho(\lambda,\lambda_i)\leq u_n}\int_{d(\psi_{\lambda_i,\lambda}(\theta),\theta_0)\leq \delta_n\eps_{n,0}}e^{\ell_n(\psi_{\lambda_i,\lambda}(\theta))-\ell_n(\theta_0)}d\Pi(\theta|\lambda_i)\\
&\leq\sum_{i=1}^{N(\Lambda_0)}\int_{\{d(\theta,\theta_0)\leq 2\delta_n\eps_{n,0}\}\cup \Theta_n^c}Q_{\lambda,n}^{\theta}(\mathcal{X}^n)d\Pi(\theta|\lambda_i)\\
&\leq N(\Lambda_0) \big(e^{-\tilde{w}_n^2 n\eps_{n,0}^2+e^{-w_n^2  n\eps_{n,0}^2}} \big), 
\end{align*}
for some $\tilde{w}_n\rightarrow\infty$, where in the first inequality we applied that following the (adjusted) condition (C3) and triangle inequality $\{\theta:\, d(\psi_{\lambda_i,\lambda}(\theta),\theta_0)\leq \delta_n\eps_{n,0}\}\subset\{d(\theta,\theta_0)\leq 2\delta_n\eps_{n,0}\}\cup \Theta_n^c$. We conclude the proof by combining the two displays.

\section{Some technical Lemmas for priors (T2) and (T3) } \label{app:lem:GP}

\begin{Lemma}\label{Lem: Entropy}
For every $\a,\tau >0$, and $\zeta\in(0,1)$, take $\eta\geq\tilde{c}_1^2 (3\zeta^{-1}K/c)^{1/\a}$ (with $c=c(\a,\tau)$ and $\tilde{c}_1$ given in (3.4))
 and define the sets
\begin{equation}\label{def: Theta_n}
\Theta_n(\alpha, \tau) = (\zeta c/3)\epsilon_n \mathbb{B}_1+R_n \mathbb{H}^{\a, \tau}_1, \quad  \eps_n = \eps_n(\alpha, \tau),\quad R_n=R_n(\alpha, \tau),
\end{equation}
with 
 \begin{align*}
R_n=-2\Phi^{-1}(e^{-\eta n\eps_n^2}),
\end{align*}
and where  $\mathbb{B}_1\subset \mathbb{R}^n$, respectively $\mathbb{H}_1^{\a,\t}$, denotes the unit ball on the Hilbert space $(\mathbb{R}^n,\|\cdot\|_2)$, respectively the reproducing kernel Hilbert space corresponding to the priors (T2) and (T3). Then 
\begin{equation*}
\begin{split}
\log N(c\zeta\eps_n,\Theta_n(\alpha, \tau),\|\cdot\|_2)\leq 5\eta n\eps_n^2, \\
\Pi(\Theta_n^c(\alpha, \tau)|\alpha, \tau )\leq e^{-\eta n\eps_n^2},\text{and}\\
\|\Theta_n(\a,\t)\|_2^2\leq 2^4\eta\t^2 n\eps_n^2\vee 1.
\end{split}
\end{equation*}
                      
Moreover, if $\a > 1/2$, then for all $u_n/\tau  <1$, 
\begin{equation}\label{priormassL1}
 \log \Pi \left( \|\theta \|_1 \leq u_n| \alpha, \tau \right)  \geq - C_\alpha \left( \frac{ u_n}{\tau} \right)^{-1/(\alpha - 1/2)}
 \end{equation}
$$C_\alpha  \lesssim ((\alpha-1/2)/8\varphi(0))^{-1/(\alpha-1/2)} .$$
Let $\mathbb{B}_1^{L_1}$ denote the  ball of radius 1 centered at 0 for the norm $L_1$ and if 
\begin{equation}\label{def:barTheta_n}
\bar \Theta_n(\alpha, \tau) = u_n \mathbb{B}_1^{L_1}+\bar R_n \mathbb{H}^{\a, \tau}_1, \quad  \bar R_n = -2\Phi^{-1}\left( e^{- C_\alpha \left( \frac{ u_n}{\tau} \right)^{-1/(\alpha - 1/2)}}\right) 
\end{equation}
then 
$$\Pi\left( \bar \Theta_n^c|\alpha , \tau \right) \leq  e^{- C C_\alpha \left( \frac{ u_n}{\tau} \right)^{-1/(\alpha - 1/2)}}$$
for some $C>0$ independent of $\a $ and $\tau$.

\end{Lemma}

\begin{proof}
We follow the lines of the proof of Theorem 2.1 of \cite{vvvz08}.
Define $\gamma_n$  such that
\begin{align*}
\Phi(\gamma_n)&\equiv \Pi(\theta:\,\|\theta\|_2\leq\zeta c\eps_n/3|\a,\t)
\geq e^{ \eta\log \Pi(\theta:\,\|\theta-\theta_0\|_2\leq K\eps_n |\a,\t)}
= e^{-\eta n\eps_n^2},
\end{align*}
where $\Phi(x)$ denotes the distribution function of the standard normal random variable.
Then we can see that $\gamma_n\geq-R_n/2$ and therefore by Borell's inequality we have that 
\begin{align}
\Pi(\Theta_n^c|\a,\tau)\leq 1-\Phi(\gamma_n+R_n)\leq 1-\Phi(R_n/2)=e^{-\eta n\eps_n^2}\label{eq: help_2}.
\end{align}

Then take a $(2\zeta c/3)\eps_n$-separated $h_1,h_2,...,h_N$ points contained in $R_n\mathbb{H}_1$ for the $\|\cdot\|_2$ norm, so the $h_i+(\zeta c/3)\eps_n$-balls are separated. 
Furthermore note that following from the tail bound on the Gaussian distribution function $\Phi(-x)\leq e^{-x^2/2}/(\sqrt{2\pi}x)\leq e^{-x^2/2}$ we have that
\begin{align}
R_n=-2\Phi^{-1}(e^{-\eta n\eps_n^2})\leq \sqrt{8\eta n\eps_n^2}.\label{UBforMn}
\end{align}
Then similarly to \cite{vvvz08} (with $C=\eta$ and $\varphi_{0,\a,\tau}(\zeta c\eps_n/3)\leq \eta n\eps_n^2$) we get that
\begin{align*}
1\geq N e^{-R_n^2/2}e^{-\varphi_{0,\a,\tau}(c\eps_n/6)}\geq Ne^{-5\eta n\eps_n^2}.
\end{align*}
This leads to the inequality
\begin{align}
\log N(\zeta c\eps_n,\Theta_n(\alpha, \tau),\|\cdot\|_2)&\leq \log N(2\zeta c\eps_n/3,R_n\mathbb{H}_1,\|\cdot\|_2)\nonumber\\
&\leq \log N
\leq 5\eta n\eps_n^2\label{eq: help_1}.
\end{align}

Finally we note that
\begin{align}
\|\Theta_n\|_2^2\leq (\t R_n+(K/6)\eps_n)^2\leq 2\t^2 R_n^2\vee 1\leq 2^4 \eta \t^2 n\eps_n^2\vee 1\label{eq: DiamThetan}
\end{align}

Let $\a >1/2$, then the above argument implies that  with $\bar R_n$ defined as in \eqref{def:barTheta_n}, 
$$\Pi(\Theta_n^c|\a,\tau)\leq 1 - \Phi(\bar R_n/2)   \leq e^{-C_\alpha ( u_n/\tau)^{-1/(\alpha-1/2)} }.$$
Also
\begin{equation*}
\begin{split}
 & \Pi(\theta:\,\|\theta\|_1\leq u_n |\a,\tau)\\
  & \quad =  \Pi(\theta:\,\|\theta\|_1\leq u_n/\tau |\a,1)\\
 &\quad \geq \left( 1 - P\left( \sum_{j=J_1+1}^\infty j^{-\alpha-1/2}|Z_j|> u_n/(2\tau)\right)\right) \prod_{j\leq J_1}P\left( j^{-\alpha-1/2}|Z_j|\leq \frac{u_n}{2\tau J_1}\right) 
 \end{split}
\end{equation*}
we choose $J_1$ such that the first factor is bounded from below by $1/2$. To do so we bound from above 
\begin{equation*}
\begin{split}
P\Big( \sum_{j=J_1+1}^\infty j^{-\alpha-1/2}|Z_j|&> u_n/(2\tau)\Big)  \leq e^{-s u_n/(2\tau)} \prod_{j=J_1+1}^\infty E\left( e^{sj^{-\alpha-1/2}|Z_j|}\right)  \\
& = 
e^{-s u_n/(2\tau)}  e^{\frac{s^2 }{ 2} \sum_{j\geq J_1+1} j^{-2\alpha-1} } \prod_{j\geq J_1+1} (2\Phi( sj^{-\alpha-1/2} ))\\
& \leq e^{-s u_n/(2\tau)} e^{ \frac{ s^2J_1^{-2\alpha }}{ 4 \alpha}} e^{  2s\varphi(0)\sum_{j \geq J_1+1} j^{-\alpha-1/2} }  \\
& \leq \exp\left( - s \left( \frac{ u_n}{ 2 \tau} - \frac{ 2 \varphi(0)J_1^{-\alpha+1/2} }{ \alpha-1/2} \right) + \frac{ s^2J_1^{-2\alpha }}{ 4 \alpha}\right)
\end{split}
\end{equation*}
We choose 
$ J_1^{-\alpha+1/2}  \leq (\alpha-1/2)\frac{ u_n}{ 8\varphi(0) \tau}$ so that for all $s>0$
\begin{equation*}
\begin{split}
P\left( \sum_{j=J_1+1}^\infty j^{-\alpha-1/2}|Z_j|> u_n/(2\tau)\right) &\leq \exp\left( - s  \frac{ u_n}{ 4 \tau} + \frac{ s^2J_1^{-2\alpha }}{ 4 \alpha}\right) \\
&\leq \exp\left( -   \frac{ \alpha J_1^{2\alpha } u_n^2}{ 8 \tau^2}\right), 
\end{split}
\end{equation*}
where the last inequality comes from choosing $s = (u_n/\tau )\alpha J_1^{2\alpha}$. This probability is smaller than $1/2$ as soon as 
$ \alpha J_1^{2\alpha } u_n^2/(8 \tau^2) \geq \log 2$, i.e. as soon as $J_1 \geq (8\log 2/\alpha)^{1/(2\alpha)} (u_n/\tau)^{-1/\alpha}$. 
Since $\alpha>1/2$, there exists a constant $J_0$ such that both constraints are satisfied as soon as $J_1 \geq J_0((\alpha-1/2)/8\varphi(0))^{-1/(\alpha-1/2)}(u_n/\tau)^{-1/(\alpha-1/2)}$. 
We can then bound from below 
\begin{equation*}
 \prod_{j\leq J_1}P\left( j^{-\alpha-1/2}|Z_j|\leq \frac{u_n}{2\tau J_1}\right)  =  \prod_{j\leq J_1}\left( 2 \Phi\left(\frac{j^{\alpha+1/2} u_n}{2\tau J_1}\right)-1\right).
\end{equation*}
For all $j^{\alpha+1/2} u_n\leq 2\tau J_1$ 
$$\left( 2 \Phi\left(\frac{j^{\alpha+1/2} u_n}{2\tau J_1}\right)-1\right) \geq 2 \frac{j^{\alpha+1/2} u_n \varphi(1) }{2\tau J_1}$$
for all $j^{\alpha+1/2} u_n>  2\tau J_1$ 
$$\left( 2 \Phi\left(\frac{j^{\alpha+1/2} u_n}{2\tau J_1}\right)-1\right) \geq 2 \Phi(1) - 1 = C$$
which leads to 
\begin{equation*}
 \prod_{j\leq J_1}P\left( j^{-\alpha-1/2}|Z_j|\leq \frac{u_n}{2\tau J_1}\right)  \geq \exp\left( - J_1 c\right)  ,
\end{equation*}
for some $c$ independent of $\alpha $ and $\tau$. 

\end{proof}

Next we show that for the scaling prior (T2) in the case $\a+1/2\leq \beta$ the second part of condition (2.12) 
holds.

\begin{Lemma}\label{Lem: Cond_LB}
For the prior (T2), when  $\a+1/2\leq \beta$, then $\eps_n(\tau)\geq n^{-\frac{2\a+1}{4\a+4}}$ for all $\tau\in\Lambda_n$ and for all $\theta_0\in \mathcal{H}_{\infty}(\beta,L)\cup \mathcal{S}_{\beta}(L)$, $\theta_0\neq0$ and $\delta_n=o(M_n^{-2})$ we have
\begin{align*}
\sup_{\tau\in \Lambda_0} \frac{ n\epsilon_{n}(\tau)^2} { -\log \Pi(\{\|\theta - \theta_0\|_2\leq \delta_n \epsilon_n(\tau)\}|\tau,\a) }=o(1).
\end{align*}
\end{Lemma}

\begin{proof}
First of all note that $\eps_n(\tau)\geq n^{-(2\a+1)/(4\a+4)}$ follows automatically from (A.2). Then for every $\tau\in\Lambda_0$
\begin{align*}
-\log \Pi(\|\theta-\theta_0\|_2\leq \delta_n \eps_n(\tau)|\tau,\a)
&\geq -\log \Pi(\|\theta\|_2\leq \delta_n \eps_n(\tau)|\tau,\a)\\
&\gtrsim \delta_n^{-1/\a}(\eps_n(\tau)/\t)^{-1/\a}\\
&\geq (M_n^2\delta_n)^{-1/\a} \eps_{n,0}^{-1/\a}n^{-\frac{1}{\a(4+4\a)}},
\end{align*}
hence for $\delta_n=o(M_n^{-2})$ following from Lemma 3.3 the right hand side of the preceding display is of higher order than $n\eps_{n,0}^2$.

\end{proof}


\section{Some technical Lemmas for the hyper-prior distributions}\label{sec: hyper}
In this section we collect the proofs of the technical lemmas on the hyper-prior distribution.

\subsection{Proof of Lemma  3.4} 

Take any $k_0$ satisfying $\eps_n(k_0)\leq 2\eps_{n,0}$.  Since from the proof of Lemma 3.1 
 $n\eps_n(k)^2>k\log\sqrt{k}(1+o(1))$ holds we get 
$$\tilde\pi(k_0)\gtrsim k_0^{-c_2k_0}\geq e^{-2c_0n\eps_n^2(k_0)}\geq e^{-\tilde{w}_n^2n\eps_{n,0}^2}.$$

For the upper bound we note that following from Lemma 3.1

\begin{align*}
\sum_{k=\eps(n/\log n)}^{\infty}\tilde\pi(k)&\lesssim \sum_{k=\eps(n/\log n)}^{\infty} e^{-c_1 k^{\frac{1}{1+2\beta_0}}}
\lesssim e^{-c_3 (n/\log n)^{\frac{1}{1+2\beta_0}}}\\
&\lesssim e^{-\tilde{w}_n^2n^{\frac{1}{1+2\beta_1}}} \leq e^{-\tilde{w}_n^2n\eps_{n,0}^2}.
\end{align*}

\subsection{Proof of Lemma 3.5} 

As a first step choose an arbitrary $\t_0\in \Lambda_0(\tilde{w}_n)$  with $\eps_n(\t_0)\leq 2\eps_{n,0}$. Then any $\t$ satisfying $\eps_n(\t)\leq 2\eps_n(\t_0)$ belongs to the set $\Lambda_0(\tilde{w}_n)$. Next consider $\t$ satisfying $\eps_n(\t)\geq 2\eps_n(\t_0)$.  Furthermore, note that for any $\t_1,\t_2>0$ the RKHSs corresponding to the priors $\Pi(\cdot|\t_1),\Pi(\cdot|\t_2)$ are the same, i.e. $\mathbb{H}^{\t_1}=\mathbb{H}^{\t_2}$.  Following from the definition of the concentration inequality (3.2) 
\begin{align*}
-\log &\Pi(\|\theta-\theta_0\|_2\leq K\eps_n(\t)|\t)\\
&\leq \inf_{h\in\mathbb{H}^{\t}:\, \|h-\theta_0\|_2\leq K\eps_n(\t)/2}\|h\|_{\mathbb{H}^\tau}^2-\log \Pi(\|\theta\|_2\leq (K/2)\eps_n(\t)|\t)\\
&\leq \inf_{h\in\mathbb{H}^{\t_0}:\, \|h-\theta_0\|_2\leq K\eps_n(\t_0)} (\t_0/\t)^2\|h\|^2_{\mathbb{H}^{\tau_0}}-\log \Pi(\|\theta\|_2\leq (\t_0/\t)K\eps_n(\t_0)|\t_0)\\
&\lesssim \max\Big\{\big(\frac{\t_0}{\t})^2,\big(\frac{\t_0}{\t}\big)^{-\a}\Big\} \Big(\inf_{\substack{h\in\mathbb{H}^{\t_0}\\ \|h-\theta_0\|_2\leq K\eps_n(\t_0)}} \|h\|^2_{\mathbb{H}^{\tau_0}}-\log \Pi(\|\theta\|_2\leq K\eps_n(\t_0)|\t_0)\Big)\\
&\leq -\max\{(\t_0/\t)^2,(\t_0/\t)^{-\a}\}\log \Pi(\|\theta-\theta_0\|_2\leq K\eps_n(\t_0)|\t_0)\\
&\leq \max\{(\t_0/\t)^2,(\t_0/\t)^{-\a}\}n\eps_n(\t_0)^2,
\end{align*}
Hence $\eps_n(\t)^2\leq \max\{(\t_0/\t)^{2},(\t_0/\t)^{-\a}\} \eps_n(\t_0)^2$, so one can conclude that $[(2/\tilde{w}_n)\t_0,(\tilde{w}_n/2)^{2/\a}\t_0]\cap \Lambda_n \subset \Lambda_0(\tilde{w}_n)$. Therefore following from the proof of Lemma 3.2 we have $n\eps_n(\t)^2\geq\tau^{2/(1+2\a)}\vee\t^{-2}$ and that $[\t_0/2,2\t_0]\in\Lambda_n$, hence
$$\int_{\t_0/2}^{2\t_0}\tilde\pi(\t)d\t\gtrsim e^{-c_1\t_0^{2/(1+2\a)}}\wedge e^{-c_3\t_0^{-2}} \geq e^{-\tilde{w}_n^2n\eps_{n,0}^2}.$$ 
concluding the first part of (H1).

The second assumption in (H1) holds trivially for $\tilde\pi$ satisfying the upper bounds in the lemma:
\begin{align*}
\int_{0}^{e^{-c_0\bar{c}_0\tilde{w}_n^2n\eps_{n,0}^2}}\tilde\pi(\t)d\t\lesssim e^{-\bar{c}_0\tilde{w}_n^2n\eps_{n,0}^2},\quad
\int_{e^{c_0\bar{c}_0\tilde{w}_n^2n\eps_{n,0}^2}}^{\infty}\tilde\pi(\t)d\t\lesssim e^{-\bar{c}_0\tilde{w}_n^2n\eps_{n,0}^2}.
\end{align*}

\subsection{Proof of Remark 3.3}
One can easily see that 
\begin{align*}
N_n(\Lambda_n)\leq 2 e^{c_0\bar{c}_0\tilde{w}_n^2n\eps_{n,0}^2}/e^{-2c_0\bar{c}_0\tilde{w}_n^2n\eps_{n,0}^2}
\lesssim  e^{3c_0\bar{c}_0\tilde{w}_n^2n\eps_{n,0}^2}.
\end{align*}
Then by noting that $\tilde{w}_n=o(w_n)$ we get our statement. The upper bound on the hyper-entropy of the set $\Lambda_0$ follows immediately from the proof of Proposition 3.2.


\subsection{Proof of Lemma 3.6} 

Similarly to the proof of Lemma 3.6 
  we choose an arbitrary $\a_0\in \Lambda_0(\tilde{w}_n)$  with $\eps_n(\a_0)\leq C\eps_{n,0}$. From Lemma 3.2 
  we have that $\a_0\geq \beta+o(1)$ in case $\theta_0\in \mathcal{S}^{\beta}(L)\cup \mathcal{H}(\beta,L)$, since $n^{-\beta/(1+2\beta)}\gtrsim\eps_n(\a_0)\gtrsim n^{-\a_0/(1+2\a_0)}$.

First assume that $\a_0\leq\log n$, then for any $\a\in[\a_0/(1+2\log\tilde{w}_n/\log n), \a_0]$ we have that
\begin{align}
-\log &\Pi(\|\theta-\theta_0\|_2\leq K\eps_n(\a)|\a)\nonumber\\
&\lesssim \inf_{h\in\mathbb{H}^{\a_0}:\, \|h-\theta_0\|\leq K\eps_n(\a_0)} \|h\|^2_{\mathbb{H}^{\a_0}}+ \{K\eps_n(\a_0)\}^{-1/\a}\nonumber\\
&\lesssim \Big(\inf_{h\in\mathbb{H}^{\a_0}:\, \|h-\theta_0\|\leq K\eps_n(\a_0)} \|h\|^2_{\mathbb{H}^{\a_0}}-\log \Pi(\|\theta\|_2\leq K\eps_n(\a_0)|\a_0)\Big)^{\a_0/\a}\nonumber\\
&\leq \{n\eps_n(\a_0)^2\}^{\a_0/\a}\leq \tilde{w}_n^2 n\eps_n(\a_0)^2,\label{eq: help_hyp_alpha}
\end{align}
hence $[\a_0/(1+2\log\tilde{w}_n/\log n), \a_0]\subset\Lambda_0$. Then the first part of condition (H1) holds for $\tilde\pi$ satisfying the lower bound in the statement, since
\begin{align}
\int_{\a_0/(1+2\log\tilde{w}_n/\log n)}^{\a_0}\tilde\pi(\a)d\a\gtrsim \frac{\a_0\log\tilde{w}_n}{(\log n)e^{c_2\a_0}}\gtrsim e^{-2c_2\log n}\geq e^{-\tilde{w}_n^2 n\eps_{n,0}^2},\label{eq: help_hyp_alpha2}
\end{align}
where in the last inequality we used that by definition $\eps_{n,0}^2\geq n^{-1}\log n$.
In case $(\log n)/2\leq\a\leq\log n\leq \a_0$ we have that  $\eps_n(\a_0)^{-1/\a}\lesssim n^{1/(2\a)}\lesssim 1$ hence similarly to $\eqref{eq: help_hyp_alpha}$ we have that $\a\in\Lambda_0$ so the statement follows from $\eqref{eq: help_hyp_alpha2}$ with $\a_0$ replaced by $\log n$.

Finally we show that the second assumption in (H1) also holds if the upper bound on $\tilde\pi$ is satisfied
\begin{align*}
\int_{c_0 n^{c_1}}^{\infty}\tilde\pi(\a)d\a\lesssim e^{-c_0n}\lesssim e^{-\tilde{w}_n^2n\eps_{n,0}^2},
\end{align*}
if $\tilde{w}_n\leq c_0^{1/2}/\eps_{n,0}$.


\section{Some technical Lemmas for the nonparametric regression model}\label{sec: Lemmas_regression}

\begin{Lemma}\label{Lem: LikeRatioUB_reg}
For $\theta_0\in\ell_2(M)$ (or equivalently $f_0\in L_2(M)$) we have for $u_n\leq n^{-3/2}/\log (n)$
\begin{align}
\sup_{\a,\tau}\sup_{\theta\in \mathbb{R}^n, \|\theta-\theta_0\|_2\leq \eps_n(\a,\tau)}Q_{\alpha,\tau,n}^{\theta}(\mathcal{X}_n)=O(1).
\end{align}
\end{Lemma}

\begin{proof}
Let us denote by $\lambda$ the hyper-parameters $\a$ or $\tau$, and by $\rho(\lambda,\lambda')$ the losses $|\a-\a'|$ or $|\log\t'-\log\t|$. Furthermore, we introduce the notations
\begin{align*}
\underline\psi_{\lambda,i}^{\theta}= \inf_{\rho(\lambda,\lambda')\leq u_n}\sum_{j=1}^{n}\psi_{\lambda,\lambda'}(\theta_j)e_j(t_i);\quad
\overline\psi_{\lambda,i}^{\theta}= \sup_{\rho(\lambda,\lambda')\leq u_n}\sum_{j=1}^{n}\psi_{\lambda,\lambda'}(\theta_j)e_j(t_i).
\end{align*}
For instance in case of $\lambda=\a$ this is
\begin{align*}
\underline\psi_{\a,i}^{\theta}= \sum_{j=1}^{n}j^{-\sign(\theta_{j}e_j(t_i)) u_n}\theta_{j}e_j(t_i);\quad
\overline\psi_{\a,i}^{\theta}= \sum_{j=1}^{n}j^{\sign(\theta_{j}e_j(t_i)) u_n}\theta_{j}e_j(t_i),
\end{align*}
while for $\lambda=\tau$
\begin{align*}
\underline\psi_{\tau,i}^{\theta}= e^{-\sign(\sum_{j=1}^{n} \theta_{j}e_j(t_i)) u_n}\sum_{j=1}^{n}\theta_{j}e_j(t_i);\,
\overline\psi_{\tau,i}^{\theta}= e^{\sign(\sum_{j=1}^{n} \theta_{j}e_j(t_i)) u_n}\sum_{j=1}^{n}\theta_{j}e_j(t_i).
\end{align*}
Then one can easily obtain (using Cauchy-Schwarz inequality) that both in the case of $\lambda=\a$ and $\lambda=\tau$ we have
\begin{align}
|\underline\psi_{\lambda,i}^{\theta}-\overline\psi_{\lambda,i}^{\theta}|&\leq (n^{u_n}-n^{-u_n})\sum_{j=1}^{n}\theta_{j}e_j(t_i)
\leq 2u_nn^{u_n}\log n \sum_{j=1}^{n}|\theta_{j}e_j(t_i)| \nonumber\\
&\leq 2M\|\theta\|_2 u_nn^{u_n+1/2}\log n.\label{eq: RegHelp01}
\end{align}

Writing out the definition of $Q^{\theta}_{\lambda,n}$:
\begin{align}
Q^{\theta}_{\lambda,n}=\int_{\mathbb{R}^n} \sup_{\rho(\lambda,\lambda')\leq u_n}\prod_{i=1}^{n}\frac{1}{\sqrt{2\pi\sigma^2}}e^{-\frac{(x_i-\sum_{j=1}^{n}\psi_{\lambda,\lambda'}(\theta_j) e_j(t_i))^2}{2\sigma^2}}d \mathbf x_n.\label{eq: help4_1}
\end{align}
We deal with the one dimensional integrals separately
\begin{align*}
&\int_{\mathbb{R}}\sup_{\rho(\lambda,\lambda')\leq u_n}\frac{1}{\sqrt{2\pi\sigma^2}}e^{-\frac{\{x_i-\sum_{j=1}^{n}\psi_{\lambda,\lambda'}(\theta_j)e_j(t_i)\}^2}{2\sigma^2}}d x_i\\
&\quad \leq \int_{x_i<\underline\psi_{\lambda,i}^{\theta}\cup x_i> \overline\psi_{\lambda,i}^{\theta} }\frac{1}{\sqrt{2\pi\sigma^2}}e^{-\frac{(x_i-\underline\psi_{\lambda,i}^{\theta})^2}{2\sigma^2}}d x_i
+ \int_{\underline\psi_{\lambda,i}^{\theta}}^{\overline\psi_{\lambda,i}^{\theta}} \frac{1}{\sqrt{2\pi\sigma^2}}d x_i\\
&\quad \leq \int_{\mathbb{R}}\frac{1}{\sqrt{2\pi\sigma^2}} e^{-z^2/(2\sigma^2)}d z
+ \frac{1}{\sqrt{2\pi\sigma^2}} |\overline\psi_{\lambda,i}^{\theta}- \underline\psi_{\lambda,i}^{\theta}| \\
&\quad \leq 1+2M\|\theta\|_2 u_n\sigma^{-1}n^{u_n+1/2}\log n,  
\end{align*}
where the last inequality follows from $\eqref{eq: RegHelp01}$. Note that for $\theta_0\in \ell_2(M)$ and $\theta\in\mathbb{R}^n$  satisfying $\|\theta-\theta_0\|_2\leq \eps_n$ we have 
\begin{align*}
\|\theta\|_2\leq \|\theta_0\|_2+\eps_n\leq 2M.
\end{align*}

Therefore the right hand side of $\eqref{eq: help4_1}$ is bounded from above by 
\begin{align}
\Big(1+2M\|\theta\|_2 u_n\sigma^{-1}n^{u_n+1/2}\log n\Big)^n\leq e^{2M\|\theta\|_2 u_n\sigma^{-1}n^{u_n+3/2}\log n }=O(1). \label{eq: help03_reg}
\end{align}
\end{proof}

\begin{Lemma}\label{Lem: ComplementQ}
Consider priors (T2) and (T3). For $u_n\lesssim n^{-3}/\log n$  if $\lambda=\a$, and $u_n\lesssim \tau_n^{-2}\wedge n^{-5/2} $ if $\lambda=\tau<\tau_n$ we have that
\begin{align*}
\int_{\Theta_n^c}Q_{\tau,\a,n}^{\theta}(\mathcal{X}_n)\Pi(d\theta|\tau,\a)\leq 3e^{- (\eta/2) n\eps_n^2},
\end{align*}
where $\eps_n=\eps(\tau)$ or $\eps_n(\alpha)$ and $\eta\geq\tilde{c}_2 (12K/c)^{1/\a}$.
\end{Lemma}

\begin{proof}
Following the proof of Lemma \ref{Lem: LikeRatioUB_reg} and using Cauchy-Schwarz inequality we get that

\begin{equation*}
\begin{split}
\int_{\Theta_n^c}Q_{\tau,\a,n}^{\theta}(\mathcal{X}_n)&\Pi(d\theta|\tau,\a)
\leq \int_{\Theta_n^c} e^{2Mn^{1+u_n}(\log n)u_n\sum_{i=1}^n|\theta_i|} d\Pi(\theta|\tau,\a) \\
&\leq \Pi(\Theta_n^c|\tau,\a)^{1/2}\sqrt{\int_{\Theta} e^{2M\sum_{i=1}^n|\theta_i|n^{1+u_n}(\log n)u_n} d\Pi(\theta|\tau,\a)}\\
&\leq  e^{-(\eta/2) n\eps_n^2} 2 e^{M^2 n^{2+2u_n}\t^{2}(\log n)^2u_n^2\sum_{i=1}^{n}i^{-1-2\a}}\leq 3e^{-(\eta/2) n\eps_n^2}.
\end{split}
\end{equation*}
\end{proof}

\begin{Lemma}\label{Lem: Test_reg}
In the nonparametric regression model for $\t\leq\t_n$ (for some $\t_n\rightarrow\infty$), $u_n\lesssim \t_n^{-2} n^{-5/2}/\log n$, and $\theta\in \Theta_n(\a,\tau)$ (defined in $\eqref{def: Theta_n}$), there exist tests $\phi_n(\theta)$ such that
\begin{align}
E_{\theta_0}^n\phi_n(\theta)\leq e^{-\frac{1}{2}n\|\theta-\theta_0\|_2^2},\nonumber\\ 
\sup_{\substack{\a>0\\ \tau\leq\t_n}}\sup_{\substack{\theta'\in\Theta_n(\a,\tau)\\ \|\theta-\theta'\|_2<\|\theta-\theta_0\|_2/18}} \int_{\mathcal{X}_n}(1-\phi_n(\theta))dQ_{\a,\tau,n}^{\theta'}(x^{(n)})\leq e^{-\frac{1}{2} n\|\theta-\theta_0\|_2^2}.\label{eq: test_reg}
\end{align}
\end{Lemma}

\begin{proof}
First of all note that the likelihood ratio test (using the sequence notation)
\begin{align}
\phi_n(\theta)= \1\Big[ \sum_{i=1}^{n} \big\{x_i-\sum_{j=1}^{n}\theta_je_j(t_i)\big\}^2-
\sum_{i=1}^{n} \big\{x_i-\sum_{j=1}^{n}\theta_{0,j}e_j(t_i)\big\}^2 \Big].
\end{align}
satisfies for all $\theta\in\mathbb{R}^{n}$
\begin{align*}
\sup_{\theta'\in \mathbb{R}^{n}:\, \|\theta-\theta'\|_2\leq \|\theta-\theta_0\|_2/18} E_{\theta'}^n\Big(1-\phi_n(\theta)\Big)\leq \exp\{-\frac{n}{2}\|\theta-\theta_0\|_2^2\}
\end{align*}
and $E_{\theta_0}^n\phi_n(\theta)\leq \exp\{-\frac{n}{2}\|\theta-\theta_0\|_2^2\}$, see for instance Section 7.7 of \cite{ghosal:vdv:07}.

For notational convenience let us again denote by $\lambda$ both of the hyper-parameters $\a$ and $\t$, then we have that
\begin{align}\label{eq: helpTest_reg}
&\int_{\mathbb{R}^{n}}\Big(1-\phi_n(\theta) \Big)\sup_{\rho(\lambda,\lambda')\leq u_n} \frac{1}{(2\pi\sigma^2)^{n/2}}e^{-\frac{\sum_{i=1}^{n}\{\sum_{j=1}^{n}\psi_{\lambda,\lambda'}(\theta_j')e_j(t_i)-x_i)\}^2}{2\sigma^2}} d\mathbf x_n\\
&\leq \int_{\|\mathbf x_n\|_2\leq\t_n n}\Big(1-\phi_n(\theta)\Big) \frac{1}{(2\pi\sigma^2)^{n/2}}e^{-\frac{\sum_{i=1}^{n}\{\sum_{j=1}^{n}\theta_j' e_j(t_i)-x_i)\}^2}{2\sigma^2}}\nonumber\\
&\quad\times\sup_{\rho(\lambda,\lambda')\leq u_n} \Big(e^{\frac{1}{2\sigma^2}\sum_{i=1}^{n}\big[\{\sum_{j=1}^{n}\theta_j'e_j(t_i)-x_i\}^2-\{\sum_{j=1}^{n}\psi_{\lambda,\lambda'}(\theta_j')e_j(t_i)-x_i\}^2\big]} \Big)d\mathbf x_n\nonumber\\
&\quad+ \int_{\|\mathbf x_n\|_2>\t_n n}\sup_{\rho(\lambda,\lambda')\leq u_n} \frac{1}{(2\pi\sigma^2)^{n/2}}e^{-\frac{1}{2\sigma^2}\sum_{i=1}^{n}\{\sum_{j=1}^{n}\psi_{\lambda,\lambda'}(\theta_j')e_j(t_i)-x_i\}^2} d\mathbf x_n.\nonumber
\end{align}
We deal with the two terms on the right hand side separately.

First we examine the first term, where it is enough to show that the multiplicative term (with the $\sup$) is bounded from above by a constant. Using Cauchy-Schwarz and triangle inequalities and the assumption $|e_j(t_i)|\leq M$ we get that

\begin{align*}
&\sup_{\rho(\lambda,\lambda')\leq u_n}\Big|\sum_{i=1}^{n}\Big[\big\{\sum_{j=1}^{n}\theta_j'e_j(t_i)-x_i\big\}^2-\big\{\sum_{j=1}^{n}\psi_{\lambda,\lambda'}(\theta_j') e_j(t_i)-x_i\big\}^2\Big] \Big|\nonumber\\
&\leq  M\sqrt{\sum_{i=1}^{n}(\overline\psi_{\lambda,i}^{\theta'}-\underline\psi_{\lambda,i}^{\theta'})^2}
\Big(2\|\mathbf x_n\|_2+M\sqrt{n}\|\theta'\|_2+M\sqrt{n}\sup_{\rho(\lambda,\lambda')\leq u_n}\|\psi_{\lambda,\lambda'}(\theta')\|_2\Big) \nonumber
\end{align*}
The right hand side of the preceding display following from $\eqref{eq: RegHelp01}$ and $\|\mathbf x_n\|_2\leq n$ is bounded from above by
\begin{align}
2M^2\|\theta'\|_2 u_nn^{u_n+1}(\log n) (2n\t_n+C\sqrt{n}\|\theta'\|_2)\lesssim u_n \t_n^2 n^{5/2}\log n=O(1).\label{{eq: regTesthelp1}}
\end{align}

Therefore it remained to deal with the second term on the right hand side of $\eqref{eq: helpTest_reg}$. Since $\|\theta'\|_2\lesssim \sqrt{n}\t_n\eps_n$ (following from $\theta'\in\Theta_n(\lambda)$) we have that $\sup_{\rho(\lambda,\lambda')\leq u_n}\|\psi_{\lambda,\lambda'}(\theta')\|\leq n^{u_n}\|\theta'\|_2\lesssim n^{u_n+1/2}\t_n\eps_n=o(n\t_n)$. Therefore 
\begin{align}
\int_{\|\mathbf x_n\|_2\geq \t_n n}&\frac{1}{(2\pi\sigma^2)^{n/2}}e^{-\frac{\sum_{i=1}^{n}\big(\sum_{j=1}^{n}\psi_{\lambda,\lambda'}(\theta_j' )e_j(t_i)-x_i\big)^2}{2\sigma^2}}d\mathbf x_n\nonumber\\
&\leq \int_{\|\mathbf x_n\|_2\geq\t_n n}\frac{1}{(2\pi\sigma^2)^{n/2}}e^{-\frac{\|\mathbf x_n\|^2/2}{2\sigma^2}}d\mathbf x_n\leq 2^n e^{-n^2\t_n^2/(2\sigma^2)},
\end{align}
where the right hand side is of smaller order than $\exp\{-(1/2) n\|\theta-\theta_0\|_2^2\}$, since $\|\theta-\theta_0\|_2\leq\|\theta_0\|_2+\|\theta\|_2\lesssim 1+ \t_n\sqrt{n}\eps_n(\lambda)=o(\t_n\sqrt{n}).$
\end{proof}

\begin{Lemma}\label{lem:loglike}
Consider the nonparametric regression model and priors of type (T1)-(T3). Take any $\t_n\rightarrow\infty$ and $0<\t\leq\t_n$. Then for $u_n\lesssim n^{-2}\t_n^{-1}/\log n$ 
$$\sup_{\|\theta-\theta_0\|\leq K\eps_n(\lambda)}P_{\theta_0}^n
\Big\{\inf_{\rho(\lambda,\lambda')\leq u_n}\ell_n\big(\psi_{\lambda,\lambda'}(\theta)\big)-\ell_n(\theta_0)\leq -c_3n\eps_n(\lambda)^2 \Big\}=e^{-n\eps_n(\lambda)^2}, $$
for $c_3\geq 2+3\sigma^{-2}K^2/2$.
\end{Lemma}

\begin{proof}
By triangle inequality we have that
\begin{align}
|\ell_n\big(\psi_{\lambda,\lambda'}(\theta)\big)-\ell_n(\theta_0)|\leq |\ell_n\big(\psi_{\lambda,\lambda'}(\theta)\big)-\ell_n(\theta)|+|\ell_n(\theta)\big)-\ell_n(\theta_0)|.\label{eq: help101}
\end{align}
We deal with the two terms on the right hand side separately.

First consider the first term on the right hand side of $\eqref{eq: help101}$ and note that in case of prior (T1) it is zero. For priors (T2)-(T3) following from Lemma \ref{Lem: Test_reg} we have that for $\|\mathbf x_n\|_2\leq n\t_n$ it is bounded above by a constant, while $P_{\theta_0}^n (\|\mathbf x_n\|_2\geq n\t_n)\leq e^{-cn^2\t_n^2}$.

For the second term on the right hand side of $\eqref{eq: help101}$ we apply Chernoff's inequality and $K(\theta,\theta_0)=\sigma^{-2}\|\theta-\theta_0\|_2^2$
\begin{align*}
&\sup_{\theta\in \|\theta-\theta_0\|_2\leq K\eps_n}P_{\theta_0}^n\big\{\ell_n(\theta)-\ell_n(\theta_0)\leq -(3\sigma^{-2}K^2/2+1)n\eps_n^2\big\}\\
&\,\leq\sup_{\theta\in \|\theta-\theta_0\|_2\leq K\eps_n}P_{\theta_0}^n\Big[\ell_n(\theta)-\ell_n(\theta_0)- E_{\theta_0}^n\big\{\ell_n(\theta)-\ell_n(\theta_0)\big\}\leq -\frac{K^2+2\sigma^{2}}{2\sigma^{2}} n\eps_n^2\Big]\\
&\,\leq \sup_{\theta\in \|\theta-\theta_0\|_2\leq K\eps_n} e^{-(\sigma^{-2}K^2/2+1)n\eps_n^2} E(e^{\sigma^{-2}\langle \sum_{j=1}^{n}(\theta_{0,j}-\theta_j)e_j,Z \rangle_n})\\
&\,\leq \sup_{\theta\in \|\theta-\theta_0\|_2\leq K\eps_n} e^{-(\sigma^{-2}K^2/2+1)n\eps_n^2}e^{\sigma^{-2}(n/2)\|\theta_0-\theta\|_2^2 }\leq e^{-n\eps_n^2},
\end{align*}
where $Z$ denotes an $n$ dimensional vector of iid standard normal random variables.
\end{proof}


\section{ Some Technical Lemmas in the density case} \label{app:dens}
\begin{Lemma}\label{lem:dist}
Let $f_0 = f_{\theta_0}$, $f_\theta$ with $\| \theta - \theta_0\|_1 < +\infty $, and $\theta, \theta_0 \in \ell_2$, then 
\begin{equation}
\begin{split}
K(f_0, f_\theta) &\leq  \|f_0\|_\infty  e^{\|\theta- \theta_0\|_1 \|\phi\|_\infty } \|\theta - \theta_0\|_2^2,\\
V_2(f_0, f_\theta) &\leq \|f_0\|_\infty \|\theta - \theta_0\|_2^2,\\
d^2(f_0, f_\theta) & \gtrsim \exp( - c_1 \|\theta  -\theta_0\|_1 )\|\theta -\theta_0\|_2^2.
\end{split}
\end{equation}
\end{Lemma}

\begin{proof}

We have following \cite{rivoirard:rousseau:12}, see also  the supplement, Section 3.2, Proof of Proposition 5 of  \cite{castillo:rousseau:suppl} if $ \| \log f_\theta \|_\infty \leq M$
 \begin{equation}
 \begin{split}
 K(f_0, f_\theta) &= \langle \theta_0 - \theta, \phi(f_0) \rangle_2 - c(\theta_0 ) + c(\theta),
 \end{split}
 \end{equation}
with 
 \begin{equation*}
 \begin{split}
c(\theta) - c(\theta_0) &= \log\Big(\int f_0(x) e^{\sum_j (\theta_j - \theta_{0,j})\phi_j(x) }dx\Big) \\
& \leq  1 - \langle \theta_0 - \theta, \phi(f_0) \rangle +\frac{ \|f_0\|_\infty}{2} e^{\|\phi\|_\infty \|\theta-\theta_0\|_1  }\| \theta-\theta_0\|_2^2.
 \end{split}
\end{equation*}
Since $\langle \theta_0 - \theta, \phi(f_0) \rangle^2_2 \leq \|f_0\|_\infty^2 \|\theta-\theta_0\|_2^2$, this leads to 
 \begin{equation*}
 K(f_0, f_\theta) \leq  \|f_0\|_\infty e^{\|\phi\|_\infty \|\theta-\theta_0\|_1  }\| \theta-\theta_0\|_2^2.
 \end{equation*}
Similarly
 \begin{equation*}
 \begin{split}
 V_{2} (f_0, f_\theta) &\leq 2 \left( E_{\theta_0} \left[ \left( \sum_j (\theta_j  - \theta_{0,j} )\phi_j \right)^{2}  \right] + \langle\theta - \theta_0, \phi(f_0)\rangle^{2} \right)\\
 & \leq  2c_2  \|f_0\|_\infty^2 \|\theta-\theta_0\|_2^2.
\end{split}
 \end{equation*}

Finally using the inequality $|e^{v}-e^w|= e^v|1-e^{w-v}|\geq e^v e^{-|w-v|}|w-v|$ and that $\log f_0>-\infty$ we have that $d(f_0,f_{\theta})$ is bounded from below by
\begin{align*}
\int_0^1 f_0 e^{-| \sum_j (\theta_{0,j}-\theta_j)\phi_j(x)+c(\theta)-c(\theta_0)| }&\big(\sum_{j}(\theta_j-\theta_{0,j})\phi_j(x)+c(\theta)-c(\theta_0) \big)^2 dx\\
&\gtrsim e^{-2\|\phi\|_{\infty} \|\theta-\theta_0\|_1}\|\theta-\theta_0\|_2^2,
\end{align*}
where in the last inequality we applied the orthonormality of the basis $\phi_j$ and the inequality $e^{|c(\theta)-c(\theta_0)|}\leq e^{\|\phi\|_{\infty}\|\theta-\theta_0\|_1}$.
\end{proof}

\begin{Lemma}\label{L1norm}
Let $\Pi( \cdot | \alpha, \tau)$ be the Gaussian prior with $\alpha >1/2$ and {$\tau \in (n^{-a}, n^{b})$} then 
\begin{align} \label{borell}
E \left( \|\theta\|_1| \alpha , \tau\right) = \tau E\left( \|\theta\|_1|\alpha , \tau=1\right) \equiv\tau A_\alpha< \infty , \\
  \Pi \left( \|\theta\|_1 > t +\tau A_\alpha | \alpha,\tau \right) \leq e^{- \frac{ t^2 }{2\sigma_{\alpha, \tau}^2} }\nonumber
\end{align}
with 
$$\sigma_{\alpha, \tau} \leq \sigma_{1/2, 1}\tau, \quad \forall \alpha \geq 1/2, \tau >0.$$
Moreover, for all $K_n > 0$ going to infinity and all $M>0$
\begin{equation} \label{L1:Kn}
\Pi  \left( \|\theta-\theta_0\|_1 > M+ \|\theta_0\|_1+ \sqrt{K_n} \|\theta-\theta_0\|_2 | \alpha,\tau \right) \leq  e^{-\frac{M^2 K_n^{2\alpha} }{2 \tau^2 } }. \end{equation} 
\end{Lemma}
\begin{proof}
The first part of Lemma \ref{L1norm} is Borell's inequality associated to the Banach space $\ell_1 = \{ \theta; \sum_i |\theta_i|<+\infty \}$ since 
$$\Pi\left( \|\theta \|_1 <+\infty | \alpha, \tau\right) = 1, \quad \forall \alpha >1/2, \tau >0.$$
To prove \eqref{L1:Kn}, let $K_n>0$ then 
$$\|\theta- \theta_0\|_1 \leq \sqrt{K_n } \|\theta - \theta_0\|_2 + \sum_{j >K_n} |\theta_{0j}|+  \sum_{j>K_n} |\theta_j|$$
and
\begin{equation*}
\Pi( \sum_{j>K_n}|\theta_j| >M |\alpha,\tau ) = P\left( \sum_{j>K_n} j^{-\alpha -1/2} |Z_j| > M/\tau \right)
 \leq e^{-\frac{M^2 K_n^{  2\alpha } }{2 \tau^{2}  } },
\end{equation*}
where $Z_j \stackrel{iid}{\sim} \mathcal N(0,1)$.
\end{proof}
 \begin{Lemma}\label{lem:KL:dens}
In the density estimation problem with a prior on $f_\theta$ defined by (3.15), 
  if there exists $M>0$ and $\epsilon_n$ such that $n \epsilon_n^2 \rightarrow +\infty $, 
and
$$\{\| \theta - \theta_0\|_2 \leq \epsilon_n \} \cap \{ \|\theta \|_1 \leq M\} \subset B_n( \theta_0, M_2\epsilon_n,2) ,$$
then there exists $a_0>0$ such that for all  $\{\| \theta - \theta_0\|_2 \leq \epsilon_n \} \cap \{ \|\theta \|_1 \leq M\} $,
\begin{equation*}
P_{\theta_0}^n\left( \ell_n(\theta) -\ell_n(\theta_0) \leq -2M_2 n \epsilon_n^2 \right)  \leq  e^{- a_0 n \epsilon_n^2 } 
\end{equation*}
\end{Lemma}

\begin{proof}
Let $\theta \in  \{\| \theta - \theta_0\|_2 \leq \epsilon_n \} \cap \{ \|\theta \|_1 \leq M\} $, 
\begin{align*}
P_{\theta_0}^n\{ \ell_n(\theta)& -\ell_n(\theta_0) \leq -2M_2 n \epsilon_n^2 \}\\
  &\leq e^{ -2s M_2 n \epsilon_n^2 } \left( 1 + K(f_0, f_\theta) + V_2(f_0, f_\theta)e^{2s \|\theta-\theta_0\|_1\|\phi\|_\infty} \right)^n\\
&\leq   e^{- a_0 n \epsilon_n^2 }  
\end{align*}
for some $a_0>0$ proportional to $M_2$.
\end{proof}

\end{document}